\documentclass[11pt]{article}
\usepackage[margin=1in]{geometry}
 
\usepackage[]{amsmath,amssymb,epsfig}
\usepackage{amsthm}
\usepackage{amsmath}
\usepackage{amssymb}
\usepackage{graphicx}
\usepackage{epstopdf}
\usepackage{comment}
\usepackage{array}
\usepackage{algorithm}
\usepackage{url}
\usepackage{booktabs}
\usepackage{lscape}
\usepackage{algpseudocode}
\usepackage{setspace}
\usepackage{multicol}
\usepackage{multirow}
\usepackage{color}
\usepackage{colortbl}
\usepackage{xcolor}
\usepackage{hyperref}
\usepackage{placeins}

\usepackage[%
    font={small,sf},
    labelfont=bf,
    format=hang,    
    format=plain,
    margin=0pt,
    width=0.8\textwidth,
]{caption}
\usepackage[list=true]{subcaption}

\newtheorem{theorem}{Theorem}

\newtheorem{corollary}{Corollary}

\newtheorem{definition}{Definition}
\newtheorem{example}{Example}

\newtheorem{lemma}{Lemma}

\newtheorem{proposition}{Proposition}
\newtheorem{remark}{Remark}

\numberwithin{equation}{section}

\DeclareMathOperator{\Tr}{Tr}

\DeclareMathOperator{\supp}{supp}

\DeclareMathOperator{\vect}{vec}
\DeclareMathOperator{\blkdiag}{blkdiag}
\DeclareMathOperator{\nullsp}{null}

\DeclareMathOperator{\polylog}{polylog}
\newcommand{\calA}{\ensuremath{\mathcal{A}}}
\newcommand{\calB}{\ensuremath{\mathcal{B}}}
\newcommand{\calC}{\ensuremath{\mathcal{C}}}

\newcommand{\calF}{\ensuremath{\mathcal{F}}}

\newcommand{\calP}{\ensuremath{\mathcal{P}}}

\newcommand{\calS}{\ensuremath{\mathcal{S}}}

\newcommand{\calN}{\ensuremath{\mathcal{N}}}
\newcommand{\calT}{\ensuremath{\mathcal{T}}}
\newcommand{\calU}{\ensuremath{\mathcal{U}}}
\newcommand{\calK}{\ensuremath{\mathcal{K}}}

\newcommand{\calE}{\ensuremath{\mathcal{E}}}

\newcommand{\norm}[1]{\left\|{#1}\right\|}
\newcommand{\normnew}[1]{\|{#1}\|}
\newcommand{\abs}[1]{\left|{#1}\right|}

\newcommand{\set}[1]{\left\{{#1}\right\}}
\newcommand{\dotprod}[2]{\langle#1,#2\rangle}
\newcommand{\est}[1]{\widehat{#1}}
\newcommand{\expec}{\ensuremath{\mathbb{E}}}
\newcommand{\matR}{\ensuremath{\mathbb{R}}}

\newcommand{\argmin}[1]{\underset{#1}{\operatorname{argmin}}}

\newcommand{\prob}{\ensuremath{\mathbb{P}}}

\newcommand{\rev}[1]{\textcolor{black}{#1}} 

\newcommand{\Atil}{\widetilde{A}}
\newcommand{\Xtil}{\widetilde{X}}
\newcommand{\Qtil}{\widetilde{Q}}
\newcommand{\Dtil}{\widetilde{D}}
\newcommand{\Vbar}{\overline{V}}
\newcommand{\vbar}{\overline{v}}
\newcommand{\Pitil}{\widetilde{\Pi}}
\newcommand{\xtil}{\widetilde{x}}
\newcommand{\ones}{\ensuremath{\mathbf{1}}} 
\newcommand{\invfac}{\mu}
\newcommand{\etatil}{\widetilde{\eta}}
\newcommand{\unitsph}{\mathbb{S}^{md^2-1}}
 
\newcommand{\Dtilacomp}{(\Dtil a^{*})_{\calS^c}}

\newcommand{\Hbar}{\overline{H}}
\newcommand{\Gbar}{\overline{G}}
\newcommand{\Fbar}{\overline{F}}
\newcommand{\ols}{\text{OLS}}

\newcommand{\calCps}{\mathcal{C}'_{\calS}}
\newcommand{\scale}{r}

\title{Joint learning of a network of linear dynamical systems via total variation penalization\footnote{Authors are listed alphabetically}}

 \author{
 Claire Donnat$^{1}$ \\ \texttt{cdonnat@uchicago.edu}\and
 Olga Klopp$^{2}$\\ \texttt{Olga.KLOPP@math.cnrs.fr} \and
 Hemant Tyagi$^{3}$\\ \texttt{hemant.tyagi@ntu.edu.sg}
 } 

 \date{
 $^1$ University of Chicago, Chicago, IL 60637, USA  \\
 $^2$ESSEC Business School,  95021 Cergy-Pontoise, France \\
 $^3$Division of Mathematical Sciences, 
       SPMS, 
       NTU Singapore 637371
 \newline\newline
 \today
 }

\begin{document}
\maketitle

\begin{abstract}
We consider the problem of joint estimation of the parameters of $m$ linear dynamical systems, given access to single realizations of their respective trajectories, each of length $T$. The linear systems are assumed to reside on the nodes of an undirected and connected graph $G = ([m], \mathcal{E})$, and the system matrices are assumed to either vary smoothly or exhibit small number of ``jumps'' across the edges. We consider a total variation penalized least-squares estimator and derive non-asymptotic bounds on the mean squared error (MSE) which hold with high probability. In particular, the bounds imply for certain choices of well connected $G$ that the MSE goes to zero as $m$ increases, even when $T$ is constant. The theoretical results are supported by extensive experiments on synthetic and real data.
\end{abstract}


\section{Introduction} \label{sec:intro}
The problem of learning a linear dynamical system\footnote{Also referred to as a vector autoregressive (VAR) model in the literature.} (LDS) from samples of its trajectory has many applications, e.g., in control systems, time-series analysis and reinforcement learning. In its simplest form, the state $x_t \in \matR^d$ of the (autonomous) system evolves over $T$ time-steps as
\begin{align*}
    x_{t+1} = A^* x_t + \eta_{t+1}; \quad t=0,\dots,T
\end{align*}
and the goal is to estimate the $d \times d$ matrix $A^*$ from a realization $(x_t)_{t=0}^{T+1}$. Here, $(\eta_{t})_{t=0}^{T+1}$ are unobserved, and typically centered and independent random vectors, often referred to as the process noise or the ``excitation'' of the system. The problem has a rich history in the control systems and statistics communities, with classical works typically consisting of either asymptotic results (e.g., \cite{LaiWei82, LaiWei83, LaiWei1986}) or involve quantities that are exponential in the degree of the system (e.g. \cite{Campi2002, VidyaSagar2008}).

A recent line of work (\cite{FaraUnstable18, Simchowitz18a, Sarkar19, Jedra20}) has focused on finite-time identification of LDSs where the goal is to provided non-asymptotic error guarantees for recovering $A^*$. Here, the ordinary-least-squares (OLS) is analyzed theoretically and error bounds are derived w.r.t the spectral norm $\normnew{\est{A}-A^*}_2$ --  the bounds holding provided $T$ is large enough. A typical assumption made in these results is that $A^*$ is stable, i.e., its spectral radius is either strictly less than one (e.g., \cite{Jedra20}) or even more broadly, is less than or equal to one (e.g., \cite{Simchowitz18a, Sarkar19}). In both these scenarios, the error rate is typically shown to go to zero at the near-parametric rate\footnote{Here, $\tilde{O}(\cdot)$ hides logarithmic factors} $\tilde{O}(1/\sqrt{T})$ provided $T$ is suitably large (typically polynomially large w.r.t $d$). The result in \cite{Jedra20} implies that with probability at least $1-\delta$, 
\begin{equation*}
       \norm{\est{A} - A^*}_2 \lesssim \sqrt{\frac{\log(1/\delta) + d}{T}} \quad \text{ if } \quad T \gtrsim  \Delta(\log(1/\delta) + d), 
\end{equation*}
where $\Delta$ is finite if the spectral radius of $A^*$ is less than one. Moreover, $\Delta$ is a constant if $\normnew{A^*}_2 < 1$. The above bound is also known to be optimal in terms of the dependence on $\delta, d$ and $T$ (\cite{Simchowitz18a}).

\subsection{Learning multiple LDSs via total variation (TV) penalization}
In this paper, we consider the problem of joint identification of $m$ LDSs. More precisely, for a given undirected (and connected) graph $G = ([m],\calE)$, we consider that at node $l \in [m]$, the state $x_{l,t} \in \matR^d$ evolves as
%
\begin{equation} \label{eq:mult_linsys}
    x_{l,t+1} = A^*_{l} x_{l,t} + \eta_{l,t+1}; \quad t=0,\dots,T, \ x_{l,0} = 0,
\end{equation}
where $\eta_{l,t}$ are assumed to be i.i.d centered standard Gaussian vectors. Furthermore, we assume that the matrices $A^*_l$'s either vary smoothly (i.e., $\sum_{\set{l,l'} \in \calE} \norm{A^*_l - A^*_{l'}}_{1,1}$, is\footnote{Here, $\norm{\cdot}_{1,1}$ is the entry-wise $\ell_1$ norm; see Section \ref{subsec:notation} for notation.} small) or exhibit only a small number of jumps across edges (i.e., $(A^*_l - A^*_{l'})
_{l \neq l'}$, has a small number of non-zeroes). 
Given the data $(x_{l,t})_{t=0}^T$ for each $l \in [m]$, we then obtain estimates $\est{A}_l$ of $A^*_l$ for each $l$, via the TV-penalized least squares estimator 
\begin{equation} \label{eq:tv_pen_ls_matform}
    (\est{A}_1,\dots,\est{A}_m) = \argmin{A_1,\dots,A_m \in \matR^{d \times d}} \ \set{\frac{1}{2m}\sum_{{l}=1}^m {\sum_{t=1}^T}\norm{x_{l,{t+1}} - A_l x_{l,t}}_F^2 + \lambda \sum_{\set{l,l'} \in \calE} \norm{A_{l} - A_{l'}}_{1,1}}
\end{equation}
where $\lambda \geq 0$ is a regularization parameter.

Our aim is to establish precise \emph{non-asymptotic} bounds\footnote{All of our results, including the conditions on $T$ and $m$ are non-asymptotic.} on the mean squared error (MSE) which hold with high probability. In particular, we would consequently like to ensure that the MSE goes to zero as $m$ increases, i.e.,
\begin{align} \label{eq:intro_weak_consis}
 \text{(Weak consistency)} \qquad   MSE := \frac{1}{m} {\sum_{l=1}^m}\norm{\est{A}_l - A^*_l}_F^2 \longrightarrow 0 \quad \text{ as } m\rightarrow \infty.
\end{align}
Note that \eqref{eq:intro_weak_consis} is particularly relevant in the ``small-$T$'' regime where $T$ is independent of $m$, and is small. In this case, we cannot hope to reliably estimate the matrices by estimating them via the OLS estimator at each node. More generally, even if $T$ were to grow (mildly) with $m$, it is still meaningful to ask whether the MSE for the estimator \eqref{eq:tv_pen_ls_matform} goes to zero at a faster rate w.r.t $m$ as compared to the naive node-wise OLS estimator.
%
\paragraph{Contributions.} 
At a high level, our main contribution is to show that when the parameters of the linear systems either vary smoothly or exhibit only a small number of changes across edges, then a joint least-squares estimator with a graph-TV penalty can pool information across nodes to learn the parameters far more efficiently than estimating each system in isolation. Concretely, by leveraging the graph structure, our method attains tighter non‑asymptotic error guarantees (and in well‑connected graphs, consistency as the number of systems grows) even when each trajectory is short, thereby outperforming node‑wise OLS in these regimes.
Our main result is to provide  non-asymptotic bounds on the MSE for the estimator \eqref{eq:tv_pen_ls_matform}, which hold with high probability (see Theorem \ref{thm:main_result_tv_pen}). More interpretable versions of Theorem \ref{thm:main_result_tv_pen} are provided as Corollaries \ref{cor:stable-main} and \ref{cor:stable-twoS}, when $\norm{A^*_l}_2 < 1$ holds for each $l$. In this case, transparent dependencies between $T, m$, the smoothness of $(A^*_l)_{l=1}^m$ w.r.t $G$, and the connectivity of $G$ are provided (see conditions \eqref{eq:C1} - \eqref{eq:C3b}). To the best of our knowledge, these are the first theoretical results for graph-based TV-penalized linear regression when the data contains dependencies; see the end of Section \ref{subsec:rel_work} for comparison with related work in this regard. 
\begin{itemize}
    \item In Corollary \ref{cor:smallT-DeltaG}, we address the case where $T$ is a constant, and show that if the graph $G$ is well connected and the matrices $(A^*_l)_{l=1}^m$ either have few jumps or vary smoothly across the edges, then weak consistency is guaranteed as $m$ increases. This is shown for two examples of well connected graphs - the complete graph (Example \ref{ex:complete-DeltaG}) and the Erd\"os-Renyi graph (Example \ref{ex:ER-DeltaG}). 

    \item For certain graphs such as the star graph (Example \ref{ex:star_graph}) and the $2$D grid (Example \ref{ex:2d_grid}), our results require $T$ to grow with $m$. Nevertheless, we show that if $(A^*_l)_{l=1}^m$ are sufficiently smooth w.r.t $G$, and $m$ is suitably large, then $T$ is only required to grow poly-logarithmically with $m, d$ for (weakly-) consistent recovery. In this case, the MSE for \eqref{eq:tv_pen_ls_matform} is $o(1)$, while that for the naive node-wise OLS is $O(1)$ w.r.t $m$ (see Remark \ref{rem:ind_est_comparison}).

    \item  In terms of the proof techniques for Theorem \ref{thm:main_result_tv_pen}, the bulk of the effort lies in establishing the restricted eigenvalue condition for block-diagonal design matrices where we now face difficulties on account of the \emph{dependencies} within each block; see Lemmas \ref{lem:RE_h_1} - \ref{lem:cross-term}. This is achieved via a range--nullspace decomposition tailored to the graph-TV operator. Our main probabilistic tool here is the concentration result of \cite{krahmer14} for controlling the suprema of second order subgaussian chaos processes involving positive semidefinite (p.s.d) matrices. Other parts in the proof of Theorem \ref{thm:main_result_tv_pen} make use of tail bounds for controlling the norm of self-normalized vector-valued martingales, using ideas from the proof of \cite[Theorem 1]{abbasi11}; see Lemmas \ref{lem:l2_bd_term1} and \ref{lem:linf_bd_term2}. Such tail bounds have been deployed recently for learning a single LDS (\cite{Sarkar19}), and also for learning multiple LDS's under a graph-based quadratic variation smoothness assumption (\cite{tyagi2024jointlearn}). However the above ingredients are -- to the best of our knowledge -- new for \emph{graph-based TV-penalized regression with dependent observations}.

    \item Finally, we provide extensive empirical results on synthetic data, and also real datasets related to the U.S Environmental Protection Agency (EPA) national air-quality monitoring network, where we compare the performance of  \eqref{eq:tv_pen_ls_matform} (w.r.t parameter estimation and prediction error) with other baseline methods. For synthetic data, we find that our graph-TV estimator typically outperforms other methods on different graph topologies, especially when the system matrices exhibit a piecewise constant structure over the edges. When the parameters vary smoothly over the edges, the graph-TV estimator is found to be competitive with Laplacian smoothing (\cite{tyagi2024jointlearn}). For the U.S EPA datasets, we find that graph-TV has comparable performance, in terms of prediction error, w.r.t the best competing method.
    
\end{itemize}
%
%
%
\subsection{Related work} \label{subsec:rel_work}
The task of joint estimation of multiple LDSs from their trajectories has received considerable attention over the past few years, mainly due to various applications arising in the modeling of, e.g., flight path dynamics at different altitudes (\cite{Bosworth92}), brain network dynamics (\cite{model_brain20, Gu2014ControllabilityOS}), and gene expressions in genomics (\cite{basu15a}), to name a few.

Recently, \cite{modi_jointlearn24} considered the setting where the system matrices $(A^*_l)_{l=1}^m$ are unknown linear combinations of $k$ unknown basis functions. They propose an estimator for estimating the system matrices along with bounds on the MSE, and show that reliable estimation can be performed provided $T > k$. This is particularly meaningful when $k \ll d^2$.

A recent line of work has also focused on federated learning of LDSs. For instance, \cite{pmlr-v211-wang23d} consider the $l$th system (or client) to be modeled as 
\begin{equation} \label{eq:intro_mult_LDS} 
    x_{l,t+1} = A^*_{l} x_{l,t} + B^*_{l} u_{l,t} +  \eta_{l,t+1}; \quad t=0,\dots,T,
\end{equation}
where $u_{l,t}$ represents the external input to the system. The clients are only allowed to communicate with a central server and $N_l$ independent trajectories are observed for the $l$th client. The goal is to find a common estimate $\est{A}, \est{B}$ to all the system matrices $A^*_l, B^*_l$. It is assumed that 
\begin{equation} \label{eq:intro_fedsys_assump}
    \max_{l,l' \in [m]} \norm{A^*_l - A^*_{l'}}_2 \leq \epsilon, \quad \max_{l,l' \in [m]} \norm{B^*_l - B^*_{l'}}_2 \leq \epsilon.
\end{equation}
Then, the estimates $\est{A}, \est{B}$ are obtained via the OLS method, and error bounds are derived w.r.t the spectral norm, for each client $l \in [m]$. Their main result states that if $N_l$ is suitably large w.r.t $T$ and the system dimensions, and $\epsilon$ is suitably small, then the estimation error at each $l$ is smaller than that obtained by individual client-wise estimation. Clearly, the edge-wise condition \eqref{eq:intro_fedsys_assump} is much stronger than requiring that the TV of the system matrices is suitably small. \cite{Chen2023MultiTaskSI} considered jointly estimating systems of the form \eqref{eq:intro_mult_LDS} under different types of structural assumptions on $A^*_l$'s (e.g., group sparsity; smoothness of the form \eqref{eq:intro_fedsys_assump} w.r.t the Frobenius norm between all pairs of matrices).  Some theoretical results are provided (albeit not for parameter estimation) for a group-lasso type estimator, along with several experiments on synthetic and real data.

A closely related setting to the one above is  where we are given 
the trajectory of the true system, and that of a ``similar'' system. This was studied by 
\cite{Xin2022IdentifyingTD} -- they proposed a weighted least-squares approach for estimating the true system parameters, and show that the estimation error can be reduced by using data from the two systems.

The recent work \cite{tyagi2024jointlearn} considered a similar setting as in the present paper, but under a different smoothness assumption wherein $\sum_{\set{l,l'} \in \calE} \norm{A^*_{l} - A^*_{l'}}_F^2$ is small. For a smoothness penalized estimator (namely Laplacian smoothing), it was shown that weak consistency is obtained when $T$ is at least as large as the condition required for identification of a single LDS, and also grows with $m$ as $\log m$. It was also shown for a subspace constrained estimator that weak consistency is obtained for $T = 2$ provided the smoothness term is sufficiently small. It is by now well known
that TV penalization is able to better preserve sharp edges and boundaries in signals as compared to the aforementioned quadratic  penalization (e.g, \cite[Figure 1]{sadhanala2016total}). Moreover 
the estimators considered in \cite{tyagi2024jointlearn} admit a closed form solution and require a very different analysis than in our setting.

In the special case where $A^*_l = A^*$ for each $l$,  i.e., the perfectly smooth setting, our problem setting reduces to learning a single LDS from $m$ independent trajectories. In this case, there exist many results for estimating the system matrix, with error bounds w.r.t the spectral norm (see e.g., \cite{Xin22MultTraj, zheng2020non}).  

Note that by stacking the vectorized versions of the matrices $A^*_1,\dots,A^*_l$ to form a tall vector $a^* \in \matR^{md^2}$, the model in \eqref{eq:mult_linsys} can be recast as the linear model \eqref{eq:model_vec_form} with a design matrix $Q$. Then, the estimator \eqref{eq:tv_pen_ls_matform} can be equivalently written as a TV-penalized least squares estimator of the form \eqref{eq:tv_pen_ls_vecform}, which is referred to as the generalized lasso in the literature  (\cite{tibshirani2011solution,land1997variable}). Estimators of the form \eqref{eq:tv_pen_ls_vecform} have been theoretically analyzed recently\footnote{We remark that \cite{LiRaskutti20}, \cite{tran2022generalizedelasticnetsquares} consider a more general graph-based penalization involving both the $\ell_1$ and $\ell_2^2$ penalties. This is referred to as the \emph{generalized elastic net} in \cite{tran2022generalizedelasticnetsquares} and the \emph{graph total-variation method} in \cite{LiRaskutti20}.} in \cite{LiRaskutti20, tran2022generalizedelasticnetsquares}. The main difference between our setting and these works is that \cite{LiRaskutti20}, \cite{tran2022generalizedelasticnetsquares} consider the rows of the design matrix to be independent samples from a centered Gaussian (with covariance $\Sigma$) and also assume that the noise is independent of the design matrix. In our case, the design matrix $Q$ is block-diagonal, where the diagonal blocks are independent (as they correspond to different systems) but the entries within a block \emph{are dependent} (as they correspond to data generated by the same system). Secondly, the ``noise'' term in our setting in \eqref{eq:model_vec_form} shares dependencies with $Q$, while the setting in \cite{LiRaskutti20}, \cite{tran2022generalizedelasticnetsquares} assumes that the noise is independent of the design matrix. These two points lead to highly non-trivial challenges in our setting, see discussion after Theorem \ref{thm:main_result_tv_pen}, and also Remark \ref{rem:comp_iid_tv_regression} for further details.

 Notice that \eqref{eq:mult_linsys} can be equivalently written as a single LDS 
    \begin{equation*}  
        \begin{bmatrix}
   x_{1,t+1} \\ 
   x_{2,t+1} \\ 
  \vdots \\
  x_{m,t+1}
 \end{bmatrix} =
\underbrace{\begin{bmatrix}
   A_1^* & 0 & \hdots & 0\\ 
  0 & A_2^* & \hdots & 0\\ 
  \vdots  &  & \ddots & \vdots\\
  0 & \hdots & \hdots & A_m^*
 \end{bmatrix}}_{ =: A^*}
\begin{bmatrix}
   x_{1,t} \\ 
   x_{2,t} \\ 
  \vdots \\
  x_{m,t}
\end{bmatrix}
+ 
\begin{bmatrix}
   \eta_{1,t+1} \\ 
   \eta_{2,t+1} \\ 
  \vdots \\
  \eta_{m,t+1}
 \end{bmatrix}; \ t=0,\dots,T.
\end{equation*}
Hence the recovery problem is equivalent to that of estimating the block-diagonal matrix $A^* \in \matR^{md \times md}$. Here $A^*$ is structured as it is not only block-diagonal (i.e., linear equality constraints), but its diagonal blocks do not change too quickly, or have few jumps over $\calE$. There is by now a substantial literature on estimating a LDS with a structured system matrix from its trajectory. 
\begin{itemize}
    \item These results are typically for penalized estimators, where the penalty consists of a suitable norm $R(\cdot)$. Here $R$ is usually the $\ell_1$ norm for sparse matrices (e.g., \cite{Loh12,Melnyk16,Basu15,Kock2015,Fattahi2019,Wang23}). Other examples include the nuclear norm for low-rank matrices (e.g., \cite{Wang23}), and the (sparse) group LASSO penalty (\cite{Melnyk16}). Penalties imposing a smoothness constraint w.r.t a graph $G$ are comparatively non-existent in the literature, barring the recent work \cite{tyagi2024jointlearn} (discussed earlier).

    \item Alternatively, one could also replace \eqref{eq:tv_pen_ls_matform} with a constrained least-squares estimator where the constraint is of the form $\set{\sum_{\set{l,l'} \in \calE} \norm{A_{l} - A_{l'}}_{1,1} \leq \lambda'}$ which is a convex set. While such a constrained least-squares estimator has not been studied explicitly for structured LDS estimation, recently \cite{tyagi2024learning} studied this for general convex constraints $\calK$ and showed error bounds that depend on the local complexity\footnote{as measured by Talagrand's $\gamma_{1}, \gamma_2$ functionals, see Appendix \ref{appsec:tal_prelim} for definition.} of $\calK$. The analysis therein also involves a restricted eigenvalue condition, which is shown using the concentration result of \cite{krahmer14}. While it seems possible to bound the $\gamma_2$ term using similar ideas as in our analysis, the $\gamma_1$ term is usually more difficult to control. In general, it is unclear whether bounds of similar nature as in the present paper can be obtained from the result in \cite{tyagi2024learning}.
\end{itemize}

Finally, we mention the related problem of signal denoising on networks where we are given $y = \beta^* + \eta$ with $\beta^* \in \matR^m$ unknown, and $\eta$ denoting noise (typically with centered and independent subgaussian entries). Assuming $\beta^*$ has few jumps or varies smoothly across the edges of $G$, a well-studied estimator is 
\begin{align*}
    \est{\beta} = \argmin{\beta \in \matR^m} \set{ \norm{y-\beta}_2^2 + \lambda \sum_{\set{l,l'}\in \calE} \abs{\beta_l - \beta_{l'}}},
\end{align*}
see e.g., \cite{huetter16, wang2016trend, ortelli2021prediction}, for non-asymptotic $\ell_2$ estimation error results. In particular, \cite{huetter16} derives an oracle inequality for the estimation error which is then instantiated for different choices of graphs $G$. Similar to \cite{huetter16}, our results also depend on parameters such as the compatibility factor  and the inverse scaling factor (see Section \ref{subsec:graph-constants}). However our results also depend on new quantities, such as an alternate inverse scaling factor (see section \ref{subsec:graph-constants}) and a dispersion functional capturing the smoothness of the controllability Grammian's of $A^*_l$'s w.r.t $G$ (see \eqref{eq:DeltaG_def}). 

\section{Problem setup and results} \label{prob_setup}
\subsection{Notation} \label{subsec:notation}
For $x \in \matR^n$ and $p \in \mathbb{N}$, $\norm{x}_p$ denotes the usual $\ell_p$ norm of $x$. For $X \in \matR^{n \times m}$, the spectral and Frobenius norms of $X$ are denoted by $\norm{X}_2$ and $\norm{X}_F$ respectively, while $\dotprod{X}{Y} := \Tr(X^\top Y)$ denotes the inner product between $X$ and $Y$. Here $\Tr(\cdot)$ denotes the trace operator. The vector $\vect(X) \in \matR^{nm}$ is formed by stacking the columns of $X$, and $\norm{X}_{p,p} := \norm{\vect(X)}_p$ denotes the entry-wise $\ell_p$ norm of $X$. 
For integers $1 \leq p,q \leq \infty$ we define the norm $\norm{X}_{p \rightarrow q}$ as
\begin{align*}
    \norm{X}_{p \rightarrow q} = \sup_{u \neq 0} \frac{\norm{Xu}_q}{\norm{u}_p}
\end{align*}
with $\norm{X}_{2 \rightarrow 2}$ corresponding to the spectral norm of $X$ (denoted by $\norm{X}_2$ for simplicity). 
%
For $n \times n$ matrices X with eigenvalues $\lambda_i \in \mathbb{C}$, we denote the spectral radius of $X$ by $\rho(X) := \max_{i=1,\dots,n} \abs{\lambda_i}$.

$I_n$ denotes the $n \times n$ identity matrix. The symbol $\otimes$ denotes the Kronecker product between matrices, and $\ones_n \in \matR^n$ denotes the all-ones vector. The canonical basis of $\matR^n$ is denoted by $e_1,e_2,\dots,e_n$. For $x \in \matR^n$ and $\calS \subseteq [n]$, the restriction of $x$ on $\calS$ is denoted by $(x)_{\calS}$, while $\supp(x) \subseteq [n]$ refers to the support of $x$. The unit ($\ell_2$-norm) sphere in $\matR^n$ is denoted by $\mathbb{S}^{n-1}$, while $\mathbb{B}^{n}_p(r)$ denotes the $\ell_p$ ball of radius $r$ in $\matR^{n}$.

For $a,b > 0$, we say $a \lesssim b$ if there exists a constant $C > 0$ such that $a \leq C b$. If $a \lesssim b$ and $a \gtrsim b$, then we write $a \asymp b$.  The values of symbols used for denoting constants (e.g., $c, C, c_1$ etc.) may change from line to line. It will sometimes be convenient to use asymptotic notation using the standard symbols $O(\cdot), \Omega(\cdot), o(\cdot)$ and $\omega(\cdot)$; see for example \cite{algoBook}. 

\subsection{Setup and preliminaries} \label{subsec:setup_prelim}
Let $G = ([m], \calE)$ be a known undirected connected graph. At each node $l \in [m]$, the state $x_{l,t} \in \matR^d$ at time $t$ is described by a linear dynamical system \eqref{eq:mult_linsys}
where $A_l^* \in \matR^{d\times d}$ is unknown. We will assume throughout for convenience that $\eta_{l,t} \stackrel{iid}{\sim} \calN(0,I_d)$ for each $l,t$. All our results continue to hold up to absolute constants, if $\eta_{l,t}$ were independent and centered subgaussian random variables, with their subgaussian norms bounded uniformly by some constant.

Given the data $(x_{l,t})_{l=1, t=0}^{m, T}$ our goal is to estimate $A^*_1,\dots,A^*_m$. Clearly, without additional assumptions on these matrices, we cannot do better then to estimate each $A^*_l$ individually via the OLS estimator. Our interest is in doing better than this naive strategy, and we will see that this is possible when the matrices either vary smoothly, or do not exhibit too many jumps across $\calE$.

The aim is to bound the mean-squared error (MSE) $\frac{1}{m} {\sum_{l=1}^m}\norm{\est{A}_l - A^*_l}_F^2$ for the estimator \eqref{eq:tv_pen_ls_matform}, and establish conditions under which the MSE goes to zero as $m \rightarrow \infty$. Ideally, this should hold under weaker conditions than the aforementioned naive strategy, e.g., when $T$ is small. 
To this end, it will be useful to denote
\begin{align*}
    \Xtil_l &= [x_{l,2} \cdots x_{l,T+1}] \in \matR^{d\times T}, \ X_l = [x_{l,1} \cdots x_{l,T}] \in \matR^{d\times T}\\
    \text{and} \quad E_l &= [\eta_{l,2} \ \eta_{l,3} \ \cdots \ \eta_{l,T+1}] \in \matR^{d\times T}.
\end{align*}
Furthermore, let 
\begin{align}\label{eq:basic_def}
\xtil_l & := \vect(\Xtil_l) \in \matR^{dT} \quad \text{and denote}\nonumber\\
Q &:= \blkdiag(X_l^\top \otimes I_d)_{l=1}^m\quad \text{with}\\
a^*_l &:= \vect(A^*_l)\quad \text{and} \quad \eta_l := \vect(E_l).\nonumber
\end{align}
We form $a^* \in \matR^{md^2}$, $\xtil \in \matR^{mdT}$ and $\eta \in \matR^{mdT}$ by column-stacking $a^*_l$'s, $\xtil_l$'s and $\eta_l$'s respectively. Then the model \eqref{eq:mult_linsys} can be rewritten as
\begin{align}\label{eq:model_vec_form}
   \xtil = Q a^* + \eta. 
\end{align}
Notice that while the diagonal blocks of $Q$ are respectively independent, the entries within each diagonal block are dependent. Moreover, each block of $\eta$, i.e. $\eta_l$, is dependent on the corresponding diagonal block $X_l^\top \otimes I_d$ of $Q$. These dependencies constitute the main challenges in our theoretical analysis.

Denote $D \in \set{-1,1,0}^{\abs{\calE} \times m}$ to be the incidence matrix of any edge-orientation of $G$, and set $\Dtil = D \otimes I_{d^2}$. Then  for any $a \in \matR^{m d^2}$ formed by column-stacking $a_1,\dots, a_m \in \matR^{d^2}$, we have
\begin{align*}
    \Dtil a = 
    \begin{pmatrix}
\vdots \\
a_l - a_{l'} \\
\vdots
\end{pmatrix} \in \matR^{\abs{\calE} d^2}.
\end{align*}
Consequently, \eqref{eq:tv_pen_ls_matform} can be written as 
\begin{align}\label{eq:tv_pen_ls_vecform}
    (\est{a}_1,\dots,\est{a}_m) = \argmin{a\in \matR^{md^2}} \set{\frac{1}{2m} \norm{\xtil - Q a}_2^2 + \lambda \normnew{\Dtil a}_1}.
\end{align}
In order to analyze the performance of \eqref{eq:tv_pen_ls_vecform}, the geometry of the underlying graph will play a central role. In the next section, we introduce a set of key quantities that capture this geometry and govern the statistical and computational behavior of our estimator.
\subsection{Graph–geometry parameters}\label{subsec:graph-constants}
Throughout the sequel we shall repeatedly invoke a small collection of
graph–dependent scalar quantities that capture key geometric properties of the underlying graph \(G = ([m], \calE)\), and influence the difficulty of joint parameter recovery across the network.
%
%
%
\paragraph{Fiedler eigenvalue.} 
Denote by
	\[
	L:=D^{\top}D\in\mathbb{R}^{m\times m},
	\]
	the combinatorial Laplacian and let
	$L^{\dagger}$ be the pseudoinverse of $L$. Denote the eigenvalues of $L$ by $\lambda_1(G) \geq \lambda_2(G) \geq \cdots \geq \lambda_{m-1}(G) >  \lambda_m(G) ( = 0)$. Throughout, $\lambda_{m-1}(G)$ (or simply $\lambda_{m-1}$) stands for the
	\emph{algebraic connectivity}—the smallest positive eigenvalue of $L$ (also known as the Fiedler eigenvalue of $G$). It is well known that $\lambda_{m-1} > 0$ iff $G$ is connected, and more generally, $\lambda_{m-1}$ measures how well-knit the graph is. A small value indicates bottlenecks or sparse regions, making global information propagation harder.
  
\paragraph{Inverse scaling factors.}  The inverse scaling factors of $D$ are defined as follows. 
    \[
	D^{\dagger}
	=\bigl[s_1\;\;s_2\;\;\dots\;\;s_{\vert \calE\vert}\bigr],\qquad
	\mu\;:=\; \max_{j\in[\abs{\calE}]}\|s_j\|_{2}, 
	\]
    and
	\[
	(D^{\dagger})^{\!\top}
	=\bigl[s'_1\;\;s'_2\;\;\dots\;\;s'_m\bigr],\qquad
	\mu' \;:=\; \max_{l\in[m]}\|s'_l\|_{2}.
	\]
These quantify how ``spread out'' the rows and columns of the (pseudo) inverse incidence matrix are -- small values of $\mu,\mu'$ enable better estimates of the model parameters. Interestingly, while $\mu$ appeared in the work \cite{huetter16} for TV based denoising on graphs, we will additionally have the term $\mu'$ appearing in our analysis. It is unclear whether this is an artefact of the analysis, or is intrinsic to the problem. 
While $\mu$ is typically bounded via arguments specialized to $G$ (see \cite{huetter16}), it is also possible to bound both $\mu$ and $\mu'$ using $\lambda_{m-1}$.
\begin{proposition}[Bounds on inverse scaling factors] \label{prp:inversescfac_bd} 
 Let $D$ be the incidence matrix of a connected graph $G$. Then,
	\begin{equation*}
		\mu \leq \dfrac{\sqrt{2}}{\lambda_{m-1}} \wedge \frac{1}{\sqrt{\lambda_{m-1}}}\quad  \text{and} \quad \mu'\le\frac{1}{\sqrt{\lambda_{m-1}}}.
	\end{equation*}
\end{proposition} 
\begin{proof}
    To bound $\mu'$, note that $\mu' \leq \norm{(D^\dagger)^\top}_2 = \norm{D^\dagger}_2$  and also $\norm{D^\dagger}_2 = \frac{1}{\sqrt{\lambda_{m-1}}}$. The same bound is readily seen to apply for $\mu$, while the bound $\mu\leq \frac{\sqrt{2}}{\lambda_{m-1}}$ is taken from \cite[Prop. 14]{huetter16}.
\end{proof}

%
%
\paragraph{Compatibility factors.} These quantities characterize the relationship between sparsity in edge differences (TV) and energy in the parameter vector. Larger \(\kappa\) ensures that sparse total-variation structure leads to well-behaved solutions.
The notion of the so-called  compatibility factor (see \cite{huetter16}) will be crucial. We recall it below.
\begin{definition} \label{def:invfac_compat}
Denoting $\Dtil = D \otimes I_{d^2}$, the compatibility factor of $\Dtil$ for a set $\calT \subseteq \left[d^2\abs{\calE} \right]$ is defined as
\begin{equation*}
    \kappa_{\emptyset} = 1, \quad \kappa_{\calT} := \inf_{\theta \in \matR^{md^2}} \frac{\sqrt{\abs{\calT}} \norm{\theta}_2}{\norm{(\Dtil \theta)_{\calT}}_1} \text{ for } \calT \neq \emptyset. 
\end{equation*}
Moreover, we denote $\kappa := \min_{\calT \subseteq \left[d^2 \abs{\calE}\right]} \kappa_{\calT}$.
\end{definition}
As a small remark, note that the compatibility factor is defined for $\Dtil$ in our setup, as opposed to $D$ in \cite{huetter16}. We refer the reader to \cite{huetter16} where upper bounds on $\invfac$ were derived for different types of graphs. For the compatibility factor $\kappa_{\calT}$, it is not difficult to obtain a lower bound for bounded degree graphs, along the lines of the proof of \cite[Lemma 3]{huetter16}. We just note that given any $\emptyset \neq \calT \subseteq [d^2 \abs{\calE}]$, we can denote the subset of $[\abs{\calE}]$ ``appearing'' in $\calT$ as 
\begin{equation} \label{eq:edge_appear}
\calT_{\calE} := \set{j \in [\abs{\calE}]: \exists i \in \calT \text{ s.t } (j-1)d^2 +1 \leq i \leq jd^2}.
\end{equation}
\begin{proposition}[Bounds on compatibility factor] \label{prop:compatfac_bd}
Let $D \in \set{-1,1,0}^{\abs{\calE} \times m}$ be the incidence matrix of a graph $G = ([m], \calE)$ with maximal degree $\triangle_{\deg}$. For any $\emptyset \neq \calT \subseteq [d^2\abs{\calE}]$, consider $\calT_{\calE} \subseteq [\abs{\calE}]$ as defined in \eqref{eq:edge_appear}. Then,
\begin{align*}
    \kappa_{\calT} \geq \frac{1}{2\min\set{\sqrt{\triangle_{\deg}}, \sqrt{\abs{\calT_{\calE}}}}}.
\end{align*}
\end{proposition}
The proof is outlined in Appendix \ref{appsec:proof_compatfac_bd}. The main difference from \cite[Lemma 3]{huetter16} is the appearance of $\abs{\calT_{\calE}}$ instead of $\abs{\calT}$ in the  bound.
In Table \ref{tb:values_graph_parameters}, we summarize the values of these graph-dependent scalar quantities for several commonly studied graph models. These will be invoked in the next section for instantiating our main result.
\begin{table}[htbp]
\centering
\renewcommand{\arraystretch}{1.3}
\begin{tabular}{|p{2cm}|p{3cm}|p{4cm}|p{4cm}|}
\hline
\textbf{Graph Model} & \(\mu\) (Inverse Scaling) & \(\mu'\) (Alt. Inverse Scaling) & \(\kappa_{\calT}\) (Compatibility Factor) \\
\hline
2D Grid & \(\lesssim \sqrt{\log m}\) (\cite{huetter16}) & \(\lesssim \sqrt{\log m}\) (App. \ref{subsec:invfac_2dgrid}) & \(\gtrsim C\) (Prop. \ref{prop:compatfac_bd}) \\
\hline
Complete Graph & \(\lesssim 1/m\) (Prop. \ref{prp:inversescfac_bd}) & \(\lesssim 1/\sqrt{m}\) (Prop. \ref{prp:inversescfac_bd}) & \(\gtrsim 1/\sqrt{m}\) (Prop. \ref{prop:compatfac_bd}) \\
\hline
Star Graph & \(\le 1\) (Prop. \ref{prp:inversescfac_bd}) & \(\le 1\) (Prop. \ref{prp:inversescfac_bd}) & \(\gtrsim 1/\sqrt{\abs{\calT_{\calE}}}\) (Prop. \ref{prop:compatfac_bd}) \\
\hline
Erdős–Rényi \(G(m, p)\) & \(\lesssim 1/\sqrt{mp}\) (Prop. \ref{prp:inversescfac_bd}) & \(\lesssim 1/\sqrt{mp}\) (Prop. \ref{prp:inversescfac_bd}) & \(\gtrsim 1/\sqrt{mp}\) (Prop. \ref{prop:compatfac_bd}) \\
\hline
\end{tabular}
\caption{Summary of bounds for inverse scaling factors \(\mu, \mu'\) and compatibility factor \(\kappa_\calT\) for some graph models.}  \label{tb:values_graph_parameters}
\end{table}

%
\subsection{Main results} \label{subsec:main_result}
Our results will depend on the so-called controllability Grammian of each system $l\in[m]$, 
\begin{equation*} 
    \Gamma_t(A^*_l) := \sum_{k=0}^{t} (A^*_l)^k ((A^*_l)^k)^\top \in\mathbb{R}^{d\times d}.
\end{equation*}
We will also require that the spectral norms of the following matrices (for each $l \in [m]$) are bounded. 
\begin{align}\label{eq:Atil_l_def}
    \Atil^*_l := \begin{bmatrix}
   I_d & 0 & \hdots & 0\\ 
   A^*_l & I_d & \hdots & 0\\ 
  \vdots  &  & \ddots & \vdots\\
  (A^*_l)^{T-1} & \hdots & A^*_l & I_d
 \end{bmatrix}, \quad \beta:= \max_{l} \norm{\Atil^*_l}_2.
\end{align}
Finally, denoting $$G_l=\sum_{t=1}^{T}\Gamma_{t-1}(A^{*}_l) \ \text{ and } \   
\overline{G}=\frac{1}{m}\sum_{l=1}^m G_l, $$ our results will depend on the quantity 
\begin{equation}\label{eq:DeltaG_def}
    \Delta_G \;:=\; \max_{a,b\in[d]}
    \Biggl(\sum_{l=1}^m\Bigl[(G_l)_{b,a}-(\overline{G})_{b,a}\Bigr]^2\Biggr)^{1/2}.
\end{equation}
{Conceptually, $\Delta_G$ denotes the maximum (over $a,b \in [d]$) standard deviation of the entries $((G_l)_{b,a})_{l=1}^m$.}
The following theorem is our main reult in full generality, its proof is detailed in Section \ref{sec:proof_main_theorem}.
%
%
%
\begin{theorem} \label{thm:main_result_tv_pen}
There exist constants $c < 1/6$, $c' > 0$, and $c_1 \in (0,1)$ such that the following is true. Let $\delta \in (0,c_1)$, $v \geq 1$, \rev{$\scale > 0$} and consider $\calS \subseteq [\abs{\calE} d^2]$. For $F_1, F_2, F_3$ as defined in Lemmas \ref{lem:l2_bd_term1} - \ref{lem:RE_h_1} respectively \rev{(with $F_3$ depending on $\scale$)}, and $\Delta_G$ as in \eqref{eq:DeltaG_def}, suppose $\lambda \geq \frac{2}{m} \max\set{F_1,F_2}$, and the following conditions are satisfied.
 \begin{enumerate}
    \item\label{item_cond1_manithm} $F_3\sqrt{v} \leq c T$,

    \item\label{item_cond2_manithm} $\frac{\beta^2}{m} (\sqrt{mT} + d)(d + v) \leq c T$, and 

    \item\label{item_cond3_manithm} $\frac{\mu}{\sqrt{m}}\left(1+\frac{\sqrt{|S|}}{\kappa_S}+ \rev{\frac{1}{\scale}}\bigl\|(\Dtil a^{*})_{S^c}\bigr\|_1\right)
\left[
d\,\Delta_G
+
\beta^2\,\rev{d}\sqrt{T} \log\!\Bigl(\frac{|\mathcal E|\,d}{\delta}\Bigr)
\right]\leq c T$.
\end{enumerate}
Then w.p at least $1 - 4\exp(-c'\sqrt{v}) - \delta$, 
\begin{equation*}
    \norm{\est{a} - a^*}_2 \leq \frac{2m \lambda}{T} \Big(1 + 3\frac{\sqrt{\abs{\calS}}}{\kappa_{\calS}}\Big) + \sqrt{8 \frac{\lambda m}{T} \norm{\Dtilacomp}_1} \rev{ + \scale}.
\end{equation*}
\end{theorem}
The statement of Theorem \ref{thm:main_result_tv_pen} 
captures the dependence on $A^*_l$ via the terms $\beta$, $\Tr(\Gamma_t(A^*_l))$ and $\Delta_G$, as can be seen from the definitions  of $F_1, F_2$ and $F_3$. Conditions \ref{item_cond1_manithm} - \ref{item_cond3_manithm} arise in the course of ensuring that the matrix $Q$ in \eqref{eq:tv_pen_ls_vecform} satisfies the Restricted Eigenvalue (RE) condition; see Lemma's \ref{lem:RE_h_1}, \ref{lem:RE_h_2} and \ref{lem:cross-term}. The main tool used here is the concentration bound of \cite{krahmer14} for controlling the suprema of second order subgaussian chaos processes involving positive semidefinite (p.s.d) matrices. The terms $F_1$ and $F_2$ arise on account of upper bounding the quantities on the RHS of \eqref{eq:proof_temp2}, see Lemmas \ref{lem:l2_bd_term1} and \ref{lem:linf_bd_term2}. Here, the dependencies inherent in our observations lead us to use tail bounds for the norm of self-normalized vector-valued martingales using ideas from the proof of \cite[Theorem 1]{abbasi11}.
\begin{remark}\label{rem:comp_iid_tv_regression}
As noted in Section \ref{subsec:rel_work}, existing theoretical results for TV-penalized regression problems (e.g. \cite{LiRaskutti20, tran2022generalizedelasticnetsquares}) assume that the rows of the design matrix are i.i.d centered Gaussian's (with a covariance $\Sigma$), and establish the RE condition by adapting the technique of \cite{raskutti_RE2010}. In our setting, due to the nature of the block-diagonal design matrix $Q$ containing dependent rows within each block, we find it convenient to establish the RE condition by a different approach. By first showing that the error term $\est{a} - a^*$ lies in \rev{$\calCps$} with \rev{$\calCps$} defined in \eqref{eq:tran_cone_set}, we decompose any \rev{$h \in \calCps$} as $h_1 + h_2$ where $h_1$ lies in the range space of $\Dtil$ while $h_2$ lies in the null space of $\Dtil$. This allows us to ensure the RE condition by using the inequality \eqref{eq:re_ineq_prelim}, and suitably controlling each of the three quantities on the RHS therein (via Lemmas \ref{lem:RE_h_1}, \ref{lem:RE_h_2} and \ref{lem:cross-term}). \rev{Since $\calCps$ is not a cone, the RE condition is actually shown for $\calCps \cap \set{\norm{h}_2 \geq \scale}$ where $\scale > 0$ is a scale parameter (appearing in Theorem \ref{thm:main_result_tv_pen}).}

Furthermore, the results in \cite{LiRaskutti20, tran2022generalizedelasticnetsquares} also assume that the noise is independent of the design matrix, while in our case, the noise term $\eta$ in \eqref{eq:model_vec_form} shares dependencies with the design $Q$ at a block level. This leads us to use a different approach wherein we bound the quantities on the RHS of \eqref{eq:proof_temp2} using tail bounds for self-normalized vector valued martingales (see Lemmas \ref{lem:l2_bd_term1} and \ref{lem:linf_bd_term2}). 
\end{remark}

In order to interpret Theorem \ref{thm:main_result_tv_pen}, we now consider for convenience the particular case  where there exists a uniform constant \(\rho_{\max}<1\) such that
	\begin{align} \label{eq:specnorm_less_1}
	    \lVert A^{\!*}_{l}\rVert_{2}\;\le\;\rho_{\max},
	\qquad l=1,\dots,m.
	\end{align}  
This condition ensures that each system matrix is Schur stable, i.e., $\rho(A^*_l) < \rho_{\max} < 1$ for each $l$. Under this condition, we can bound $\beta$, as well as $F_1,F_2, F_3$, thus leading to useful corollaries of Theorem \ref{thm:main_result_tv_pen}.
%
%
To this end, set
\begin{align*}
\Delta &:=\; (1-\rho_{\max})^2,\qquad
L_1 \;:=\; \log\!\Big(\frac{d T}{\delta\,\Delta}\Big),\qquad
L_2 \;:=\; \log\!\Big(\frac{d|\mathcal E|}{\delta}\Big), \\
\text{ and } \quad  \Phi_{\calS} &=\; 1 \;+\; \frac{\sqrt{|S|}}{\kappa_{\calS}}
\;+\; \rev{\frac{1}{\scale}}\big\|(\Dtil a^*)_{\calS^c}\big\|_1,
\end{align*}
with $\kappa_{\calS}$ the compatibility factor from Def.~\ref{def:invfac_compat} and $S\subseteq [|\mathcal E|d^2]$ arbitrary. 
Under $\rho_{\max}<1$,  we have  $$\beta\le (1-\rho_{\max})^{-1}= \Delta^{-1/2} \ \text{ and } \ 
{\rm Tr}\,(\Gamma_t(A_l^*)) \leq d/\Delta$$
which then allows us to suitably bound the terms $F_1, F_2$ and $F_3$; see Appendix \ref{subsec:proof_maincorr_stab} for details. 
%
%
%
%
\paragraph{Choice of \(\lambda\).}
Now Theorem~\ref{thm:main_result_tv_pen} requires $\lambda$ to satisfy
\[
\lambda \;\ge\; \frac{2}{m}\max\{F_1,F_2\},
\quad\text{with $F_1,F_2$ from Lemmas~\ref{lem:l2_bd_term1} - \ref{lem:linf_bd_term2}.}
\]
In the stable case, using the aforementioned bounds on $F_1, F_2$, this yields the following valid choice for $\lambda$ (for a suitably large constant $c_1 > 0$)
 \begin{align} \label{eq:lambda_corr_stab}
 \lambda \;=\; \frac{c_1}{m}\sqrt{\frac{T}{\Delta}}\;\max\!\Big\{\,d^{3/2}L_1\,,\;\mu L_2\,\Big\}.
    \end{align}

\medskip
\noindent\textbf{Sample–size conditions.}
%
%
The above considerations allow us to obtain a simplified set of sufficient conditions (for a suitably large constant $C > 0$) which ensure Conditions \ref{item_cond1_manithm}-\ref{item_cond3_manithm} in Theorem \ref{thm:main_result_tv_pen}.
\begin{align}
\text{(C1)}\quad 
& T \;\geq\; C \frac{v}{\Delta^2}\Big(1+\mu'\Phi_{\calS}\sqrt{L_2}\Big)^{\!2} && \text{(from Condition \ref{item_cond1_manithm} in Theorem 1)}, \label{eq:C1}\\
\text{(C2)}\quad 
& T \;\geq\; C \frac{(d+v)^2}{m\,\Delta^2} && \text{(from Condition \ref{item_cond2_manithm} in Theorem 1)}, \label{eq:C2}\\
\text{(C3a)}\quad 
& T \;\geq \; C \frac{\mu^2 \,\Phi_{\calS}^2\,L_2^2 \rev{d^2}}{m\Delta^2} && \text{(from Condition \ref{item_cond3_manithm} in Theorem 1)}, \label{eq:C3a}\\
\text{(C3b)}\quad 
& T \;\geq \; C \frac{\mu}{\sqrt{m}} \,\Phi_{\calS}\,d\,\Delta_G 
&& \text{(from Condition \ref{item_cond3_manithm} in Theorem 1)}. \label{eq:C3b}
\end{align}
%
%
%
%
\begin{remark}[Interpreting and controlling $\Delta_G$]\label{rem:DG}
Note that the dispersion functional $\Delta_G$ defined in \eqref{eq:DeltaG_def}
can be bounded as
\[
  \Delta_G \le \bigg(\,\sum_{l=1}^m \|G_l - \overline{G}\|_F^2\,\bigg)^{1/2},
\]
so any mechanism that suppresses edgewise variations of $G_l$ (e.g., graph smoothness) controls $\Delta_G$. 
Under Schur stability $\|A_l^\ast\|_2\le \rho_{\max}<1$, we also have the crude bounds
\(\Gamma_{t-1}(A_l^\ast)\preceq \Delta^{-1}I_d\) with \(\Delta:=(1-\rho_{\max})^2\), hence
\(\|G_l\|_2\le T/\Delta\) and \(\|G_l\|_F\le \sqrt{d}\,T/\Delta\). Therefore, in the worst case (no cross‑node cancellations) one gets
\[
  \Delta_G \;\lesssim\; \frac{T}{\Delta}\,\sqrt{md},
\]
while whenever the $G_l$’s vary smoothly across edges the  control of $\Delta_G$ can be much tighter. In particular, in Appendix \ref{appendix:control_Delta_gr} we show that
 \begin{equation}\label{eq:interpret_DelG_Frob}
  \Delta_G \;\le\; \frac{L_T(\rho_{\max})}{\sqrt{\lambda_{m-1}(G)}}
\Bigg(\sum_{\{l,l'\}\in E}\|A_l^\ast-A_{l'}^\ast\|_F^2\Bigg)^{1/2}  \quad \text{(see Lemma \ref{lem:deltaG-edge-Frob})}
\end{equation}
where $L_T(\rho_{\max}) \le \frac{2\rho_{\max}\,T}{(1-\rho_{\max}^2)^2}$. Hence $\Delta_G$ is small when the system matrices $(A^\star_l)_{l=1}^m$ vary little across the edges of $G$. A small $\Delta_G$ relaxes the condition~\eqref{eq:C3b}, thus weakening the requirement on $T$. Finally, observe that the numerator of \eqref{eq:interpret_DelG_Frob} can be upper bounded in terms of the total-variation $\sum_{\{l,l'\}\in E}\|A_l^\ast-A_{l'}^\ast\|_{1,1}$ in a straightforward manner.
\end{remark}
We now have the following useful corollary of Theorem \ref{thm:main_result_tv_pen} when each $A^*_l$ is stable. Its proof is outlined in Appendix \ref{subsec:proof_maincorr_stab}.
\begin{corollary}[Stable $A_l^*$]
\label{cor:stable-main}
Assume $\|A^*_l\|_2\le \rho_{\max}<1$ for all $l$, choose $\lambda$ as in \eqref{eq:lambda_corr_stab}, and suppose the sample–size side conditions \eqref{eq:C1}–\eqref{eq:C3b} hold for \rev{$\scale = T^{-1/2}$ and for} some $\delta\in(0,c_2)$, $v\ge 1$ for a suitably small constant $c_2 < 1$. Then there exists a constant $c' > 0$ such that, w.p at least $1-4\exp(-c'\sqrt{v})-\delta$,
\begin{equation}
\label{eq:err-main}
\|\widehat a-a^*\|_2
\;\le\; \frac{2m \lambda}{T} \Big(1 + 3\frac{\sqrt{\abs{\calS}}}{\kappa_{\calS}}\Big) + \sqrt{8 \frac{m\,\lambda}{T} \norm{\Dtilacomp}_1}  \rev{+ \frac{1}{\sqrt{T}}}.
\end{equation}
\end{corollary}
Two natural choices for $S$ are the empty set $S=\emptyset$ in the ``smooth regime'', where the signal  is smooth over the network, and $S=\supp(\Dtil a^{*}) $ in the ``few-changes regime'', wherein the signal is expected to have a small number of changes over the network. Next we specialize our results to these two cases.  
Define the geometry–log multiplier
\[
\mathfrak{M}\ :=\ \Delta^{-1/2}\,\max\!\big\{\, d^{3/2}L_1,\ \mu L_2\,\big\}.
\]
\rev{Since $\mathfrak{M} > 1$, the term $T^{-1/2}$ in \eqref{eq:err-main} is  subsumed by the first term on the RHS therein.}
\begin{corollary}[Two canonical choices of $S$]
\label{cor:stable-twoS}
Under the assumptions of Corollary~\ref{cor:stable-main} and with $\lambda$ as in \eqref{eq:lambda_corr_stab}, there exist absolute $C,c'>0$ such that, with probability at least $1-4e^{-c'\sqrt{v}}-\delta$,
\begin{equation}
\label{eq:twoS-bound}
\frac{1}{\sqrt{m}}\;\|\widehat a-a^*\|_2
\;\le\;
C\left(
\frac{\mathfrak{M}}{\sqrt{mT}}\Big(1+\tfrac{\sqrt{|\calS|}}{\kappa_{\calS}}\Big)
\;+\;
\sqrt{\frac{\mathfrak{M}}{m\sqrt{T}}\;\big\|(\Dtil a^*)_{\calS^c}\big\|_1}
\right).
\end{equation}
 The sample‑size side conditions are \eqref{eq:C1}–\eqref{eq:C3b}.  In particular, the following holds.
\begin{itemize}
\item {\bf \emph{Smooth regime} ($\calS=\emptyset$, $\kappa_\emptyset=1$)}: 
\[
\frac{1}{\sqrt{m}}\|\widehat a-a^*\|_2
\;\lesssim\;
\frac{\mathfrak{M}}{\sqrt{mT}}
\;+\;
\sqrt{\frac{\mathfrak{M}}{m\sqrt{T}}\;\|\Dtil a^*\|_1}.
\]
Here $\Phi_\calS= \rev{\sqrt{T}} \|\Dtil a^*\|_1+1$.

\item {\bf \emph{Few‑changes regime} ($\calS=\operatorname{supp}(\Dtil a^*)$, $s=\|\Dtil a^*\|_0$)}: 
\[
\frac{1}{\sqrt{m}}\|\hat a-a^*\|_2
\;\lesssim\;
\frac{\mathfrak{M}}{\sqrt{mT}}\Big(1+\frac{\sqrt{s}}{\kappa_S}\Big),
\qquad
\Phi_\calS = 1+\frac{\sqrt{s}}{\kappa_{\calS}}\;.
\]
\end{itemize}
\end{corollary}
\begin{remark}[Regime selection]
Equating the two bounds in Corollary~2 yields the switch criterion
\(
\|\Dtil a^*\|_1 \lesssim \frac{\mathfrak{M}}{\sqrt{T}}\big(s/\kappa_{\calS}^2\big),
\)
favoring \(S=\emptyset\) when many changes are tiny (small TV magnitude), and
\(\calS=\mathrm{supp}(\Dtil a^*)\) when the number of changes \(s\) is small  on a graph with large \(\kappa_{\calS}\).
\end{remark}
Notice from the sampling conditions \eqref{eq:C1}–\eqref{eq:C3b} that $T \gtrsim \frac{v}{\Delta^2}$ is necessary for them to be satisfied. In case $T = c \frac{v}{\Delta^2}$ for a large enough constant $c$, then \eqref{eq:C1} would hold provided $\mu' = o(1)$ as $m$ increases.  More precisely, $\mu'$ would need to go to zero sufficiently fast w.r.t $m$. 
%
\begin{corollary}[Small-$T$]
\label{cor:smallT-DeltaG}
Assume $\|A^*_l\|_2\le \rho_{\max}<1$ for all $l$. There exists $c_1 < 1$ such that for any $\delta\in(0,c_1)$, the following holds. Let $v \asymp \log^2\!\big(1/\delta\big)$, $T \asymp v/\Delta^2$  (so $T\asymp \frac{\log^2(1/\delta)}{\Delta^2}$). If in addition
{
\begin{enumerate}
\item[(i)] (Condition \ref{eq:C2}) $m\;\gtrsim\; \frac{d^2}{\log^2(1/\delta)} + \log^2(1/\delta)$, 
%
%
\item[(ii)] \textbf{Smooth regime }($\calS= \emptyset$):
\begin{align*}
 \mu'\Big(1+\rev{\frac{\log(1/\delta)}{\Delta}}\|\Dtil \,a^*\|_1\Big) \ &\lesssim\ 1/\sqrt{\log(\abs{\calE}d/\delta)}, \tag{Condition \ref{eq:C1}}\\
\mu\Big(1+\rev{\frac{\log(1/\delta)}{\Delta}}\|\Dtil \,a^*\|_1\Big)\, &\lesssim\
\min\set{\frac{\sqrt{m}\log^2(1/\delta)}{\,d\Delta_G \Delta^2\,}, \frac{\sqrt{m} \log(1/\delta)}{\rev{d}\log(d\abs{\calE}/\delta)}} \tag{Conditions \ref{eq:C3a}, \ref{eq:C3b}},
\end{align*}
%
%
\item[(iii)] \textbf{Few‑changes regime }($\calS=\operatorname{supp}(\Dtil \,a^*), \ s=\|\Dtil \,a^*\|_0$):
\begin{align*}
\mu'\!\Big(1+\frac{\sqrt{s}}{\kappa_{\calS}}\Big) \ &\lesssim\ 1/\sqrt{\log(\abs{\calE}d/\delta)}, \tag{Condition \ref{eq:C1}} \\
\mu\!\Big(1+\frac{\sqrt{s}}{\kappa_{\calS}}\Big)  \ &\lesssim\
\min\set{\frac{\sqrt{m}\log^2(1/\delta)}{\,d\Delta_G \Delta^2\,}, \frac{\sqrt{m} \log(1/\delta)}{\rev{d}\log(d\abs{\calE}/\delta)}} \tag{Conditions \ref{eq:C3a}, \ref{eq:C3b}},
\end{align*}
\end{enumerate}
}
then with $\lambda$ as in \eqref{eq:lambda_corr_stab} the corresponding bounds from Corollary \ref{cor:stable-twoS} hold w.p at least $1-\delta$. 
\end{corollary}
%
%
%

Let us now instantiate Corollary \ref{cor:smallT-DeltaG} for two examples of well connected graphs $G$  for which $\mu$ and $\mu'$ go to zero sufficiently fast w.r.t $m$. For simplicity, we focus on the smooth-regime ($\calS = \emptyset$). 
%
\begin{example}[Complete graph, $S=\emptyset$, small $T$]
\label{ex:complete-DeltaG}
For the complete graph, $\lambda_{m-1} = m$, and Proposition  \ref{prp:inversescfac_bd} yields $\mu\lesssim 1/m$, $\mu'\lesssim 1/\sqrt{m}$. Since $\abs{\calE} \asymp m^2$, 
%
%
%
$T \asymp \frac{\log^2(1/\delta)}{\Delta^2}$ and $m \gtrsim\; \frac{d^2}{\log^2(1/\delta)} + \log^2(1/\delta)$, hence the conditions involving $\mu, \mu'$ in Corollary \ref{cor:smallT-DeltaG} are satisfied provided
\begin{align*}
\Big(1+\rev{\frac{\log(1/\delta)}{\Delta}}\|\Dtil \,a^*\|_1\Big)\, &\lesssim\
\min\set{\frac{m^{3/2}\log^2(1/\delta)}{\,d\Delta_G \Delta^2\,}, \frac{\sqrt{m}}{\sqrt{\log(dm/\delta)}}
}.
\end{align*}
Moreover, we have that $\mathfrak{M} \lesssim \frac{d^{3/2}}{\Delta^{1/2}} \log(\frac{dT}{\delta \Delta})$. Consequently,  for $\lambda$ as in \eqref{eq:lambda_corr_stab}, 
\begin{align*}
    \frac{1}{m} \norm{ \est{a} - a^*}_2^2 \lesssim \left(\frac{d^{3}}{\Delta} \frac{\log^2(\frac{dT}{\delta \Delta})}{mT} +  \frac{d^{3/2}}{\Delta^{1/2}} \frac{\log(\frac{dT}{\delta \Delta}) \norm{\Dtil a^*}_1}{m\sqrt{T}} \right) 
\end{align*}
holds w.p at least $1-\delta$. Hence if $\norm{\Dtil a^*}_1$ is small enough, then the MSE goes to zero w.r.t $m$, even when $T$ is fixed.
\end{example}

\begin{example}[Erdős–Rényi $G(m,p)$, $S=\emptyset$, small $T$]
\label{ex:ER-DeltaG}
It is well known that for any $\delta \in (0,1)$,  $$\prob(\lambda_{m-1} \geq mp/2) \geq 1-\delta$$ holds provided $p \geq c \log(m/\delta)/m$ for some suitably large constant $c > 0$ (see e.g., \cite[Prop. 4]{tyagi2025jointestimationsmoothgraph}).
Then, Proposition~ ~\ref{prp:inversescfac_bd} gives $\mu,\mu'\lesssim (mp)^{-1/2}$. 
%
Now the quantity $\normnew{\Dtil a^*}_1$ is dependent on $G$, and hence also random, however it is possible to control it via Bernstein's inequality. Indeed, note that 
\begin{align*}
    \normnew{\Dtil a^*}_1 = \sum_{l < l'} \norm{a^*_l - a^*_{l'}}_1 Y_{l,l'}
\end{align*}
where $Y_{l,l'} \stackrel{i.i.d}{\sim} \calB(p)$ (i.i.d Bernoulli random variables) for each $l < l' \in [m]$. Then denoting 
\begin{align*}
    S_1 := \sum_{l < l'} \norm{a^*_l - a^*_{l'}}_1, \ S_2 := \sqrt{\sum_{l < l'} \norm{a^*_l - a^*_{l'}}_1^2}, \ S_3 := \max_{l < l'} \norm{a^*_l - a^*_{l'}}_1
\end{align*}
it follows readily from Bernstein's inequality (see e.g., \cite{boucheron_book}) that w.p $\geq 1-2n^{-c}$ (for constants $c, c_1 > 0$),  
\begin{align*}
    \normnew{\Dtil a^*}_1 \leq c_1 \left(p S_1 + \sqrt{p\log m} S_2 + S_3 \log m \right).
\end{align*}
Since $\abs{\calE} \leq \frac{m(m-1)}{2}$ a.s, $T \asymp \frac{\log^2(1/\delta)}{\Delta^2}$ and $m \gtrsim\; \frac{d^2}{\log^2(1/\delta)} + \log^2(1/\delta)$, hence the conditions involving $\mu, \mu'$ in Corollary \ref{cor:smallT-DeltaG} are satisfied provided 
\begin{align}\label{eq:corr_ergraph_ex_tmp1}
        \frac{1 + \Big(p S_1 + \sqrt{p\log m} S_2 + S_3 \log m \Big) \rev{\frac{\log(1/\delta)}{\Delta}}}{\sqrt{mp}} \lesssim 
           \min\set{\frac{\sqrt{m} \log^2(1/\delta)}{d \Delta_G \Delta^2}, \frac{1}{\rev{d}\log(md/\delta)}}.
    \end{align}
Note that \eqref{eq:corr_ergraph_ex_tmp1} is ensured, for instance, if $p \asymp 1$ and $S_1 + S_2 + S_3 = o(\sqrt{m}/ \polylog (m))$. 

Now, we also have  
$$\mathfrak{M} \lesssim   
        \frac{1}{\sqrt{\Delta}}
        \max\set{d^{3/2} \log\!\Big(\tfrac{d T}{\delta\,\Delta}\Big) , \frac{\log\!\Big(\tfrac{d m}{\delta}\Big)}{\sqrt{mp}}} =: \mathfrak{M}^*.$$
Hence putting everything together,  
we have for $\lambda \asymp \mathfrak{M}^* \frac{\sqrt{T}}{m}$ that w.p at least $1- \delta - 2n^{-c}$,  
\begin{align*}
    \frac{1}{m} \norm{ \est{a} - a^*}_2^2 \lesssim  \frac{(\mathfrak{M}^*)^2}{mT} \;
\;+\;
\frac{\mathfrak{M}^* \left(p S_1 + \sqrt{p\log m} S_2 + S_3 \log m \right)}{m\sqrt{T}}.
\end{align*}
Hence if $(A^*_l)_{l=1}^m$ are sufficiently smooth w.r.t the complete graph (in the sense that $S_1, S_2$ and $S_3$ are suitably small), and $p$ is suitably large, then the MSE goes to zero w.r.t $m$, even when $T$ is fixed.
\end{example}
Next we instantiate Corollary \ref{cor:stable-twoS} for examples of $G$ where $T$ is required to grow with $m$, albeit mildly, for consistency of the MSE (w.r.t $m$). For simplicity, we again focus on the smooth-regime ($\calS = \emptyset$). 
\begin{example}[Star graph, $S=\emptyset$, moderate-sized $T$] \label{ex:star_graph}
For the star graph $\lambda_{m-1} = 1$, and Proposition \ref{prp:inversescfac_bd} yields $\mu,\mu' \leq 1$. Also,  $\abs{\calE} \asymp m$. Take $v \asymp \log^2(1/\delta)$ as before, then conditions \ref{eq:C1}-\ref{eq:C3b} correspond to
\begin{align*}
    T \gtrsim \max\set{\frac{\log^2(1/\delta)}{\Delta^2}\Phi_\calS^2  \log(dm/\delta), 
    \frac{d^2 + \log^4(1/\delta)}{m\,\Delta^2},
   \frac{\Phi_\calS\,d\,\Delta_G}{\sqrt{m}}, \rev{\frac{\Phi_\calS^2  d^2\log^2(dm/\delta)}{m\Delta^2}}},
\end{align*}
where we recall $\Phi_\calS = 1+\rev{\sqrt{T}}\|\Dtil \,a^*\|_1$. 

To get a feel for the above condition on $T$, consider the very smooth regime where $\|\Dtil \,a^*\|_1 \rev{\lesssim 1/\sqrt{T}}$ \rev{and} $\Delta_G \leq c$ for some constant $c \geq 1$. Then if $m \gtrsim\; d^2 + \log^4(1/\delta)$, we require $T \gtrsim \frac{\log^2(1/\delta)}{\Delta^2} \rev{\log^2(dm/\delta)}$ i.e., $T$ is only required to grow logarithmically w.r.t both $m$ and dimension $d$.

We also have the bound 
$$\mathfrak{M} \lesssim \frac{1}{\sqrt{\Delta}} \max\set{d^{3/2} \log\Big(\frac{dT}{\delta \Delta} \Big), \log(dm/\delta)} =: \mathfrak{M}_1^*,$$ 
hence, for $\lambda$ as in \eqref{eq:lambda_corr_stab},
\begin{align*}
  \frac{1}{m}\|\widehat a-a^*\|_2^2
\;\lesssim\;
\frac{(\mathfrak{M}_1^*)^2}{mT}
\;+\;
\frac{\mathfrak{M}_1^*}{m\sqrt{T}}\;\|\Dtil a^*\|_1
\end{align*}
holds w.p at least $1-\delta$. 
\end{example}
\begin{example}[2D grid, $S=\emptyset$, moderate-sized $T$]\label{ex:2d_grid}
For the 2D grid, we have $\mu, \mu' \lesssim \sqrt{\log m}$ as shown in Table \ref{tb:values_graph_parameters}. Also,  $\abs{\calE} \asymp m$. Take $v \asymp \log^2(1/\delta)$ as before, then conditions \ref{eq:C1}-\ref{eq:C3b} correspond to
\begin{align*}
    T \gtrsim \max\set{\frac{\log^2(1/\delta)}{\Delta^2} \Phi_\calS^2  \log^2(dm/\delta) , 
    \frac{d^2 + \log^4(1/\delta)}{m\,\Delta^2},
    \Phi_\calS\,d\,\Delta_G \sqrt{\frac{\log m}{m}}, \rev{\frac{\Phi_\calS^2  d^2 \log m \log^2(dm/\delta)}{m\Delta^2}}},
\end{align*}
where $\Phi_\calS = 1+ \rev{\sqrt{T}}\|\Dtil \,a^*\|_1$. As before, consider for simplicity the very smooth regime where $\rev{\|\Dtil \,a^*\|_1 \lesssim 1/\sqrt{T}}$ \rev{and}  $\Delta_G \leq c$ for a constant $c \geq 1$. Then if $\frac{m}{\log m} \gtrsim d^2 + \log^4(1/\delta)$, we require $T \gtrsim \frac{\log^2(1/\delta)}{\Delta^2} \log^2(dm/\delta)$.

Using the bound
$$\mathfrak{M} \lesssim \frac{1}{\sqrt{\Delta}} \max\set{d^{3/2} \log\Big(\frac{dT}{\delta \Delta} \Big), \log^{3/2}(dm/\delta)} =: \mathfrak{M}_2^*$$ 
we then have for $\lambda$ as in \eqref{eq:lambda_corr_stab}, that  
\begin{align*}
  \frac{1}{m}\|\widehat a-a^*\|_2^2
\;\lesssim\;
\frac{(\mathfrak{M}_2^*)^2}{mT}
\;+\;
\frac{\mathfrak{M}_2^*}{m\sqrt{T}}\;\|\Dtil a^*\|_1
\end{align*}
holds w.p at least $1-\delta$. 
\end{example}
\begin{remark}[Estimating each $A^*_l$ separately] \label{rem:ind_est_comparison}
    As a sanity check, it is worth comparing the error bounds in the above examples with that obtained by simply estimating each $A^*_l$ individually via the ordinary least-squares (OLS) estimator. Let $\est{A}_{\ols,l}$ be the OLS estimate at node $l$, then it was shown in \cite{Jedra20} that w.p at least $1-\delta$, 
    \begin{align*}
           \norm{\est{A}_{\ols,l} - A^*_l}_2 \lesssim \sqrt{\frac{\log(1/\delta) + d}{T}} \quad \text{ if } \quad T \gtrsim  \Delta^{-1}(\log(1/\delta) + d). 
    \end{align*}
    By taking a union bound, this means that w.p at least $1-\delta$, 
    \begin{align} \label{eq:ols_mse_bound}
       \frac{1}{m} \sum_{l=1}^m \norm{\est{A}_{\ols,l} - A^*_l}_F^2 \lesssim  \frac{d\log(m/\delta) + d^2}{T} \quad \text{ if } \quad T \gtrsim  \Delta^{-1}(\log(m/\delta) + d).    
    \end{align}
    For sufficiently smooth $A^*_l$'s, we saw in Examples \ref{ex:star_graph} and \eqref{ex:2d_grid} that the dependence of $T$ can at times be logarithmic w.r.t both $m$ and $d$. Furthermore, notice that if $T = \Theta(\log(m))$, then the MSE in \eqref{eq:ols_mse_bound} is $O(1)$ w.r.t $m$ while the corresponding bounds in Examples \ref{ex:star_graph} and \ref{ex:2d_grid} are $o(1)$.
\end{remark}

\section{Proof of Theorem \ref{thm:main_result_tv_pen}} \label{sec:proof_main_theorem}
Since $\est{a}$ is a solution of \eqref{eq:tv_pen_ls_vecform} and $a^*$ is feasible, we have %
\begin{align*}
    \frac{1}{2m} \norm{\xtil - Q \est{a}}_2^2 + \lambda \norm{\Dtil \est{a}}_1 \leq  \frac{1}{2m} \norm{\xtil - Q a^*}_2^2 + \lambda \norm{\Dtil a^*}_1.
\end{align*}
Using \eqref{eq:model_vec_form} in the above inequality, we obtain after some simple calculations
\begin{align} \label{eq:proof_temp1}
    \frac{1}{2m}\norm{Q (a^* - \est{a})}_2^2 \leq \frac{1}{m} \dotprod{\est{a} - a^*}{Q^\top \eta} + \lambda \norm{\Dtil a^*}_1 - \lambda \norm{\Dtil \est{a}}_1.
\end{align}
Let $\Pi, \Pitil$ denote the projection matrices for $\nullsp(D)$ and $\nullsp(\Dtil)$ respectively. Since $G$ is connected, we have 
$$\Pi = \frac{1}{m}\ones_m \ones_m^\top \ \text{ and } \  \Pitil = (\frac{1}{m}\ones_m \ones_m^\top) \otimes I_{d^2}.$$ 
Moreover, $\Dtil^\dagger \Dtil = (D^\dagger D) \otimes I_{d^2} = I_{md^2} - \Pitil$ is the projection matrix for the orthogonal complement of $\nullsp(\Dtil)$. With this in mind, we can bound 
\begin{align} \label{eq:proof_temp2}
     \dotprod{\est{a} - a^*}{Q^\top \eta} &= \eta^\top Q \Pitil(\est{a} - a^*) + \eta^\top Q (\Dtil^\dagger \Dtil)(\est{a} - a^*) \nonumber \\
    &\leq \norm{\Pitil Q^\top \eta}_2 \norm{\est{a} - a^*}_2 + \norm{((D^\dagger)^\top \otimes I_{d^2}) Q^\top \eta}_{\infty} \norm{\Dtil(\est{a} - a^*)}_1.
\end{align}
The following lemma's bound the terms $\norm{\Pitil Q^\top \eta}_2$ and $\norm{((D^\dagger)^\top \otimes I_{d^2}) Q^\top \eta}_{\infty}$. The crux of the proof is based on  tail bounds for the norm of self-normalized vector-valued martingales, using ideas in the proof of \cite[Theorem 1]{abbasi11}. 
%
%
\begin{lemma} \label{lem:l2_bd_term1}
    There exist constants $c_1 > 0$, $c_2 \in (0,1)$ such that the following is true. For any $\delta \in (0,1)$, denote
    \begin{equation*}
        \zeta_1(m,T,\delta) :=  c_1 \Big(\sum_{l=1}^m \sum_{t=0}^{T-1} \Tr(\Gamma_t(A^*_l)) \Big)  \log(1/\delta). 
    \end{equation*}
    Then for $\delta{\in}(0,c_2)$, it holds with probability at least $1-2\delta$ that 
    \begin{align*}
        \norm{\Pitil Q^\top \eta}_2 \leq \sqrt{2}\Big(\frac{\zeta_1(m,T,\delta)}{m} + 1 \Big)^{1/2} \Big( \log(1/\delta) + \frac{d^2}{2} \log\Big(\frac{\zeta_1(m,T,\delta)}{m} + 1 \Big)\Big)^{1/2} =: F_1.
    \end{align*}
\end{lemma}

%
%
\begin{lemma} \label{lem:linf_bd_term2}
    There exist constants $c_1,c_2 > 0$ and $c_3 \in (0,1)$ such that the following is true. For any $\delta \in (0,1)$, denote
    \begin{equation*}
        \zeta_2(m,T,\delta) := c_1 \invfac^2 \Big[\max_{l\in[m], i\in [d]} e_i^\top \Big( \sum_{t=1}^T\Gamma_{t-1}(A^*_l) \Big) e_i\Big]  \log^2 \Big(\frac{d^2 \abs{\calE}}{\delta} \Big).
    \end{equation*}
    Then for any $\delta \in (0,c_3)$, it holds with probability at least $1-2\delta$ that
    \begin{align*}
        \norm{((D^\dagger)^\top \otimes I_{d^2}) Q^\top \eta}_{\infty} \leq c_2 \zeta_2^{1/2}(m,T, \delta)  =: F_2.
    \end{align*}
\end{lemma}
The proofs are provided in Sections \ref{subsec:proof_l2_bd_term1} and \ref{subsec:proof_linf_bd_term2}. 
Conditioned on the events of Lemma's \ref{lem:l2_bd_term1} and \ref{lem:linf_bd_term2}, we obtain from \eqref{eq:proof_temp1} and \eqref{eq:proof_temp2} that 
\begin{align*}
    \frac{1}{2m}\norm{Q (a^* - \est{a})}_2^2 
    \leq \frac{F_1}{m}\norm{\est{a} - a^*}_2 + \frac{F_2}{m} \norm{\Dtil(\est{a} - a^*)}_1 + \lambda \norm{\Dtil a^*}_1 - \lambda \norm{\Dtil \est{a}}_1.
\end{align*}
Choosing $\lambda \geq \frac{2}{m}\max\set{F_1, F_2}$, this implies
\begin{align} \label{eq:proof_temp3}
    \frac{1}{m}\norm{Q (a^* - \est{a})}_2^2 
    \leq \lambda\norm{\est{a} - a^*}_2 + \lambda \norm{\Dtil(\est{a} - a^*)}_1 + 2\lambda \norm{\Dtil a^*}_1 - 2\lambda \norm{\Dtil \est{a}}_1.
\end{align}
Now for any $\calS \subseteq \left[d^2 \abs{\calE} \right]$, it is not difficult to verify that 
\begin{align*}
    \norm{\Dtil(\est{a} - a^*)}_1 + \norm{\Dtil a^*}_1 -  \norm{\Dtil \est{a}}_1 \leq  2\norm{(\Dtil(\est{a} - a^*))_{\calS}}_1 + 2\norm{(\Dtil a^*)_{\calS^c}}_1. 
\end{align*}
Plugging this in \eqref{eq:proof_temp3}, we obtain the inequality
\begin{align} \label{eq:proof_temp4}
        \frac{1}{m}\norm{Q (a^* - \est{a})}_2^2 
    \leq \lambda\norm{\est{a} - a^*}_2 + 3\lambda \norm{(\Dtil(\est{a} - a^*))_{\calS}}_1 + 4\lambda \norm{(\Dtil a^*)_{\calS^c}}_1 - \lambda \norm{(\Dtil(\est{a} - a^*))_{\calS^c}}_1.
\end{align}
%
%
%
This implies that $\est{a} - a^*$ lies in the set $\rev{\calCps}$, where 
%
\begin{align} \label{eq:tran_cone_set}
    \rev{\calCps} := \set{h: \norm{(\Dtil h)_{\calS^c}}_1  \leq \norm{h}_2 + 3\norm{(\Dtil h)_{\calS}}_1 + 4\norm{(\Dtil a^*)_{\calS^c}}_1}.
\end{align}
Moreover, \eqref{eq:proof_temp4} also implies 
\begin{align}\label{eq:main_ineq_proof}
\frac{1}{m}\norm{Q (a^* - \est{a})}_2^2 
    &\leq 
    \lambda\norm{\est{a} - a^*}_2 + 3\lambda \norm{(\Dtil(\est{a} - a^*))_{\calS}}_1 + 4\lambda \norm{(\Dtil a^*)_{\calS^c}}_1 \nonumber \\
    &\leq \lambda\Big(1 + 3\frac{\sqrt{\abs{\calS}}}{\kappa_{\calS}}\Big) \norm{\est{a} - a^*}_2 + 4\lambda \norm{(\Dtil a^*)_{\calS^c}}_1 \quad \text{(recall Definition \ref{def:invfac_compat})}. 
\end{align}
%
%
%
Note that conditioned on Lemmas \ref{lem:l2_bd_term1} and \ref{lem:linf_bd_term2}, the inequality \eqref{eq:main_ineq_proof} holds simultaneously for all $\calS \subseteq [\abs{\calE} d^2]$. From this point, \rev{a typical goal would be} to establish a restricted eigenvalue (RE) condition wherein for some $\kappa > 0$ and any fixed $\calS$ (the choice of which is arbitrary), $\norm{Qh}_2^2 \geq \kappa \norm{h}_2^2$ holds simultaneously for all \rev{$h \in \calCps$}. \rev{However, $\calCps$ is not a cone, hence $h \in \calCps$ does not necessarily imply $h/\norm{h}_2 \in \calCps$. Thus some care is needed. Introducing a scale parameter $\scale > 0$, note that we only need to show the RE condition for $\calCps \cap \set{\norm{h}_2 \geq \scale}$. Indeed, if $\est{a} - a^* \in \calCps \cap \set{\norm{h}_2 < \scale}$ then the stated error bound in Theorem \ref{thm:main_result_tv_pen}  holds trivially.}
%
%
%
%
\paragraph{RE analysis.}
Recall from earlier the definition of $\Pitil$, we can write 
\begin{equation*}
h = h_1 + h_2 \ \text{ where } \ h_1 = \Dtil^\dagger \Dtil h \ \text{ and } \ h_2 = \Pitil h.    
\end{equation*}
\rev{Now define the set (for $\scale > 0)$
\begin{equation*}
        \calC_\calS := \set{h: \norm{(\Dtil h)_{\calS^c}}_1  \leq \norm{h}_2 + 3\norm{(\Dtil h)_{\calS}}_1 + \rev{\frac{4}{\scale}}\norm{(\Dtil a^*)_{\calS^c}}_1}
\end{equation*}
and observe that $h/\norm{h}_2 \in \calC_\calS \cap \unitsph$ if $h \in \calCps \cap \set{\norm{h}_2 \geq \scale}.$
This means that for any $h \in \calCps \cap \set{\norm{h}_2 \geq \scale}$}
\begin{equation*}
    \norm{Qh}_2^2 \geq \Big(\inf_{h \in \calC_{\calS} \cap \unitsph} \norm{Qh}_2^2 \Big) \norm{h}_2^2
\end{equation*}
and so our focus now is on lower bounding the term within parentheses. To this end, we begin by observing that
\begin{align} \label{eq:re_ineq_prelim}
    \inf_{h \in \calC_{\calS} \cap \unitsph} \norm{Qh}_2^2 \geq \inf_{h \in \calC_{\calS} \cap \unitsph} \norm{Qh_1}_2^2 + \inf_{h \in \calC_{\calS} \cap \unitsph} \norm{Qh_2}_2^2 - 2\sup_{h \in \calC_{\calS} \cap \unitsph} \abs{\dotprod{Qh_1}{ Qh_2}},
\end{align}
hence the strategy now is to suitably bound the terms in the RHS above. This is stated in the following lemma.
%
%
%
\begin{lemma}[RE for $h_1$ term] \label{lem:RE_h_1}
    %
        %
    %
    There exist constants $c_1, c_2, C_1 \geq 1$ such that the following is true. For any $v \geq 1$, it holds w.p at least $1-2\exp(-c_2\sqrt{v})$ that 
    \begin{align*}
     \forall h \in \calC_{\calS} \cap \unitsph: \  \norm{Qh_1}_2^2 \geq T\norm{h_1}_2^2 
        &- c_1 F_3 - G_3 \sqrt{v}, 
    \end{align*}
    where $F_3, G_3$ are defined as follows.
    \begin{align*}
        F_3 &:= C_1 \beta^2 \Bigl[(\invfac')^2  \Big(4\frac{\sqrt{\abs{\calS}}}{\kappa_{\calS}} + \rev{\frac{4}{\scale}}\norm{(\Dtil a^*)_{\calS^c}}_1 + 1 \Big)^2 \log(d^2\abs{\calE}) \\
        &+ 
    \invfac' \sqrt{T}  \Big(4\frac{\sqrt{\abs{\calS}}}{\kappa_{\calS}} + \rev{\frac{4}{\scale}}\norm{(\Dtil a^*)_{\calS^c}}_1 + 1 \Big) \sqrt{\log(d^2\abs{\calE})} + \sqrt{T} \Bigr], \\
         G_3 &:= C_1 \beta^2 \left(\invfac'  \Big(4\frac{\sqrt{\abs{\calS}}}{\kappa_{\calS}} + \rev{\frac{4}{\scale}}\norm{(\Dtil a^*)_{\calS^c}}_1 + 1 \Big) \sqrt{\log(d^2\abs{\calE})} + \sqrt{T} \right),
    \end{align*}
    where $\kappa_{\calS}$ is the compatibility factor from Definition \ref{def:invfac_compat}.
\end{lemma}
The proof, which is outlined in Section \ref{subsec:proof_RE_h1}, makes use of a concentration result of \cite{krahmer14} for controlling the suprema of second order subgaussian chaos processes involving positive semidefinite (p.s.d) matrices (recalled as Theorem \ref{thm:krahmer_chaos} in Appendix \ref{appsec:sup_chaos_krahmer}). 
%
%
\begin{lemma}[RE for $h_2$ term] \label{lem:RE_h_2}
 There exist constants $c_1, c_2, c_3\geq 1$ such that the following is true. For any $v \geq 1$, it holds w.p at least $1-2\exp(-c_3\sqrt{v})$ that 
   
\[
\forall h \in \mathcal{C}_S \cap \mathbb{S}^{md_2 - 1} : 
\|Q h_2\|_2^2 \geq T \|h_2\|_2^2 
- \frac{c_1}{m} \beta^2 (\sqrt{mT} + d)d 
- \frac{c_2}{m} \beta^2 (\sqrt{mT} + d)v.
\]
\end{lemma}
The proof, presented in Section~\ref{subsec:proof_RE_h2}, also relies on the aforementioned concentration inequality due to~\cite{krahmer14}. We now present a bound on the ``cross term'' involving both $h_1$ and $h_2$. 
{\begin{lemma}[Control of the cross term]\label{lem:cross-term}
There exist absolute constants $c_1\ge 1$ and $c_2 \in (0,1)$ such that the following holds. 
For any $\delta\in(0,c_2)$, with probability at least $1-\delta$, for all 
$h\in C_\calS \cap \mathbb{S}^{md^2-1}$,
\begin{equation}\label{eq:L5-new}
\bigl|\langle Q h_1,\, Q h_2\rangle\bigr|
\;\le\;
c_1\frac{\!\mu}{\sqrt{m}}\left(1+\frac{\sqrt{|S|}}{\kappa_S}+\rev{\frac{1}{\scale}}\bigl\|(D a^{*})_{S^c}\bigr\|_1\right)
\!\left[
d\,\Delta_G
\;+\;
\beta^2 \rev{d}\,\sqrt{T}\,\log\!\Bigl(\frac{|\mathcal E|\,d}{\delta}\Bigr)
\right],
\end{equation}
where $\Delta_G$ as defined in \eqref{eq:DeltaG_def}
\end{lemma}
}

The proof of this lemma is given in  Section \ref{subsec:proof_RE_h1_h_2}.

\paragraph{Putting everything together.} 
Using Lemma's \ref{lem:RE_h_1}, \ref{lem:RE_h_2} and \ref{lem:cross-term} in \eqref{eq:re_ineq_prelim}, 
we get absolute constants $c_1,\dots,c_6>0$ such that, with probability at least $1-4e^{-c_1\sqrt v}-\delta$,
\begin{align*}
\inf_{h\in C_S\cap \mathbb{S}^{md^2-1}}\|Qh\|_2^2
\;\ge\;&
T\|h_1\|_2^2 \;-\; c_2 F_3 \;-\; c_3 G_3\sqrt v
\;\\
&+\; T\|h_2\|_2^2 \;-\; \frac{c_4}{m}\,\beta^2(\sqrt{mT}+d)\,d \;-\; \frac{c_5}{m}\,\beta^2(\sqrt{mT}+d)\,v \\
&\;-\; \frac{c_6\, \mu}{\sqrt{m}}\!\left(1+\frac{\sqrt{|S|}}{\kappa_S}+ \rev{\frac{1}{\scale}}\bigl\|(D a^{*})_{\calS^c}\bigr\|_1\right)
\left[
d\,\Delta_G
\;+\;
\beta^2\,\sqrt{T}\log\!\Bigl(\frac{|\mathcal E|\,d}{\delta}\Bigr)
\right].
\end{align*}
Using $G_3\le F_3$ and $\|h_1\|_2^2+\|h_2\|_2^2=1$, the bound simplifies to
\begin{equation*}
\begin{aligned}
\inf_{h\in C_S\cap \mathbb{S}^{md^2-1}}\|Qh\|_2^2
\;\ge\;
T \;&-\; c\!\Bigg(
F_3\sqrt v \;+\; \frac{\beta^2}{m}\,(\sqrt{mT}+d)\,(d+v)
\;\\
&+\frac{\mu}{\sqrt{m}}\;
\left(1+\frac{\sqrt{|S|}}{\kappa_S}+ \rev{\frac{1}{\scale}}\bigl\|(\Dtil a^{*})_{S^c}\bigr\|_1\right)
\left[
d\,\Delta_G
+
\beta^2\,\sqrt{T}\log\!\Bigl(\frac{|\mathcal E|\,d}{\delta}\Bigr)
\right]\!\Bigg).
\end{aligned}
\end{equation*}
Therefore, if the following (sufficient) conditions hold:
\begin{equation*}
\begin{aligned}
\text{(i)}\;& F_3\sqrt v \;\le\; \frac{T}{6c},\\[2pt]
\text{(ii)}\;& \displaystyle \frac{\beta^2}{m}\,(\sqrt{mT}+d)\,(d+v) \;\le\; \frac{T}{6c},\\[2pt]
\text{(iii)}\;& \displaystyle
\frac{\mu}{\sqrt{m}}\left(1+\frac{\sqrt{|S|}}{\kappa_S}+ \rev{\frac{1}{\scale}}\bigl\|(\Dtil a^{*})_{S^c}\bigr\|_1\right)
\left[
d\,\Delta_G
+
\beta^2\,\sqrt{T} \log\!\Bigl(\frac{|\mathcal E|\,d}{\delta}\Bigr)
\right]
\;\le\; \frac{T}{6c},
\end{aligned}
\end{equation*}
then $\inf_{h\in C_S\cap \mathbb{S}^{md^2-1}}\|Qh\|_2^2 \ge T/2$. 

\rev{Hence if $\est{a} - a^* \in \calCps \cap \set{\norm{h}_2 \geq \scale}$, then this implies $\norm{Q(\est{a} - a^*)}_2^2 \geq (T/2) \norm{\est{a} - a^*}_2^2$}. Plugging this into \eqref{eq:main_ineq_proof} yields
\[
\frac1m\|a^*-\widehat a\|_2^2
\;\le\;
\frac{2\lambda}{T}\left(1+3\,\frac{\sqrt{|S|}}{\kappa_S}\right)\|\widehat a-a^*\|_2
\;+\;
\frac{8\lambda}{T}\,\bigl\|(\Dtil a^{*})_{S^c}\bigr\|_1,
\]
%
%
%
\rev{which implies
\begin{equation*}
  \|a^*-\widehat a\|_2 \leq \frac{2m \lambda}{T} \Big(1 + 3\frac{\sqrt{\abs{\calS}}}{\kappa_{\calS}}\Big) + \sqrt{8 \frac{\lambda m}{T} \norm{\Dtilacomp}_1}.   
\end{equation*}
The above bound is seen to satisfy the error bound in Theorem \ref{thm:main_result_tv_pen} (which is larger than $\scale$) as well. This completes the proof.
}
\subsection{Proof of Lemma \ref{lem:l2_bd_term1}} \label{subsec:proof_l2_bd_term1}
We start by observing that
\begin{align} \label{eq:Qtop_eta}
    Q^\top \eta 
    = 
    \begin{bmatrix}
(X_{{1}}\otimes I_d) \eta_1 \\
\vdots\\
(X_m \otimes I_d) \eta_m 
\end{bmatrix}
\end{align}
where we recall $\eta_l = \vect(E_l)$, and the definitions of $X_l, E_l \in \matR^{d \times T}$ from Section \ref{subsec:setup_prelim}. Denoting \rev{$M_{l,t} := \sum_{s=1}^t \eta_{l,s+1} x_{l,s}^\top$}, note that $(X_l \otimes I_d) \eta_l = \vect(M_{l,T})$, and we have
\begin{align} \label{eq:pitil_lem_temp1}
    \norm{\Pitil Q^\top \eta}_2^2 = \frac{1}{m} \norm{\sum_{l=1}^m \vect(M_{l,T})}_2^2 = \frac{1}{m} \norm{\sum_{l=1}^m M_{l,T}}_F^2.
\end{align}
For any system $l \in [m]$, denote $\calF_{l,t} := \sigma(\eta_{l,1},\dots,\eta_{l,t})$ to be the sigma algebra at time $t \geq 1$, leading to a filtration $(\calF_{l,t})_{t=1}^{\infty}$. Since $M_{l,t} = M_{l,t-1} + \eta_{l,t+1} x_{l,t}^\top$, and $M_{l,t}$ is $\calF_{l,t+1}$-measurable, hence $\expec[M_{l,t} \vert \calF_{l,t}] = M_{l,t-1}$, and $(M_{l,t})_{t=1}^{T}$ is a martingale. 

Before deriving a high-probability upper bound for $\norm{\Pitil Q^\top \eta}_2^2$, let us examine what its expectation looks like. To this end, note that
\begin{equation*}
    \expec\left[\norm{M_{l,t}}_F^2 \vert \calF_{l,t} \right] = \norm{M_{l,t-1}}_F^2 + d\norm{x_{l,t}}_2^2.
\end{equation*}
Together with \eqref{eq:pitil_lem_temp1}, and the fact that $M_{l,T}, M_{l',T}$ are independent and centered, this implies 
\begin{align*}
  \expec\norm{\Pitil Q^\top \eta}_2^2 = \frac{1}{m} \sum_{l=1}^m \expec\left[\norm{M_{l,T}}_F^2\right] = \frac{d}{m}\sum_{l=1}^m \sum_{s=1}^T \expec\left[\norm{x_{l,s}}_2^2 \right] = \frac{d}{m}\sum_{l=1}^m \expec\left[\norm{\vect(X_l)}_2^2 \right] = \frac{d}{m}\sum_{l=1}^m\norm{\Atil^*_l}_F^2,
\end{align*}
where we recall $\Atil^*_l$ from \eqref{eq:Atil_l_def}.

Now to find an upper bound on $\norm{\Pitil Q^\top \eta}_2^2$, we first write
\begin{align*}
   \sum_{l=1}^m \vect(E_l X_l^\top) = \sum_{l=1}^m \sum_{s=1}^T \vect(\eta_{l,s+1} x_{l,s}^\top) = \sum_{s=1}^T\left(\sum_{l=1}^m x_{l,s} \otimes \eta_{l,s+1} \right) =: S_T.
\end{align*}
Clearly, $(S_{t})_{t=1}^{T}$ is a martingale and $\norm{\Pitil Q^\top \eta}_2^2 = \frac{1}{m}\norm{S_T}_2^2$. 
Denoting $V_T := \sum_{s=1}^T \sum_{l=1}^m (x_{l,s} x_{l,s}^\top) \otimes I_d$, we obtain for any $\Vbar \succ 0$ that
\begin{align} \label{eq:ptilbd_temp2}
    \frac{1}{m}\norm{S_T}_2^2 = \frac{1}{m} \norm{(V_T + \Vbar)^{1/2}(V_T + \Vbar)^{-1/2} S_T}_2^2 \leq \frac{1}{m} \norm{V_T + \Vbar}_2 \norm{(V_T + \Vbar)^{-1/2} S_T}_2^2.
\end{align}
The term $(V_T + \Vbar)^{-1/2} S_T$ is a self-normalized vector valued martingale. In a completely analogous manner to the proof of \cite[Proposition 1]{tyagi2024jointlearn}, which in turn follows the steps in the proof of \cite[Theorem 1]{abbasi11}, it is easy to show that (see Appendix \ref{appsec:self_norm_args} for details) 
\begin{align} \label{eq:pitilbd_selfnorm}
    \prob\left(\norm{(V_T + \Vbar)^{-1/2} S_T}_2^2 \leq 2\log\left(\frac{\det\left[(V_T + \Vbar)^{1/2} \right] \det\left[\Vbar^{-1/2} \right]}{\delta} \right) \right) \geq 1-\delta.
\end{align}
We will now bound $\norm{V_T}_2$, which together with \eqref{eq:ptilbd_temp2}, \eqref{eq:pitilbd_selfnorm}, and a suitable choice of $\Vbar$ will complete the proof.

To this end, note that
\begin{align} \label{eq:vT_bd_quad}
    \norm{V_T}_2 = \norm{\sum_{s=1}^T \sum_{l=1}^m (x_{l,s} x_{l,s}^\top)}_2 \leq \sum_{s=1}^T \sum_{l=1}^m \norm{x_{l,s}}_2^2 = \etatil^\top (\Atil^*)^\top \Atil^* \etatil
\end{align}
where $\etatil_l$ is formed by column-stacking $\eta_{l,1},\dots,\eta_{l,T}$; $\etatil$ is formed by column-stacking $\etatil_1,\dots, \etatil_m$, and $\Atil^* := \blkdiag(\Atil^*_1,\dots,\Atil^*_m)$. Invoking the tail-bound in  \cite[Theorem 2.1]{hsu2012tail} for random positive-semidefinite quadratic forms, and denoting $\Sigma := (\Atil^*)^\top \Atil^*$, we have for any $t > 0$,
\begin{align*} 
    \prob\left(\etatil^\top \Sigma \etatil \geq  \Tr(\Sigma) + 2\sqrt{\Tr(\Sigma^2) t} + 2\norm{\Sigma}_2 t \right)  \leq e^{-t}.
\end{align*}
Since $\norm{\Sigma}_2 \leq \Tr(\Sigma)$ and $\Tr(\Sigma^2) \leq \norm{\Sigma}_2\Tr(\Sigma) \leq (\Tr(\Sigma))^2$, 
we obtain for any $t \geq 1$ the simplified bound
\begin{align} \label{eq:rand_quad_tail_bd}
    \prob\left(\etatil^\top \Sigma \etatil \geq 5t \Tr(\Sigma)\right) \leq e^{-t}.
\end{align}
Using the fact
\begin{align*} 
   \Tr(\Sigma) = \sum_{l=1}^m \Tr((\Atil^*_l)^\top \Atil^*_l) = \sum_{l=1}^m \sum_{t=0}^{T-1} \Tr(\Gamma_t(A^*_l)) 
\end{align*}
in \eqref{eq:rand_quad_tail_bd}, and choosing $t = \log(1/\delta)$ for $\delta \in (0,e^{-1})$, we obtain from \eqref{eq:vT_bd_quad} that 
\begin{align*}
    V_T \preceq 5 \left(\sum_{l=1}^m \sum_{t=0}^{T-1} \Tr(\Gamma_t(A^*_l)) \right) \log(1/\delta) I_{d^2}.
\end{align*}
Applying this in \eqref{eq:pitilbd_selfnorm} and choosing $\Vbar = m I_{d^2}$, we then obtain from \eqref{eq:ptilbd_temp2} (with some minor simplifications) the stated error bound on $\norm{\Pitil Q^\top \eta}_2 = \frac{1}{\sqrt{m}}\norm{S_T}_2$.

\subsection{Proof of Lemma \ref{lem:linf_bd_term2}} \label{subsec:proof_linf_bd_term2}
Recall the expression of $D^{\dagger}$ from Definition \ref{def:invfac_compat}, we then have 
\begin{equation*}
    (D^{\dagger})^{\top} \otimes I_{d^2} = 
      \begin{bmatrix}
s_1^\top \otimes I_{d^2} \\
\vdots\\
s_{\abs{\calE}}^\top \otimes I_{d^2}
\end{bmatrix}
\end{equation*}
which together with \eqref{eq:Qtop_eta} implies
\begin{align*}
    \left[(D^{\dagger})^{\top} \otimes I_{d^2}\right] Q^\top \eta 
    =  \begin{bmatrix}
\sum_{l=1}^m (s_1)_l (X_l \otimes I_{d}) \eta_l \\
\vdots\\
\sum_{l=1}^m (s_{\abs{\calE}})_l (X_l \otimes I_{d}) \eta_l 
\end{bmatrix}.
\end{align*}
Hence we obtain 
\begin{align*}
   \norm{\left[(D^{\dagger})^{\top} \otimes I_{d^2} \right] Q^\top \eta }_{\infty} = \max_{i \in [\abs{\calE}]} \norm{\sum_{l=1}^m (s_i)_l (X_l \otimes I_{d}) \eta_l}_{\infty}. 
\end{align*}
For a given $w = (w_1,\dots,w_m)^{\top} \in \matR^m$ with $\norm{w}_2 \leq \invfac$, we will now bound $\norm{\sum_{l=1}^m w_l (X_l \otimes I_{d^2}) \eta_l}_{\infty}$, and then take a union bound over $\set{s_1,\dots,s_{\abs{\calE}}}$ to conclude.

To this end, denoting $$v = \sum_{l=1}^m w_l (X_l \otimes I_{d^2}) \eta_l = \sum_{s=1}^T \sum_{l=1}^m w_l (x_{l,s} \otimes \eta_{l,s+1}) \in \matR^{d^2},$$
we can consider $v$ to be formed by column-stacking the vectors $v_1,v_2,\dots,v_{d} \in \matR^{d^2}$ where $$v_i = \sum_{s=1}^T \sum_{l=1}^m w_l (x_{l,s})_i  \ \eta_{l,s+1} \quad \text{for $i =1,\dots, d$}.$$  
For a given $i,j \in [d]$, we will now first bound $\abs{(v_i)_j}$ with high probability, and then taken a union bound to bound $\norm{v}_{\infty}$. To this end, we start by writing
\begin{equation*}
    (v_i)_j = \underbrace{\sum_{s=1}^T \sum_{l=1}^m w_l (x_{l,s})_i  \ (\eta_{l,s+1})_j}_{:= S_T} \quad \text{ and } \quad (\vbar_i)_j:=  \sum_{s=1}^T \sum_{l=1}^m  w_l^2 (x_{l,s})_i^2 \ \geq 0.
\end{equation*}
Similar to Section \ref{subsec:proof_l2_bd_term1}, we can see that  $(S_t)_{t=1}^T$ is a martingale. By writing 
\begin{equation} \label{eq:vij_abs_tmp1}
\abs{(v_i)_j} = ((\vbar_i)_j + a)^{1/2} \abs{((\vbar_i)_j + a)^{-1/2} (v_i)_j}
\end{equation}
for any fixed $a > 0$, and noting that $((\vbar_i)_j + a)^{-1/2} (v_i)_j$ is a (scalar-valued) self-normalized martingale, we can show in a completely analogous manner to the proof of 
\cite[Theorem 1]{abbasi11} (as explained in Section \ref{subsec:proof_l2_bd_term1}) that for any $\delta \in (0,1)$, 
%
\begin{align} \label{eq:vij_selfnorm}
    &\prob\left(\abs{((\vbar_i)_j + a)^{-1/2} (v_i)_j}^2 \leq 2\log\left(\frac{((\vbar_i)_j + a)^{1/2}}{\delta a^{1/2}} \right) \right)  \geq 1-\delta.
\end{align}
It remains to bound $(\vbar_i)_j$. Recall $\Atil_l^*$ from \eqref{eq:Atil_l_def}. It will be useful to denote $\Atil^*_l(s)$ to be the $s$'th ``row-block'' of matrices of $\Atil_l^*$, i.e.,
\begin{align*}
   \Atil^*_l(s) := [(A^*_l)^{s-1} \ \cdots \ A^*_l \ I_d \ 0 \cdots 0]. 
\end{align*}
Moreover, denote $\Atil^*_l(s,i) := e_i^\top \Atil^*_l(s)$ to be the $i$'th row vector of $\Atil^*_l(s)$, for $i \in [d]$. Then we can write $(x_{l,s})_i = \Atil^*_l(s,i) \etatil_l$; recall $\etatil_l$ and $\etatil$ from Section \ref{subsec:proof_l2_bd_term1}. Using the expression for $(\vbar_i)_j$, this implies
\begin{align*}
    (\vbar_i)_j 
    &= \sum_{s=1}^T \sum_{l=1}^m  w_l^2 (x_{l,s})_i^2  \\
    &= \sum_{s=1}^T \sum_{l=1}^m \etatil_l^\top (w_l^2 \Atil^*_l(s,i) ^\top \Atil^*_l(s,i) ) \etatil_l \\
    &= \etatil^\top \left[\underbrace{\blkdiag\left(w_1^2 \sum_{s=1}^T (\Atil^*_1(s,i))^\top \Atil^*_1(s,i), \dots, w_m^2 \sum_{s=1}^T (\Atil^*_m(s,i))^\top \Atil^*_m(s,i) \right)}_{=:\Sigma} \right] \etatil.
\end{align*}
Since $\norm{w}_2^2 \leq \invfac^2$ by assumption,
\begin{align*}
 \implies  
 \Tr(\Sigma) = \sum_{l=1}^m w_l^2 \sum_{s=1}^T \Tr\left(\Atil^*_l(s,i))^\top \Atil^*_l(s,i) \right)
 \leq \invfac^2 \max_{\substack{l \in [m] \\ i \in [d]}} e_i^\top \left(\sum_{s=1}^T \Gamma_{s-1}(A^*_l) \right) e_i.
\end{align*}
Then invoking the tail bound in \eqref{eq:rand_quad_tail_bd} with $t=\log(1/\delta)$ for $\delta \in (0,e^{-1})$, we obtain with probability at least $1-\delta$,
\begin{align} \label{eq:vbarij_temp1}
     (\vbar_i)_j \leq 5\invfac^2  \left[\max_{\substack{l \in [m] \\ i \in [d]}} e_i^\top \left(\sum_{s=1}^T \Gamma_{s-1}(A^*_l) \right) e_i \right] \log(1/\delta) =: \zeta_2'(m,T,\delta). 
\end{align}
Using \eqref{eq:vbarij_temp1} and \eqref{eq:vij_selfnorm}, and choosing $a = \zeta_2'(m,T,\delta)$, we have thus far shown the following (via a union bound over $[d^2]$). For any given $w\in \matR^m$ with $\norm{w}_2 \leq \invfac$, it holds with probability at least $1-2\delta$ that for all $i,j \in [d]$, 
\begin{align} \label{eq:vij_bds_temp2}
    \abs{((\vbar_i)_j + a)^{-1/2} (v_i)_j}^2 \leq 2\log\left(\frac{\sqrt{2}d^2}{\delta} \right) \ \text{ and } \ ((\vbar_i)_j + a)^{1/2} \leq \sqrt{2 \zeta_2'\Big(m,T,\frac{\delta}{d^2} \Big)}. 
\end{align}
Using \eqref{eq:vij_bds_temp2} and \eqref{eq:vij_abs_tmp1}, this implies that for any given $w\in \matR^m$ with $\norm{w}_2 \leq \invfac$,  it holds with probability at least $1-2\delta$ that
\begin{align*}
    \norm{\sum_{l=1}^m w_l (X_l \otimes I_{d^2}) \eta_l}_{\infty} \leq {2} \zeta_2'^{1/2}\Big(m,T,\frac{\delta}{d^2} \Big) \log^{1/2}(\sqrt{2}{d^2}/\delta).
\end{align*}
%
The statement of the lemma now follows by taking a union bound over $\set{s_1,\dots,s_{\abs{\calE}}}$.

%
%
\subsection{Proof of Lemma \ref{lem:RE_h_1}} \label{subsec:proof_RE_h1}
Denoting $\Qtil = Q (D^\dagger \otimes I_{d^2}) = Q \Dtil^\dagger$, we have $\norm{Q h_1}_2^2 = \norm{\Qtil \Dtil h}_2^2$. Recall from Definition \ref{def:invfac_compat} that $D^{\dagger} = [s_1 \ s_2 \  \cdots \ s_{\abs{\calE}}]$ where $s_i = (s_{i,1}, \dots,s_{i,m})^\top \in \matR^m$  for $i=1,\dots,\abs{\calE}$. Let us denote 
\begin{align*}
    (D^{\dagger})^\top = \left[s'_1 \ s'_2 \  \cdots \ s'_{m} \right] \in \matR^{\abs{\calE} \times m} \quad \text{ where } \quad s'_l  = (s'_{1,l}, \dots, s'_{\abs{\calE}, l})^\top.
\end{align*}
Then, it is not difficult to verify that
\begin{align*}
    \Qtil = \begin{bmatrix}
           {s'_1}^\top \otimes (X_1^\top \otimes I_d) \\
           {s'_2}^\top \otimes (X_2^\top \otimes I_d) \\
           \vdots \\
           {s'_m}^\top \otimes (X_m^\top \otimes I_d)
         \end{bmatrix}
\end{align*}
where we observe that $\Qtil$ consists of independent ``blocks'' of rows, with the entries within a block being dependent.

It will be useful to denote $u(h) = \Dtil h$, and the matrix version of $u(h)$ by $U(h)$, where $u(h) = \vect(U(h))$ and  
\begin{align*}
    U(h) \ = \ \left[U_1(h) \ U_2(h) \ \cdots \ U_{\abs{\calE}}(h) \right] \in \matR^{d \times (\abs{\calE} d)}.
\end{align*}
Then we can write $\norm{Q h_1}_2^2$ as 
\begin{align*}
   \norm{Q h_1}_2^2 = \norm{\Qtil u(h)}_2^2 
    &= \sum_{l=1}^m \norm{\left((s'_l \otimes X_l)^\top \otimes I_d \right) u(h)}_2^2 \\
    &= \sum_{l=1}^m \norm{U(h) (s_l' \otimes X_l)}_F^2 \\
    &=  \sum_{l=1}^m \norm{ \Big(\sum_{i=1}^{\abs{\calE}} s_{i,l} U_i(h) \Big)  X_l}_F^2 \\
    &= \sum_{l=1}^m \norm{\Big[\underbrace{I_T \otimes \Big(\sum_{i=1}^{\abs{\calE}} s_{i,l} U_i(h) \Big)}_{=: P_l(u(h))} \Big] \vect(X_l)}_2^2 
    = \sum_{l=1}^m \norm{P_l(u(h)) \vect(X_l)}_2^2.
\end{align*}
%
%
%
%
\rev{Now recall that 
$\vect(X_l) = \Atil_l^* \etatil_l$ where $\etatil_l, \etatil$ are as defined earlier in Section \ref{subsec:proof_l2_bd_term1}, and $\Atil_l^*$ is defined in \eqref{eq:Atil_l_def}}. This implies
\begin{align} \label{eq:quad_form_rel}
    \norm{Q h_1}_2^2 = \sum_{l=1}^m \norm{P_l(u(h)) \Atil_l^* \rev{\etatil_l}}_2^2 = \norm{P(u(h)) \rev{\etatil}}_2^2
\end{align}
where we denote
\begin{align} \label{eq:Dtilcone_to_Pmat}
    P(u(h)) := \text{blkdiag}\left[P_1(u(h)) \Atil_1^*, \cdots,P_m(u(h)) \Atil_m^* \right].
\end{align}

So our focus is now on lower bounding $\norm{P(u(h)) \rev{\etatil}}_2^2$ uniformly over $h \in \calC_{\calS} \cap \unitsph$. More specifically, denoting the set of matrices
\begin{align*}
    \calP = \set{P(u(h)): \ h \in \calC_{\calS} \cap \unitsph}
\end{align*}
we will control the quantity
\begin{align} \label{eq:chaos_trans_form}
    \sup_{h \in \calC_{\calS} \cap \unitsph} \abs{\norm{P(u(h)) \rev{\etatil}}_2^2 - \expec \norm{P(u(h)) \rev{\etatil}}_2^2} 
\end{align}
by invoking the concentration bound of \cite{krahmer14}, recalled as Theorem \ref{thm:krahmer_chaos} in Appendix \ref{appsec:sup_chaos_krahmer}. This theorem captures the complexity of the set $\calP$ via Talagrand's $\gamma_2$ functionals; their definition and properties are recalled in Appendix \ref{appsec:tal_prelim}. Thereafter, we will bound the $\gamma_2$ functionals of $\calP$ in terms of the $\gamma_2$ functionals of the set 
$$\Dtil (\calC_{\calS} \cap \unitsph) = \set{\Dtil h: h \in \calC_{\calS} \cap \unitsph}$$ 
using the Lipschitz property outlined in Appendix \ref{appsec:tal_prelim}. Finally, the $\gamma_2$ functionals of $\Dtil (\calC_{\calS} \cap \unitsph)$ will be bounded in terms of the Gaussian width of this set, namely $w(\Dtil (\calC_{\calS} \cap \unitsph))$, which will then be controlled \rev{in a standard manner}. 
This, along with some simplifications will conclude the proof.

With the above strategy in mind, let us first establish that 
\begin{align}\label{eq:lem1_expec_lowbd}
    \expec \norm{P(u(h)) \rev{\etatil}}_2^2 \geq T \norm{\Dtil^\dagger \Dtil h}_2^2 = T \norm{h_1}_2^2, \quad \forall \ h \in \calC_{\calS} \cap \unitsph.
\end{align}
Indeed, first note that
\begin{align*}
    \expec \norm{P(u(h)) \rev{\etatil}}_2^2 = \norm{P(u(h))}_F^2 = \sum_{l=1}^m \norm{P_l(u(h)) \Atil_l^*}_F^2.
\end{align*}
Recall the expressions for $P_l(u(h))$ and $\Atil_l^*$, we then see that 
\begin{align*}
    \norm{P_l(u(h)) \Atil_l^*}_F^2 \geq T \norm{\sum_{i=1}^{\abs{\calE}} s_{i,l} u_i(h)}_F^2 = T \norm{U(h) (s'_l \otimes I_d)}_F^2.
\end{align*}
This in turn implies
\begin{align}
    \sum_{l=1}^m \norm{P_l(u(h)) \Atil_l^*}_F^2 
    &\geq T \sum_{l=1}^m \norm{U(h) (s'_l \otimes I_d)}_F^2 \nonumber \\
    &= T \norm{U(h) [s'_1 \otimes I_d \ \cdots\cdots \ s'_m \otimes I_d ] }_F^2 \nonumber \\
    &= T \norm{U(h)((D^\dagger)^\top \otimes I_d)}_F^2 \nonumber \\
    &= T \norm{U(h) (D^\dagger \otimes I_d)^\top}_F^2 = T\norm{(D^\dagger \otimes I_{d^2}) u(h)}_2^2 = T \norm{\Dtil^\dagger \Dtil h}_2^2, \label{eq:tmp1_proof_h1_RE}
\end{align}
thus establishing \eqref{eq:lem1_expec_lowbd}. Observe that in the penultimate equality above, we used the fact $$\vect(U(h) (D^\dagger \otimes I_d)^\top) = ((D^\dagger \otimes I_d) \otimes I_d) \vect(U(h)) = (D^\dagger \otimes I_{d^2}) u(h).$$

\paragraph{Bounding $d_2(\calP)$ and $d_F(\calP)$.} To use Theorem \ref{thm:krahmer_chaos}, we need to bound the terms $d_2(\calP)$ and $d_F(\calP)$. To this end, we have for any $h \in  \calC_{\calS} \cap \unitsph$ that
\begin{align*}
    \norm{P(u(h))}_F^2 = \sum_{l=1}^m \norm{P_l(u(h)) \Atil_l^*}_F^2 \leq \beta^2 \sum_{l=1}^m \norm{P_l(u(h))}_F^2 \leq \beta^2 T, \\
\end{align*}
and also, 
%
 %
%
\begin{align*}
 \norm{P(u(h))}_2 
    = \max_{l=1,\dots,m} \norm{P_l(u(h)) \Atil_l^*}_2 
    \leq \beta \max_{l} \norm{P_l(u(h))}_2 
    &= \beta \max_{l \in [m]} \norm{U(h) (s'_l \otimes I_d)}_2 \\ 
    &\leq \beta \sqrt{\sum_{l=1}^m\norm{U(h) (s'_l \otimes I_d)}_2^2} \\
    &\leq \beta \sqrt{\sum_{l=1}^m\norm{U(h) (s'_l \otimes I_d)}_F^2} \\ 
    &= \beta \norm{U(h) ((D^\dagger)^\top \otimes I_d)}_F \\
    &= \beta \norm{\Dtil^\dagger \Dtil h}_2 \tag{as seen in \eqref{eq:tmp1_proof_h1_RE}} \\
    &\leq \beta.
\end{align*}
Hence we have shown that
\begin{equation} \label{eq:rad_bounds}
    d_2(\calP) \leq \beta \text{ and }
d_F(\calP) \leq \beta \sqrt{T}.
\end{equation}

\paragraph{Bounding $\gamma_2(\calP, \norm{\cdot}_2)$.} For any $h, h' \in \calC_{\calS} \cap \unitsph$ we have
\begin{align*}
    \norm{P(u(h)) - P(u(h'))}_2 &\leq \beta \max_l\norm{P_l(u(h)) - P_l(u(h'))}_2 \\
    &\leq \beta \max_l \norm{(U(h) - U(h')) (s'_l \otimes I_d)}_2 \\ 
    &\leq \beta \invfac' \norm{U(h) - U(h')}_2 \\
    &\leq \beta \invfac' \norm{U(h) - U(h')}_F \\
    &\leq \beta \invfac' \norm{u(h) - u(h')}_2.
\end{align*}
%
%
%
Since the mapping $\Dtil(\calC_{\calS} \cap \unitsph) \mapsto \calP$ via \eqref{eq:Dtilcone_to_Pmat} is onto, thus invoking the Lipschitz property of $\gamma_{\alpha}$ functionals (see Appendix \ref{appsec:tal_prelim}), we obtain in conjunction with Talagrand's majorizing measure theorem (see \eqref{eq:talag_maj_meas_thm}) that for some constants $c, c_1 > 0$,
\begin{align}
    \gamma_2(\calP, \norm{\cdot}_2) 
    \ &\leq \ c \beta \invfac'  \gamma_2\Big(\Dtil(\calC_{\calS} \cap \unitsph ), \norm{\cdot}_2 \Big) \ \leq \ c_1 \beta \invfac' w\Big(\Dtil(\calC_{\calS} \cap \unitsph )\Big). \label{eq:gamma_2_bd1_tmp} 
    %
\end{align}
\paragraph{Bounding the Gaussian width of $\Dtil(\calC_{\calS} \cap \unitsph)$.} We will show this using the following set inclusion. The proof is outlined in Appendix \ref{appsec:proof_lem_set_incl}. 
%
%
\begin{proposition} \label{prop:set_incl_basu}
    For $\rev{z} > 0$, and a positive integer $n,$ recall that $\mathbb{B}^{n}_p(\rev{z})$ denotes the $\ell_p$ ball of radius $\rev{z}$ in $\matR^{n}$. Also recall the  compatibility factor $\kappa_{\calS}$ from Definition \ref{def:invfac_compat}. Then, it holds that
    \begin{align*}
        \Dtil\Big(\calC_{\calS} \cap \unitsph \Big) 
        \subseteq 
        \mathbb{B}^{\abs{\calE}d^2}_1\Big(4\frac{\sqrt{\abs{\calS}}}{\kappa_{\calS}} + \rev{\frac{4}{\scale}}\norm{\Dtilacomp}_1 + 1 \Big). 
    \end{align*}
\end{proposition}
%
%
%
It then follows from Proposition 
\ref{prop:set_incl_basu} along with standard properties of the Gaussian width (see e.g., \cite[Proposition 7.5.2]{HDPbook}), that 
\begin{align*}
    w\left(
\Dtil \Big(\calC_{\calS} \cap \unitsph \Big)\right) 
    &\leq w\Big(\mathbb{B}^{\abs{\calE}d^2}_1\Big(4\frac{\sqrt{\abs{\calS}}}{\kappa_{\calS}} + \rev{\frac{4}{\scale}}\norm{\Dtilacomp}_1 + 1 \Big) \Big)  \\
    &\leq C \Big(4\frac{\sqrt{\abs{\calS}}}{\kappa_{\calS}} + \rev{\frac{4}{\scale}}\norm{\Dtilacomp}_1 + 1 \Big) \sqrt{\log(d^2\abs{\calE})}
   \tag{see \cite[Example 7.5.8]{HDPbook}}  
\end{align*}
for some constant $C \geq 1$. 
Using this together with \eqref{eq:gamma_2_bd1_tmp}, we have hence shown that
\begin{align}
    \gamma_2(\calP, \norm{\cdot}_2) 
     &\leq \ c_2 \beta \invfac' \Big(4\frac{\sqrt{\abs{\calS}}}{\kappa_{\calS}} + \rev{\frac{4}{\scale}}\norm{\Dtilacomp}_1 + 1 \Big) \sqrt{\log(d^2\abs{\calE})}, \label{eq:gamma_2_bd1} 
    %
\end{align}
for a constant $c_2 \geq 1$. 

\paragraph{Putting it together.} It will be useful to bound the terms $F,G$ and $H$ stated in Theorem \ref{thm:krahmer_chaos} using \eqref{eq:gamma_2_bd1} and \eqref{eq:rad_bounds}. To this end, we obtain for a suitably large constant $C_1 \geq 1$ that
\begin{align*}
H &\leq \beta^2  =: \Hbar, \\
G &\leq C_1 \beta^2 \left(\invfac'  \Big(4\frac{\sqrt{\abs{\calS}}}{\kappa_{\calS}} + \rev{\frac{4}{\scale}}\norm{\Dtilacomp}_1 + 1 \Big) \sqrt{\log(d^2\abs{\calE})} + \sqrt{T} \right) =: \Gbar, \\
F &\leq  C_1 \beta^2 \Bigg((\invfac')^2  \Big(4\frac{\sqrt{\abs{\calS}}}{\kappa_{\calS}} + \rev{\frac{4}{\scale}}\norm{\Dtilacomp}_1 + 1 \Big)^2 \log(d^2\abs{\calE}) \\
&+ 
\invfac' \sqrt{T}  \Big(4\frac{\sqrt{\abs{\calS}}}{\kappa_{\calS}} + \rev{\frac{4}{\scale}}\norm{\Dtilacomp}_1 + 1 \Big) \sqrt{\log(d^2\abs{\calE})} + \sqrt{T} \Bigg) =: \Fbar.
\end{align*}
Note that $\Fbar \geq \Gbar \geq \Hbar$. Now applying Theorem \ref{thm:krahmer_chaos} to the quantity \eqref{eq:chaos_trans_form}, and recalling the relation \eqref{eq:quad_form_rel} along with the bound in \eqref{eq:lem1_expec_lowbd}, we have thus far shown that there exist constants $c_1, c_2$ such that for any $t > 0$,  
\begin{align*}
    \prob &\left(\forall h \in \calC_{\calS} \cap \unitsph : \ \norm{Q h_1}_2^2 \geq T \norm{h_1}_2^2 - c_1 \Fbar  - t \right) 
    \rev{\geq 1 - } 2\exp\Big(-c_2 \min\set{\frac{t^2}{\Gbar^2}, \frac{t}{\Hbar}} \Big).
\end{align*}
We now set $t = \Gbar \sqrt{v}$ for any $v \geq 1$, $t = \Gbar \sqrt v$ for any $v \geq 1$, and simply note that 
\begin{align*}
    \min\set{\frac{t^2}{\Gbar^2}, \frac{t}{\Hbar}} \geq  \sqrt{v} 
\end{align*}
since $\Gbar \geq \Hbar$, as mentioned earlier. This leads to the statement of the lemma.
\subsection{Proof of Lemma \ref{lem:RE_h_2}} \label{subsec:proof_RE_h2}
We have
$$
h_{2}=\Pitil h=\left[\left(\frac{1}{m} \ones_{m} \ones_{m}^{\top}\right) \otimes I_{d^{2}}\right] h
$$
and $\left\|Q h_{2}\right\|_{2}^{2}=\|Q \Pitil h\|_{2}^{2}$.
Plugging $$Q= \operatorname{blkdiag} \left(\underbrace{X_{l}^{\top} }_{T \times d}\otimes I_{d}\right)_{l=1}^{m}\quad \text{and}
\quad \Pitil=\frac{1}{m}\left(\ones_{m} \otimes I_{d^{2}}\right)\left(\ones_{m} \otimes I_{d^{2}}\right)^{\top}
$$
we compute
$$
Q\Pitil h= \left(\frac{1}{\sqrt{m}} Q\left (\ones_{m}\otimes I_{d^{2}}\right )\right) \underbrace{\left(\frac{\ones_{m}^{\top}}{\sqrt{m}} \otimes I_{d^{2}}\right) h}_{=:u(h)}.
$$
Also note that
$$
\begin{aligned}
& h_{2}=\Pitil h=\left(\frac{\ones_{m}}{\sqrt{m}} \otimes I_{d^{2}}\right) \underbrace{u(h)}_{d^{2} \times 1}=\frac{1}{\sqrt{m}}\left[\begin{array}{c}
I_{d^{2}} \\
\vdots \\
I_{d^{2}}
\end{array}\right] u(h)=\frac{1}{\sqrt{m}}\left[\begin{array}{c}
u(h) \\
\vdots \\
u(h)
\end{array}\right]_{m d^{2} \times 1}
\end{aligned}
$$
which implies
$\left\|h_{2}\right\|_{2}^{2}=\|u(h)\|_{2}^{2}$.
So, we have
$$
\begin{aligned}
& Q \Pitil h=\left[\frac{1}{\sqrt{m}} Q\left(\ones_{m} \otimes I_{d^{2}}\right)\right] u(h) =\frac{1}{\sqrt{m}}\left[\begin{array}{c}
X_{1}^{\top} \otimes I_{d} \\
\vdots \\
X_{m}^{\top} \otimes I_{d}
\end{array}\right] u(h) \\
& \Rightarrow\|Q \Pitil h\|_{2}^{2}=\frac{1}{m} \sum_{l=1}^{m}\left\|\left(X_{l}^{\top} \otimes I_{d}\right) u(h)\right\|_{2}^{2}.
\end{aligned}
$$
Now, denote $U(h) \in \mathbb{R}^{d\times d }$ to be matrix version of $u(h) \in \mathbb{R}^{d^{2} \times 1}$, so that $u(h)=\operatorname{vec}(U(h))$.
Then we have,
$$
\begin{aligned}
\left\|\left(X_{l}^{\top} \otimes I_{d}\right) u(h)\right\|_{2}^{2} & =\left\|\underbrace{U(h)}_{d \times d} X_{l}\right\|_{F}^{2} =\left \|\left(I_{T} \otimes U(h)\right) \underbrace{\operatorname{vec}\left(X_{l}\right)}_{d T\times l}\right \|_{2}^{2} \\
& =\left\|\left(I_{T} \otimes U(h)\right) \Atil_{l}^{*} \rev{\etatil_{l}}\right\|_{2}^{2} .
\end{aligned}
$$
where we use that $\operatorname{vec}\left(X_{l}\right)=\Atil_{l}^{*} \rev{\etatil_{l}}$.
So, 
$$
\|Q \Pitil h\|_{2}^{2}=\frac{1}{m} \sum_{l=1}^{m}\left\|\left(I_{T} \otimes U(h)\right) \Atil_{l}^{*} \rev{\etatil_{l}}\right\|_{2}^{2}=\frac{1}{m}\left \Vert P(U(h)) \rev{\etatil}\right \Vert^{2}_{2}
$$
where $P(U(h)):=\operatorname{blkdiag}\left (\left(I_{T} \otimes U(h) \right)\Atil_{l}^{*}\right )_{l=1}^{m}$.
We define $$\calP:=\left\{P(U(h)): h \in \calC_{\calS} \cap \unitsph \right\}.$$
We want to control
$$
M:=\frac{1}{m} \sup _{h \in \calC_{\calS} \cap \unitsph}\abs{\|P(U(h)) \rev{\etatil}\|_{2}^{2}-\mathbb{E}\|P(U(h)) \rev{\etatil}\|_{2}^{2}}
$$
for which we will use Theorem \ref{thm:krahmer_chaos}.
\paragraph{Controlling $M$.}
We start by lower bounding $\mathbb{E}\|P(U(h)) \rev{\etatil}\|_{2}^{2}$. We compute
$$
\begin{aligned}
	\mathbb{E}\|P(U(h)) \rev{\etatil}\|_{2}^{2}&=\|P(U(h))\|_{F}^{2}
 =\sum_{l=1}^{m}\left\|\left(I_{T} \otimes U(h)\right) \Atil_{l}^{*}\right\|_{F}^{2} \\
& \geqslant Tm\|U(h)\|_{F}^{2} \quad\binom{\text { similar calculation }}{\text { as done for } h_{1} \text { term }} \\
& =Tm\|u(h)\|_{2}^{2}=Tm\left\|h_{2}\right\|_{2}^{2} .
\end{aligned}
$$
where we used $\|u(h)\|_{2}^{2}=\left\|h_{2}\right\|_{2}^{2}$ as shown earlier.
\begin{enumerate}
	\item [2.1] \emph{ Bounding $d_{2}(\calP)$ and $d_{F}(\calP)$.}
	For all $h \in \calC_{\calS} \cap \unitsph$ 
	$$
	\begin{aligned}
		\|P(U(h))\|_{F}^{2}&=\sum_{l=1}^{m}\|(I_{T} \otimes U(h)) \Atil_{l}^{*}\|_{F}^{2}
		 \leqslant \beta^{2} m\left\|I_{T} \otimes U(h)\right\|_{F}^{2} \\
		& =\beta^{2} m T\|U(h)\|_{F}^{2}  =\beta^{2} m T\left\|h_{2}\right\|_{2}^{2} \leqslant \beta^{2}m T
	\end{aligned}
	$$
On the other hand, 	
	$$
		\begin{aligned}
\|P(U(h))\|_{2}&=\max _{l}\left\|\left(I_{\top} \otimes U(h)\right) \Atil_{l}^{*}\right\|_{2}	\leqslant \beta\left\|I_{T} \otimes U(h)\right\|_{2}\\
&	=\beta\|U(h)\|_{2} \leqslant \beta\left\|h_{2}\right\|_{2} \leqslant \beta
\end{aligned}
	$$
		So, we get
        \begin{equation}\label{lm:RE_h1_eq3}
			d_{F}(\calP) \leqslant \beta \sqrt{m T},\quad d_{2}(\calP) \leqslant \beta.
		\end{equation}
		\item [2.2] \emph{Bounding $\gamma_{2}\left(\calP,\|\cdot\|_{2}\right)$}.
		We have that
		$$ 
		\begin{aligned}
		 \left\|P(U(h))-P\left(U\left(h^{\prime}\right)\right)\right\|_{2} 
			& =\max _{l}\left\|\left(\left[I_{T} \otimes U(h)\right]-\left[I_{T} \otimes U\left(h^{\prime}\right)\right]\right) \Atil_{l}^{*}\right\|_{2} \\
			& \leq \beta\left\|I_{T} \otimes\left(U(h)-U\left(h^{\prime}\right)\right)\right\|_{2} \\
			& =\beta\left\|U(h)-U\left(h^{\prime}\right)\right\|_{2} \\ 
            &\leq \beta\left\|U(h)-U\left(h^{\prime}\right)\right\|_{F} \\
			& =\beta\left\|u(h)-u\left(h^{\prime}\right)\right\|_{2}.
		\end{aligned}
		$$	
		Denote as before the set	
		$$
		\calU:=\left\{u(h)=\left(\frac{\ones_{m}^{\top}}{\sqrt{m}} \otimes I_{d^{2}}\right) h: h \in \calC_{\calS} \cap \unitsph\right\} \subset \mathbb{R}^{d^{2}}.
		$$
		Since the map from $\calU$ to $\calP$ is onto, we have that	
		$$
		\gamma_{2}\left(\calP,\|\cdot\|_{2}\right) \leqslant c \beta \gamma_{2}\left(\calU,\|\cdot\|_{2}\right).
		$$
			Now we will bound $\gamma_{2}\left(\calU,\|\cdot\|_{2}\right)$ by showing a set-inclusion type lemma as before. The difference is that now $\calU$ contains elements of the form	
		$$
		\left(\frac{\ones_{m}^{\top}}{\sqrt{m}} \otimes I_{d^{2}}\right) h.
		$$
		
		\textbf{Set inclusion}.
		Denote
		$$\calU^{\prime}:=\left\{u(h)=\left(\frac{\ones_{m}^{\top}}{\sqrt{m}} \otimes I_{d^{2}}\right) h: h \in \calC_{\calS} \cap B_{2}^{m d^{2}}(1)\right\} \subset \mathbb{R}^{d^{2}}.$$
		Then it's easy to see that $\calU \subset \calU^{\prime}$.
		Now for any $h \in \calC_{\calS} \cap B_{2}^{m d^{2}}(1)$, we have 
		\begin{equation}\label{lm:RE_h1_eq1}
				\begin{aligned}
					\|u(h)\|_{1}&=\left\|\left[\frac{\ones_{m}^{\top}}{\sqrt{m}} \otimes I_{d^{2}}\right] h\right\|_{1}
					=\frac{1}{\sqrt{m}}\left \|\left [\underbrace{I_{d^{2}} \ I_{d^{2}} \ldots . I_{d^{2}}}_{m \text { times }}\right ] h\right \|_{1}\\
					&\leqslant \frac{1}{\sqrt{m}}\|h\|_{1} \leqslant d\|h\|_{2} \leqslant d, 
				\end{aligned}	
		\end{equation}
	 and also
	 \begin{equation}\label{lm:RE_h1_eq2}
	 		\quad\|u(h)\|_{2}=\left\|\left[\frac{\ones_{m}^{\top}}{\sqrt{m}} \otimes I_{d^{2}}\right] h\right\|_{2}
	 		\leqslant \underbrace{\left \|\frac{\ones_{m}^{\top}}{\sqrt{m}} \otimes I_{d^{2}}\right \|_{2}}_{=1}
	 		\underbrace{\left \| h \right \|_2}_{\leqslant 1}
	 		\leqslant 1.
	 \end{equation}
	Together, \eqref{lm:RE_h1_eq1} and \eqref{lm:RE_h1_eq2} imply  $\calU^{\prime} \subseteq B_{1}^{d^{2}}(d) \cap B_{2}^{d^{2}}(1)$.	So, we have shown that
		$$\calU \subseteq \calU^{\prime} \subseteq B_{1}^{d^{2}}(d) \cap B_{2}^{d^{2}}(1).$$
		Now using \cite[Excercise 9.26]{HDPbook} we obtain
		$$
		\begin{aligned}
			w\left(B_{1}^{d^{2}}(d) \cap B_{2}^{d^{2}}(1)\right) \leqslant c \sqrt{d^{2} \log \left(e \frac{d^{2}}{d^{2}}\right)}  =c d
		\end{aligned}
		$$
	which implies that 	
		$w(\calU) \leqslant c d$
		and $\gamma_{2}\left(\calU,\|\cdot\|_{2}\right) \leqslant c^{\prime} d$
		for some constant $c^{\prime}>0$.
		Hence, for a  constant $c_{1} \geqslant 1$, we have	
		\begin{equation}\label{lm:RE_h1_eq4}
			\gamma_{2}\left(\calP,\|\cdot\|_{2}\right) \leqslant c_{1} \beta d.
		\end{equation}		
\end{enumerate}
Next we will bound the $F, G, H$ terms which we recall below.
$$
\begin{aligned}
	F&=\gamma_{2}\left(\mathcal P,\|\cdot\|_{2}\right)\left[\gamma_{2}\left(\mathcal P,\|\cdot\|_{2}\right)+d_{F}(\mathcal P)\right] 
	+d_{F}(\mathcal P) d_{2}(\mathcal P), \\
	G&=d_{2}(\mathcal P)\left[\gamma_{2}\left(\mathcal P,\|\cdot\|_{2}\right)+d_{F}(\mathcal P)\right], \\
	H&=d_{2}^{2}(\mathcal P).
\end{aligned}
$$
Using \eqref{lm:RE_h1_eq3} - \eqref{lm:RE_h1_eq4} we obtain the following bounds:
 $$
\begin{aligned}
	H &\leqslant \beta^{2}=: \bar{H}\\
 G &\leqslant \beta\left[c_{1} \beta d+\beta \sqrt{m T}\right] \leqslant c_{1} \beta^{2}[\sqrt{m T}+d]=: \bar{G}\\
F &\leqslant c_{1} \beta d \left[c_{1} \beta d+\beta \sqrt{m T}\right]+\beta^{2} \sqrt{m T}
\leqslant c_{1}^{\prime} \beta^{2} d[d+\sqrt{m T}]+  \beta^{2}\sqrt{m T}
\\
& \leqslant c_{2}^{\prime} \beta^{2} d(d+\sqrt{m T}) =c_{2} \bar{G} d=\bar{F}
\end{aligned}
$$
where 
$c_{1}, c_{2} \geqslant 1$. So we have $\bar{F} \geqslant \bar{G} \geqslant \bar{H}$.
Now, Theorem \ref{thm:krahmer_chaos} implies that there exist  $c_{1}, c_{2}>0$ s.t for any $t>0$, 
$$
\begin{aligned}
\mathbb{P}\Big(\forall h \in \calC_{\calS} \cap \unitsph &: \ \|P(U(h)) \rev{\etatil}\|_{F}^{2} \geq \mathbb{E}\|P(U(h)) \rev{\etatil}\|_{F}^{2}-c_{1} \bar{F}-t\Big) \\  &\rev{\geqslant 1 -} 2 \exp \left(-c_{2} \min \left\{\frac{t^{2}}{\bar{G}^{2}}, \frac{t}{\bar{H}}\right\}\right).
\end{aligned}
$$
Now using $
\|Q \Pitil h\|_{2}^{2}=\frac{1}{m}\left \Vert P(U(h))\rev{\etatil}\right \Vert^{2}_{2}
$ and $	\mathbb{E}\|P(U(h)) \rev{\etatil}\|_{2}^{2} \geqslant Tm\left\|h_{2}\right\|_{2}^{2}$ we obtain 
$$
\begin{aligned}
	\mathbb{P}\left(\forall \ h \in \calC_{\calS} \cap \unitsph: \ 
\left\|Q h_{2}\right\|_{2}^{2}  \geqslant T\left\|h_{2}\right\|_{2}^{2}-\frac{c_{1} \bar{G} d}{m}-\frac{t}{m} \right )
 \rev{\geqslant 1 -} 2 \exp \left(-c_{2} \min \left\{\frac{t^{2}}{\bar{G}^{2}}, \frac{t}{\bar{H}}\right\}\right).
\end{aligned}
$$
Set $t=\bar{G} \rev{\sqrt{v}}$ for $v \geqslant 1$, then
$$
\min \left\{\frac{t^{2}}{\bar{G}^{2}}, \frac{t}{\bar{H}}\right\} \rev{\geqslant} \sqrt{v} \quad(\text { Since } \bar{G} \geqslant \bar{H})
$$
which implies the statement of the lemma.

 \subsection{Proof of Lemma \ref{lem:cross-term}} \label{subsec:proof_RE_h1_h_2}
We start by writing 
$$h_{1}=\Dtil^\dagger \Dtil h  =\Dtil^\dagger(\Dtil h)_{\calS}+\Dtil^\dagger(\Dtil h)_{\calS^{c}},$$ 
and 
$$
\begin{aligned}
	h_{2}=\widetilde{\Pi} h =\left(\frac{\ones_m \ones_m^{\top}}{m} \otimes I_{d^{2}}\right) h 
 &=\left(\frac{\ones_{m}}{\sqrt{m}} \otimes I_{d^{2}}\right) \underbrace{\left(\frac{\ones_{m}^{\top}}{\sqrt{m}} \otimes I_{d^{2}}\right) h}_{\widetilde{h}} \\
& =\left(\frac{\ones_{m}}{\sqrt{m}} \otimes I_{d^{2}}\right) \widetilde{h} =U\widetilde{h}
\end{aligned}
$$
where $U:= \frac{\ones_{m}}{\sqrt{m}} \otimes I_{d^{2}}$ has orthonormal columns.
	\paragraph{1.  A bookkeeping inequality.}
		Starting from
        $$
\left\langle Q h_{1} , Q h_{2}\right\rangle  =h_{1}^{\top} Q^{\top} Q h_{2}$$
		and substituting $h_{1},h_{2}$ we obtain
\begin{equation*}
\begin{aligned}
 \left|\left\langle Q h_{1}, Q h_{2}\right\rangle\right| 
&=  \left|\left[\Dtil^\dagger(\Dtil h)_{\calS}+\Dtil^\dagger(\Dtil h)_{\calS^{c}}\right]^{\top} Q^{\top} QU \widetilde{h}\right| \\
\leqslant & \left\|\left[(\Dtil h)_{\calS}^{\top}\left(\Dtil^\dagger\right)^{\top}+(\Dtil h)_{\calS^{c}}^{\top}\left(\Dtil^\dagger\right)^{\top}\right] Q^{\top} Q
U\right\|_{2}\left\|h_{2}\right\|_{2} \quad \left (\text{using}\quad \left\|\widetilde{h}\right\|_{2}=\left\| h_{2} \right\|_{2}\right )\\
\leqslant & \left(\left\|(\Dtil h)_{\calS}^{\top}\left(\Dtil^\dagger\right)^{\top} Q^{\top} QU\right\|_{2}+\left \| \left(\Dtil h\right)_{\calS^{c}}^{\top}\left(\Dtil^\dagger\right)^{\top} Q^{\top} QU\right \|_{2} \right) \| h_{2}\|_2 .
\end{aligned}
\end{equation*}


The first term is bounded as follows. 
$$
\begin{aligned}
&\left\|(\Dtil h)_{\calS}^{\top}\left(\Dtil^\dagger\right)^{\top} Q^{\top} QU\right\|_{2} =\left\|U^\top Q^{\top} Q\left(\Dtil^\dagger\right)(\Dtil h)_{\calS}\right\|_{2} \\
& \leqslant\left\|(\Dtil h)_{\calS}\right\|_{1}\left \| \underbrace{U^\top}_{d^{2} \times m d^{2}} \underbrace{Q^{\top} Q}_{md^{2} \times m d^{2}} \underbrace{\left(\Dtil^\dagger\right)}_{m d^{2} \times|\mathcal E| d^{2}}\right\|_{1 \rightarrow 2}  =\left\|(\Dtil h)_{\calS}\right\|_{1} \max _{j \in\left[\vert\mathcal E\vert d^{2} \right]}\left\|U^\top Q^{\top} Q\left(\Dtil^\dagger\right) e_{j}\right\|_{2}
\end{aligned}
$$
where  we recall the fact 
$$
\|X\|_{1 \rightarrow 2}=\max _{j}\left\|X_{: j}\right\|_{2}.
$$

For the second term we have
 $$
 \begin{aligned}
\left  \| \left (\Dtil h \right )_{\calS^{c} }^{\top}\left(\Dtil^{\dagger}\right)^{\top} Q^{\top} QU\right  \|_{2} 
&\leqslant\left\|(\Dtil h)_{S^{c}}\right\|_{1} \max _{j \in\left[\left|\mathcal E\right| d^{2}\right]}\left\|U^\top Q^{\top} Q \left(\Dtil^\dagger\right) e_{j}\right\|_{2}\\
&\leqslant\left(1+3\left\|\left(\Dtil h\right)_{\calS}  \right\|_{1} + \rev{\frac{4}{\scale}}\norm{\Dtilacomp}_1\right)\max _{j \in\left[\vert \mathcal E \vert d^{2}\right]} \left \|U^\top Q^{\top} Q\left(\Dtil^\dagger\right ) e_j \right \|_{2}
\end{aligned}
$$
where we used $h \in \calC_{\calS} \cap \unitsph$. Putting the two bounds together we get
\begin{equation}\label{lm_cross_term_eq12}
 \left|\left\langle Q h_{1}, Q h_{2}\right\rangle\right| \leqslant  
\left(1+4\left\|(\Dtil h)_{S}\right\|_{1} + \rev{\frac{4}{\scale}}\norm{\Dtilacomp}_1 \right) \max _{j \in\left[\vert \mathcal E \vert d^{2}\right]}\left\|U^\top Q^{\top} Q b_j\right\|_{2}\left\|h_{2}\right\|_{2},
\end{equation}
where $b_{j}:=\Dtil^{\dagger}e_{j}$ denotes the $j$‑th column of $\Dtil^{\dagger}$.

	\smallskip
		\noindent\textit{Interpretation.}  
		Equation \eqref{lm_cross_term_eq12} shows that the entire problem
		boils down to a single quantity
		\[
		M\;:=\;\max_{j\in[|\calE|d^{2}]}
		\norm{U^{\!\top}Q^{\!\top}Q\,b_{j}}_2,
		\]
		plus the harmless $\norm{h_{2}}_{2}\le1$.
\paragraph{2.  Controlling $M$.}
{Let $u_i$ be the $i$‑th column of \rev{$\frac{\ones_{m}}{\sqrt{m}} \otimes I_{d^{2}}$  so that $U = \frac{1}{\sqrt{m}}[u_1  \cdots u_{d^2}]$}. Then
\begin{equation}\label{lem_cross_term_eq1}
    M
= \frac1{\sqrt m}\max_{j}\Bigl(\sum_{i=1}^{d^2}(u_i^\top Q^\top Q\,b_j)^2\Bigr)^{\!1/2}
\ \le\ \frac{d}{\sqrt m}\max_{i\in[d^2],\,j\in[|\mathcal E|d^2]}\bigl|u_i^\top Q^\top Q\,b_j\bigr|.
\end{equation}
}
{
\paragraph{Block/Kronecker expansion.}
Partition $u_i=(u_{i,1}^\top \ldots u_{i,m}^\top)^\top$ and $b_j=(b_{j,1}^\top \ldots b_{j,m}^\top)^\top$ with
$u_{i,l},b_{j,l}\in\mathbb R^{d^2}$ and let $U_{i,l},B_{j,l}\in\mathbb R^{d\times d}$ be their
matrix “unvec” versions, i.e., $\mathrm{vec}(U_{i,l})=u_{i,l}$ and $\mathrm{vec}(B_{j,l})=b_{j,l}$.
Using $(X\otimes I)\mathrm{vec}(B)=\mathrm{vec}(BX^\top)$ and
$\langle A,C\rangle=\mathrm{vec}(A)^\top\mathrm{vec}(C)$,
\begin{equation}\label{lem_cross_term_eq2}
u_i^\top(Q^\top Q)b_j
=\sum_{l=1}^m u_{i,l}^\top\!\bigl((X_l X_l^\top)\otimes I_d\bigr)b_{j,l}
=\sum_{l=1}^m \big\langle B_{j,l}^{\!\top}U_{i,l},\,X_l X_l^\top\big\rangle.
\end{equation}
}
{
\paragraph{Controllability expansion.}
With the block‑lower‑triangular matrix $\Atil^*_l$ from \eqref{eq:Atil_l_def} and its $t$‑th block row
$\Atil^*_l(t)$, one has
\[
\expec[X_l X_l^\top] 
= \sum_{t=1}^{T} \tilde A_l^{*}(t)\,\tilde A_l^{*}(t)^\top
= \sum_{t=1}^{T}\ \sum_{s=0}^{t-1} (A_l^*)^s\bigl((A_l^*)^s\bigr)^\top
=:\ G_l,
\]
where $G_l:=\sum_{t=1}^{T}\Gamma_{t-1}(A_l^*)$ and
$\Gamma_{t-1}(A)=\sum_{s=0}^{t-1}A^s(A^s)^\top$. Thus \eqref{lem_cross_term_eq2} becomes
\begin{equation}\label{lem_cross_term_eq3}
\mathbb E\bigl[ u_i^\top Q^\top Q\,b_j\bigr]
= \sum_{l=1}^m \big\langle Z_l^\top,\,G_l\big\rangle,
\qquad
Z_l:=U_{i,l}^{\!\top}B_{j,l}.
\end{equation}
}
{
\paragraph{Step 1 (deterministic offset).}
Recall $D^{\dagger}
	=\bigl[s_1\;\;s_2\;\;\dots\;\;s_{\vert \calE\vert}\bigr]$. Notice from the structure of $U_{i,l}, B_{j,l}$ that for any $i\in[d^2],\,j\in[|\mathcal E|d^2]$, there \rev{are two possibilities. Either $Z_l = 0$ for all $l \in [m]$ -- this implies $u_i^\top Q^\top Q\,b_j = 0$. Or, there} exists a corresponding $k \in [\abs{\calE}]$ and $a,b \in [d]$ \rev{($k, a, b$ depending only on $i,j$)} such that
$Z_l = s_{k,l} e_a e_b^\top$, where \rev{$e_a, e_b$ are} canonical vectors of $\matR^d$. Then writing $\bar G:=\frac{1}{m}\sum_{l=1}^m G_l$,
\[
\mathbb E\bigl[u_i^\top Q^\top Q\,b_j\bigr]
=\sum_{l=1}^m s_{k,l}\,(G_l)_{b,a}
=\sum_{l=1}^m s_{k,l}\bigl((G_l)_{b,a}-\bar G_{b,a}\bigr)
\le \|s_k\|_2\,
\sqrt{\sum_{l=1}^m\bigl((G_l)_{b,a}-\bar G_{b,a}\bigr)^2}
\]
where we used $\ones^\top_{m}s_k=\sum_{l=1}^m s_{k,l}=0$ for any $k\in [\vert\mathcal E \vert]$ since $s_k\in \operatorname{span} (D^{\dagger})$ and $\operatorname{span} (D^{\dagger}) \perp\operatorname{span} (\ones_{m})$.
Using the definition of $\mu$ we \rev{obtain for all $i\in[d^2],\,j\in[|\mathcal E|d^2]$ that}
\begin{equation}\label{lem_cross_term_eq4}
\boxed{\;
\mathbb E\bigl[u_i^\top Q^\top Q\,b_j\bigr]
\ \le\
\mu\, \rev{\max_{a,b\in[d]}} \sqrt{\sum_{l=1}^m\bigl((G_l)_{b,a}-\bar G_{b,a}\bigr)^2}\ .
\;}
\end{equation}
}
{
\paragraph{Step 2 (Hanson–Wright fluctuation).}
Recall the definition of $\tilde\eta_l:=(\eta_{l,1}^\top \ldots \eta_{l,T}^\top)^\top \in \mathbb R^{dT}$ and let
\[
W_l=\tilde A_l^{*\,\top}\,(I_T\otimes Z_l)\,\tilde A_l^{*},
\qquad
W=\mathrm{blkdiag}(W_1,\ldots,W_m).
\]
Then $u_i^\top Q^\top Q\,b_j=\tilde\eta^\top W\tilde\eta$ with
$\tilde\eta:=(\tilde\eta_1^\top \ldots \tilde\eta_m^\top)^\top$, and the Hanson–Wright inequality \rev{(see eg., \cite{rudelson2013})} yields, for universal $c>0$,
\[
\mathbb P\!\left(\left|\tilde\eta^\top W\tilde\eta-\mathbb E[\tilde\eta^\top W\tilde\eta]\right|\ge t\right)
\le 2\exp\!\Big(-c\min\{t^2/\|W\|_F^2,\ t/\|W\|_2\}\Big).
\]
Using $\|A^\top B A\|_F\le\|A\|_2^2\|B\|_F$ and $\|A^\top B A\|_2\le\|A\|_2^2\|B\|_2$ with
$\beta:=\max_l\|\Atil^*_l\|_2$ (see \eqref{eq:Atil_l_def}), and noting
$\|I_T\otimes Z_l\|_F=\sqrt{T}\,\|Z_l\|_F \  \rev{\leq} \ \sqrt{T}\,|s_{k,l}|$ and
$\|I_T\otimes Z_l\|_2=\|Z_l\|_2 \ \rev{\leq} \ |s_{k,l}|$, one gets
\[
\|W\|_F^2=\sum_{l=1}^m\|W_l\|_F^2\ \le\ \beta^4\,T\,\sum_{l=1}^m s_{k,l}^2
\ \le\ \beta^4 T\,\mu^2,\qquad
\|W\|_2\ \le\ \beta^2\,\mu.
\]
Hence, for $\delta \in (0,c_2)$ for a suitably small constant $c_2 < 1$, it holds with probability at least $1-\delta$ that 
\begin{equation}\label{lem_cross_term_eq5}
\boxed{\;
\bigl|u_i^\top Q^\top Q\,b_j-\mathbb E[u_i^\top Q^\top Q\,b_j]\bigr|
\ \leq c_1 \sqrt{T} \beta^2\mu \log(\tfrac{1}{\delta})}, 
\end{equation}
for a suitably large constant $c_1 > 1$.
}
{
\paragraph{Putting \eqref{lem_cross_term_eq4} - \eqref{lem_cross_term_eq5} into \(M\).}
Apply \eqref{lem_cross_term_eq5} with a union bound over $i\in[d^2]$ and $j\in[|\mathcal E|d^2]$ and plug into \eqref{lem_cross_term_eq1}. \rev{Using \eqref{lem_cross_term_eq4}}, we then have that there exist absolute constants $c_1 > 1$ and $c_2 \in (0,1)$ such that, for every $\delta\in(0,c_2)$,
\begin{equation}\label{lem_cross_term_eq6}
\boxed{\;
\mathbb P\Bigg[
M\ \le\ \frac{d\,\mu}{\sqrt m}\,
\underbrace{\max_{a,b\in[d]}\Big(\sum_{l=1}^m\bigl((G_l)_{b,a}-\bar G_{b,a}\bigr)^2\Big)^{\!1/2}}_{=: \Delta_G}
\ +\ c_1\, \rev{d} \beta^2\mu\,\sqrt{\frac{T}{m}}\ \log\!\Big(\frac{|\mathcal E|\,d}{\delta}\Big)\Bigg]\ \ge\ 1-\delta\ .
\;}
\end{equation}
}
\paragraph{3.  Putting pieces together.}
{
Starting from the bookkeeping inequality \eqref{lm_cross_term_eq12},
\[
\bigl|\langle Qh_1,\,Qh_2\rangle\bigr|
\;\le\;
M\, \Bigl(1+4\|(\Dtil h)_S\|_1+\rev{\frac{4}{\scale}}\|(\Dtil a^{*})_{S^c}\|_1\Bigr)\,\|h_2\|_2,
\]
use $\|(\Dtil h)_S\|_1\le \frac{\sqrt{|S|}}{\kappa_S}\|h\|_2\le \frac{\sqrt{|S|}}{\kappa_S}$ and $\|h_2\|_2\le 1$.
Inserting the bound for $M$ from \eqref{lem_cross_term_eq6},
we obtain that, with probability at least $1-\delta$, simultaneously for all 
$h\in C_S\cap \mathbb{S}^{md^2-1}$,
\begin{equation}\label{eq:cross-term-final}
\bigl|\langle Qh_1,\,Qh_2\rangle\bigr|
\;\le\;
c_3\!\frac{\mu}{\sqrt{m}}\left(1+\frac{\sqrt{|S|}}{\kappa_S}+\rev{\frac{1}{\scale}}\bigl\|(\Dtil a^{*})_{S^c}\bigr\|_1\right)
\!\left[
d\,\Delta_G
\;+\;
\beta^2\, \rev{d} \sqrt{T}\log\!\Bigl(\frac{|\mathcal E|\,d}{\delta}\Bigr)
\right].
\end{equation}
This is exactly the claim in Lemma~\ref{lem:cross-term}. 
\smallskip}

\section{Experiments}
\label{sec:experiments}


In this section, we validate the performance of the graph total variation regularizer for estimating the LDS matrices $A^*_{l}$ on a series of synthetic experiments\footnote{The code  can be found at: \url{https://github.com/donnate/LDS}.}. To this end, we consider linear dynamical systems on a graph on $m$ nodes where each node $l\in\{1,\dots,m\}$ is associated with a time series that evolves as
\begin{equation}
x_{l,t+1} \;=\; A_l^* x_{l,t} + \eta_{l, t+1}, 
\qquad \eta_{l,t}\overset{\text{i.i.d}}{\sim}\mathcal N(0, I_d),
\label{eq:lds}
\end{equation}
with (unless stated otherwise) $d=2$. The system matrices $(A^*_{l})_{l=1}^m$ vary across nodes but are structured by the underlying graph. We compare (a) \textit{our graph-TV LDS}; (b) an \textit{individual OLS (OLS\textsubscript{ind}) estimator,}  where a separate LDS is learnt per node; \textit{ (c) the pooled OLS (OLS\textsubscript{pooled}),} where we learn a single common $A$  across all nodes; \textit{(d) a Laplacian LDS} (\cite{tyagi2024jointlearn}), similar to our graph-TV estimator but that uses an $\ell_2^2$ penalty of the form $\| \tilde{D}a\|_2^2$  on the differences across edges rather than a TV penalty; and \textit{(e) a group-lasso estimator}, that groups coefficients together (across the graph) and imposes a lasso penalty. This lasso penalty allows us to compare graph penalties against another type of structural prior (here, group sparsity).

\subsection{Experimental Protocol}
\paragraph{Graphs and ground truths.} To evaluate the importance of the graph topology in the final performance of the graph-TV regularizer, we consider various types of graph topologies: the chain graph, the 2d grid, the star graph, as well as an Erdős-Rényi random graph, parameterized by the connection probability between any pair of nodes, with varying levels of sparsity. 

Each node $l$ is assigned a matrix of the form
    \begin{equation}
A^*_{l} \;=\;
\begin{bmatrix}
\beta^*_l & 0.1 \\
0 & 0.6 
\end{bmatrix} \in \mathbb{R}^{2 \times 2},
\end{equation} where $(\beta^*_l)_{l =1}^m$ are chosen in one of the following ways.
\begin{description}
    \item[(i) Piecewise constant.] In this case, nodes are assigned to one of three groups, and
    \begin{equation}\label{eq:pw}
      \beta^*_l \in \{ s, 0, -s\},  
    \end{equation} 
      with $s$ a scaling parameter that controls the maximal jump size: $s = \| \tilde{D} a^*\|_{\infty}$. 
      
    \item[(ii) Smoothly varying.] In this case, we compute the 2D spectral embedding for each node in the graph and define the signal as (for $\omega$ a parameter)
\begin{equation}\label{eq:smooth}
    \beta^*_l = s\!\left(\cos\!\left(\omega t_1\right)
    \times \cos\!\left(\omega\, t_2\right)\right),
    \quad
    \text{where } t_1 = l \bmod \sqrt{m}, \;
    t_2 = \lfloor l / \sqrt{m} \rfloor.
\end{equation}
\end{description}

In each case, the signal is scaled to achieve either the same $\ell_{\infty}$ norm in the signal difference  across edges (i.e., $\|\tilde{D}  a^*\|_{\infty}  = \|D  \beta^*\|_{\infty}, \text{ with } \beta^* \in \mathbb{R}^m$ is made similar across signals). This ensures appropriate comparison across all estimators. Figure~\ref{fig:example1} shows examples of the piecewise constant $\beta^*_l$ generated on different graph topologies, while Figure~\ref{fig:example2} shows examples of its smooth counterpart.

\begin{figure}
    \centering
    \includegraphics[width=0.7\linewidth]{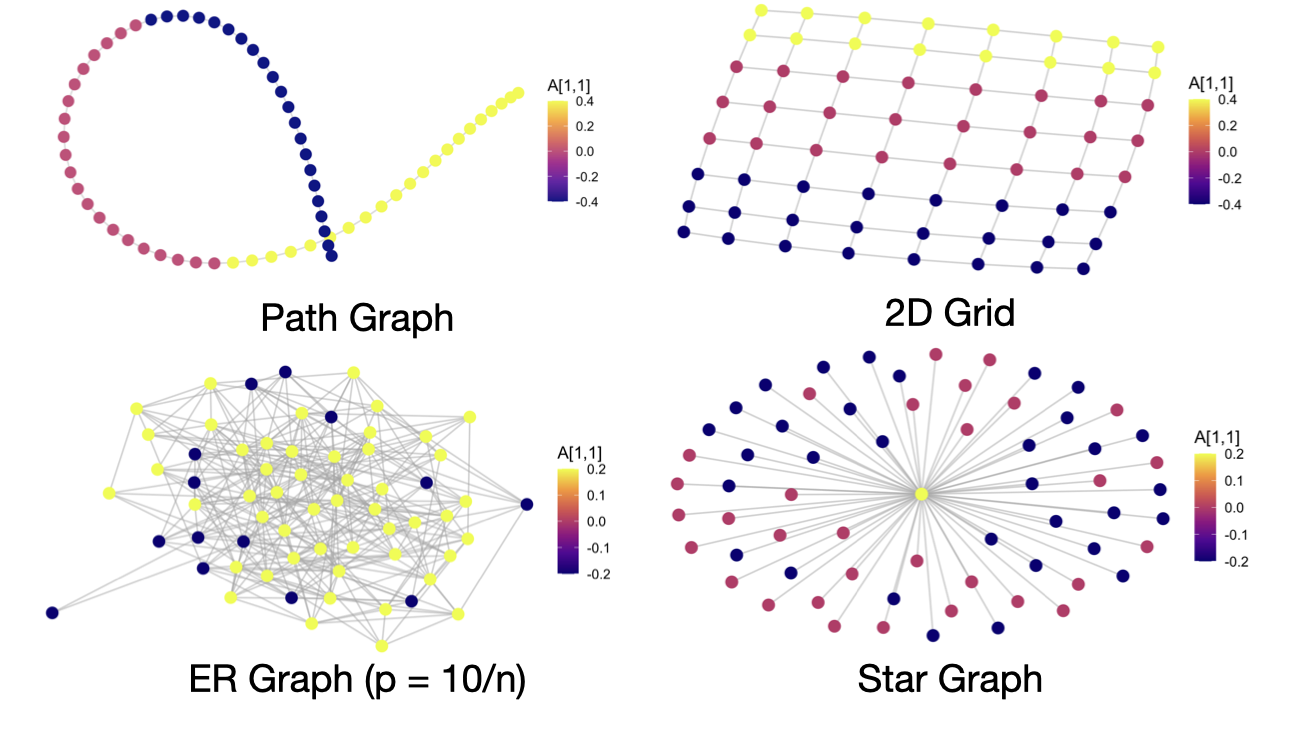}
    \caption{Examples of coefficients $\beta^*_l$ generated, overlaid on the graph. Here, the coefficients $(\beta^*_l)_{l=1}^m$ are chosen to be piecewise constant across the graph, and the maximal jump size is constrained to $\| \Dtil a^*\|_{\infty} = 0.4$.}
    \label{fig:example1}
\end{figure}

\begin{figure}
    \centering
    \includegraphics[width=0.7\linewidth]{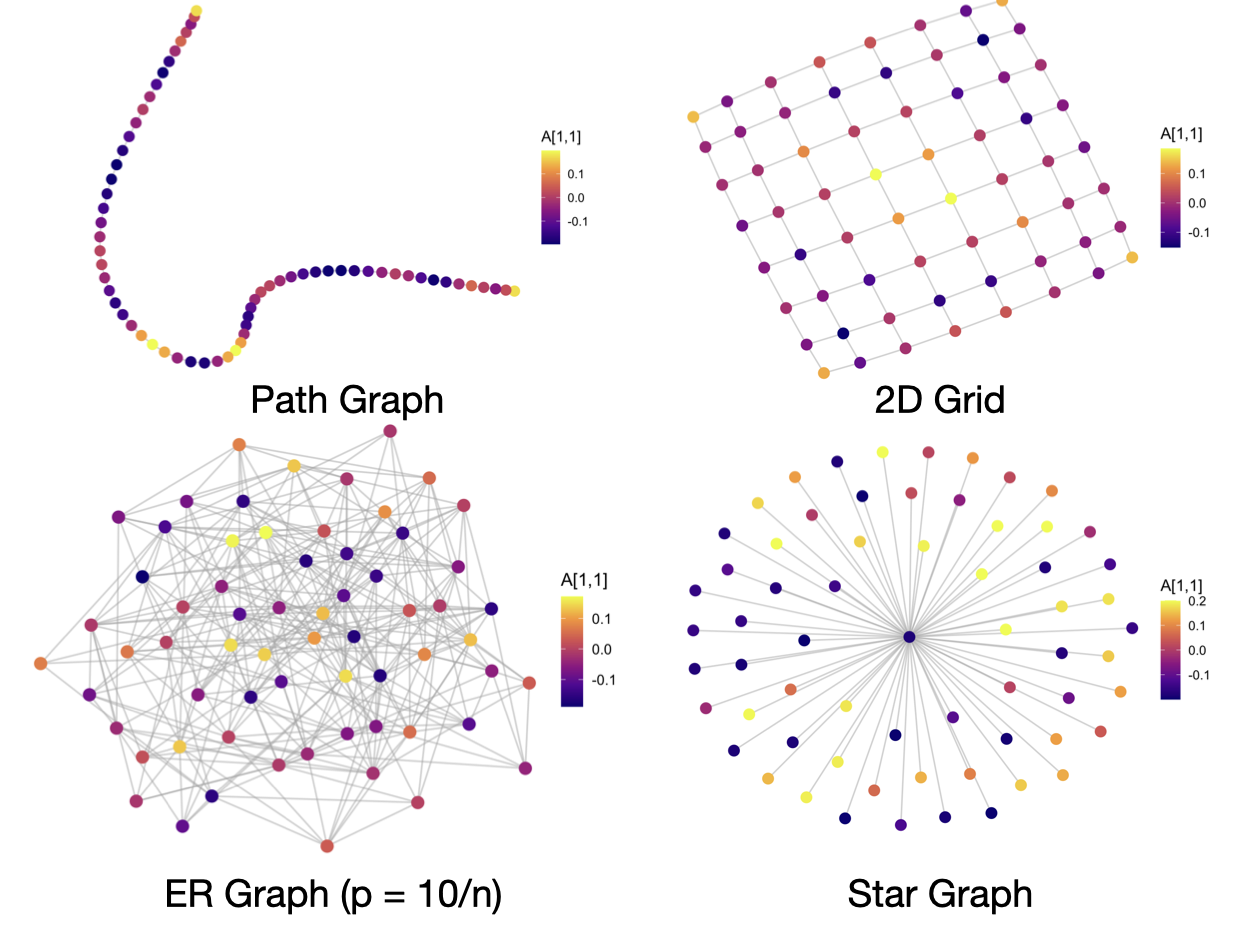}
    \caption{Examples of coefficients $\beta_l$ generated, overlaid on the graph. Here, the coefficients $(\beta^*_l)_{l=1}^m$ are  chosen to be smoothly varying across the graph, as per ~\eqref{eq:smooth}.}
    \label{fig:example2}
\end{figure}

\paragraph{Estimators and parameter selection.}
We vectorize each $A^*_{l}\in\mathbb{R}^{d\times d}$ and stack them in a single vector $a^*=\big(\operatorname{vec}(A^*_{1})^\top,\dots,\operatorname{vec}(A^*_{m})^\top\big)^\top\in\mathbb{R}^{m d^2}$. Let $T_{\text{train}}$ denote the length of the time series used for training. For node $l$, with time series $X_{l}\in\mathbb{R}^{d\times T_{\text{train}}}$ and response $\Xtil_l \in\mathbb{R}^{d\times T_{\text{train}}} $, the local design is $Q_l=(X_{l}^{\top}\otimes I_d)\in\mathbb{R}^{(T_{\text{train}}d)\times d^2}$. We form $Q=\operatorname{blkdiag}(Q_1,\dots,Q_m)$ and $\tilde{x}=\vect(\Xtil_l) \in \mathbb{R}^{d T_{\text{train}}}.$ 

\paragraph{TV-joint LDS.}
We solve the generalized lasso (as per equation~\ref{eq:tv_pen_ls_vecform})
\begin{equation}
\min_{a\in\mathbb{R}^{m d^2}}\;\frac{1}{2m}\,\big\|Qa-\xtil\big\|_2^2
\;+\;\lambda\,\big\|\,\Dtil\,a\,\big\|_1,
\label{eq:tv-lds}
\end{equation}
using a full regularization path (up to 200 steps). The regularization parameter $\lambda$ is selected by minimizing the validation MSE (defined below). 

\paragraph{Data generation, splits, and metrics}
Given a graph and a collection of LDS matrices $(A^*_l)_{l=1}^m$, we simulate $T=T_{\text{train}}+ 2b + T_{\text{val}}+T_{\text{test}}$ steps per node  from~\eqref{eq:lds}, with $b$ as a buffer parameter (here, chosen to be 100) that ensures that the validation and testing sets are (approximately) independent of one another and of the training set. We use the first $T_{\text{train}}$ steps for training, the next $T_{\text{val}}$ steps for validation, and the last $T_{\text{test}}$ for testing. Unless stated otherwise, we set $T_{\text{val}}=4$, $T_{\text{test}}=8$, and repeat over seeds $r=1,\dots,n_{\text{rep}}$.

We report the parameter and test prediction MSEs, defined as
\begin{align}
\text{(Parameter MSE)}&&
MSE
&=
\frac{1}{m}\sum_{l=1}^m\big\|\,\est{A}_{l}-A^*_{l}\,\big\|_F^2,
\\
\text{(Test prediction MSE)}&&
MSE_{\text{pred}}
&=
\frac{1}{m\,T_{\text{test}}}
\sum_{l=1}^m\big\|\,\est{A}_{l}X_{\text{test}}^{(l)}-\Xtil_{\text{test}}^{(l)}\,\big\|_F^2,
\end{align}
where $X_{\text{test}}^{(l)}, \Xtil_{\text{test}}^{(l)} \in \matR^{d \times T_{\text{test}}}$. 
Curves show means and 95\% normal confidence intervals (mean $\pm$ $1.96\times$ standard error) over seeds.

\subsection{Results}
\label{subsec:results}

\paragraph{(a) Error vs topology.} We begin evaluating the error as a  function of the graph topology. To this end, we fix the number of nodes to $64$ and generate different graph topologies and corresponding piecewise constant signals. Figure~\ref{fig:linf} shows the error of the different estimators as a function of the maximal jump size $\|\Dtil a^*\|_{\infty}.$ Note that, since $\beta^*_l$ is the only varying coefficient in the matrix $A^*_l$, $\| \Dtil a^*\|_{\infty} = \| D \beta^*\|_{\infty}$. To emphasize the link with models~\ref{eq:pw} and \ref{eq:smooth}, we will  use the latter notation to denote the gap size. As expected, our estimator significantly outperforms the naive estimator $\hat{\beta}_{OLS_{\text{ind}}}$, and improves upon the pooled estimator $\hat{\beta}_{OLS_{\text{pool}}}$, particularly in regimes where the jump size is large. In particular, for the path graph, the TV  $(\ell_1$) denoiser substantially improves upon the Laplacian  $(\ell_2^2$)  regularizer.

\begin{figure}[h]
    \centering
\includegraphics[width=\linewidth]{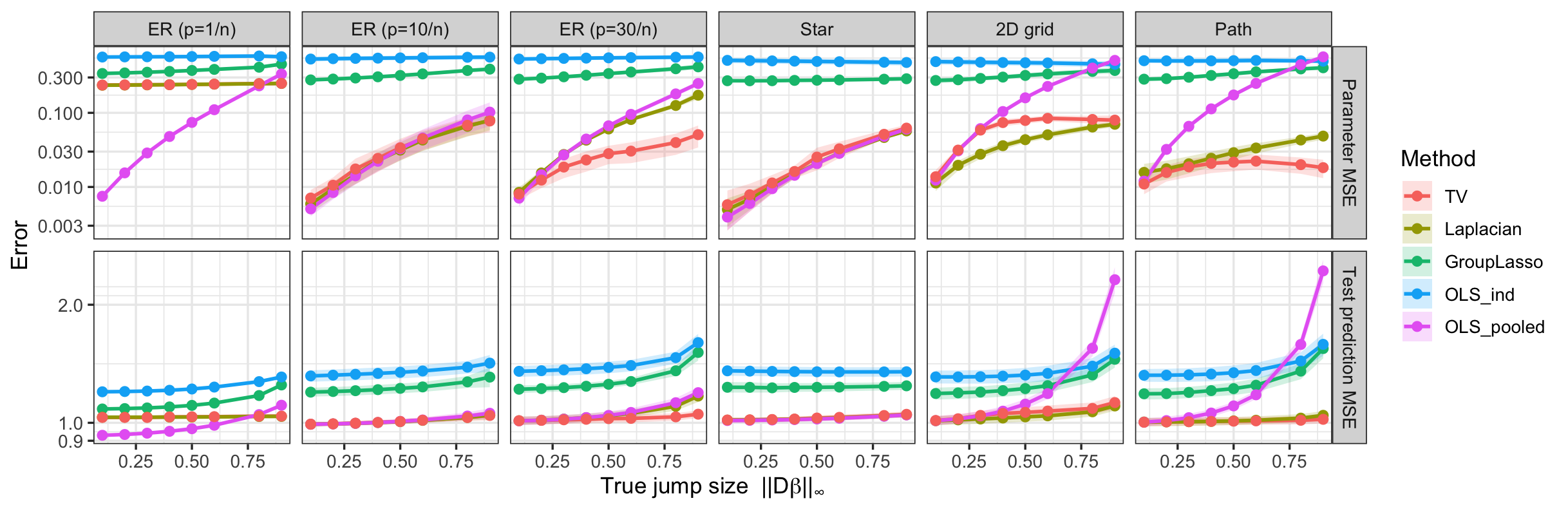}
    \caption{Performance of the TV regularizer and other competing methods as a function of the maximal jump size $\|D \beta^*\|_{\infty}$ on a piecewise constant  signal (as per \eqref{eq:pw}).}
    \label{fig:linf}
\end{figure}

\begin{figure}[h]
    \centering
\includegraphics[width=\linewidth]{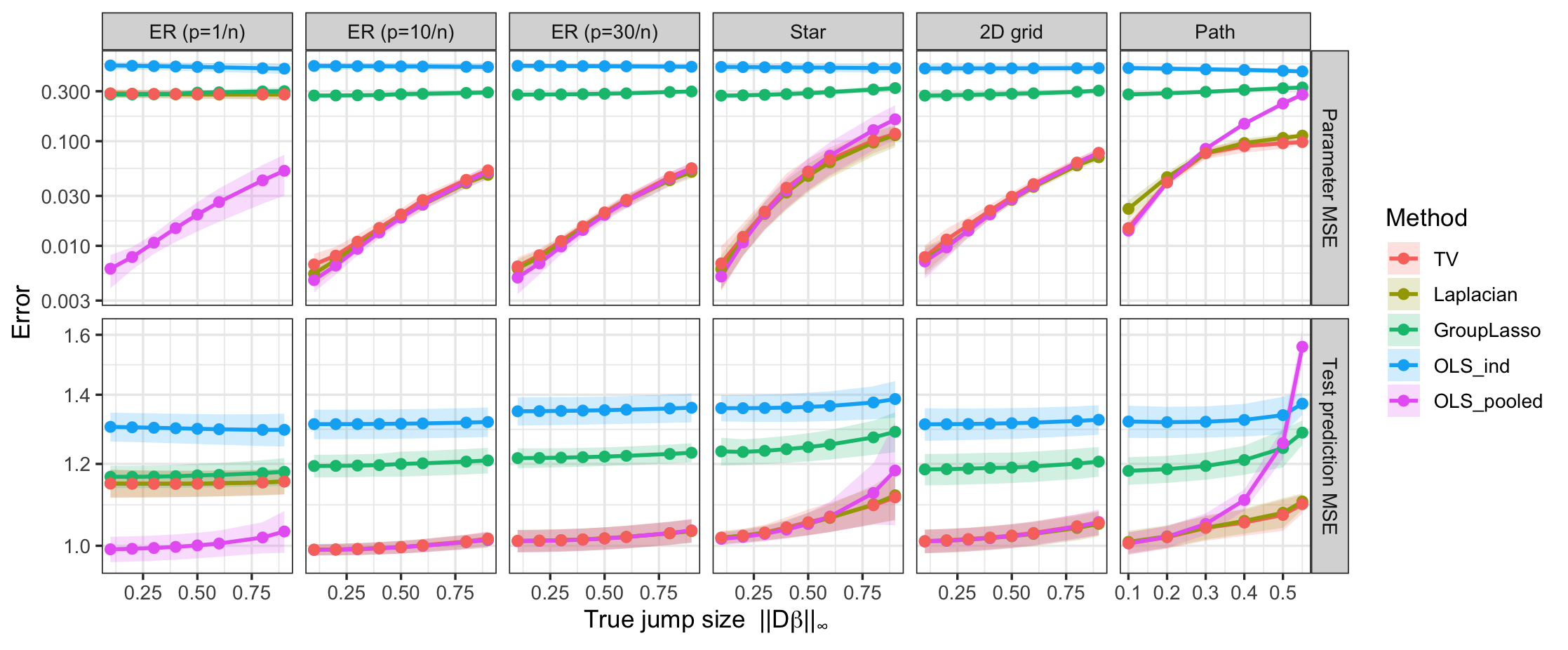}
    \caption{Performance of the TV regularizer as a function of the maximal jump size $\|D \beta^*\|_{\infty}$ on smooth  signal (defined as per ~\eqref{eq:smooth}).}
    \label{fig:linf_smooth}
\end{figure}

\paragraph{(b) Error vs.\ training length $T_{\text{train}}$.}
We also vary $T_{\text{train}}\in\{4,8,16,32,64,128,256\}$ (log-scale $x$-axis). 
We expect TV to clearly outperform OLS\textsubscript{ind} at small $T_{\text{train}}$ by borrowing strength across nodes, while approaching OLS\textsubscript{ind} as $T_{\text{train}}$ grows; OLS\textsubscript{pooled} is expected to perform well when $T_{\text{train}}$ is extremely small, as the benefits of the smoothing induced by fitting the estimator over the entire graph are expected to outweigh the bias incurred in this case. The results for the different topologies and signals are reported in Figures~\ref{fig:function_of_time_pw} and \ref{fig:function_of_time_smooth}, and confirm our expected findings. In particular, for the path graph and dense Erdős–Rényi graph, the TV regularizer consistently outperforms all other estimators, including the Laplacian estimator. We also note that for the piecewise-constant signal, the graph TV estimator offers significant advantages in terms of estimation MSE over the Laplacian estimator (e.g. in the path and ER cases). Even in the 2D grid, as $T_{\text{train}}$ increases, we see that the TV regularizer improves upon the Laplacian. This is expected, as the TV regularizer's ability to preserve sharp contrasts is advantageous for signals defined by a few sparse, high-magnitude differences rather than small, distributed ones.
\begin{figure}[H]
  \centering
\begin{subfigure}[t]{\textwidth}
        \includegraphics[width=.95\linewidth]{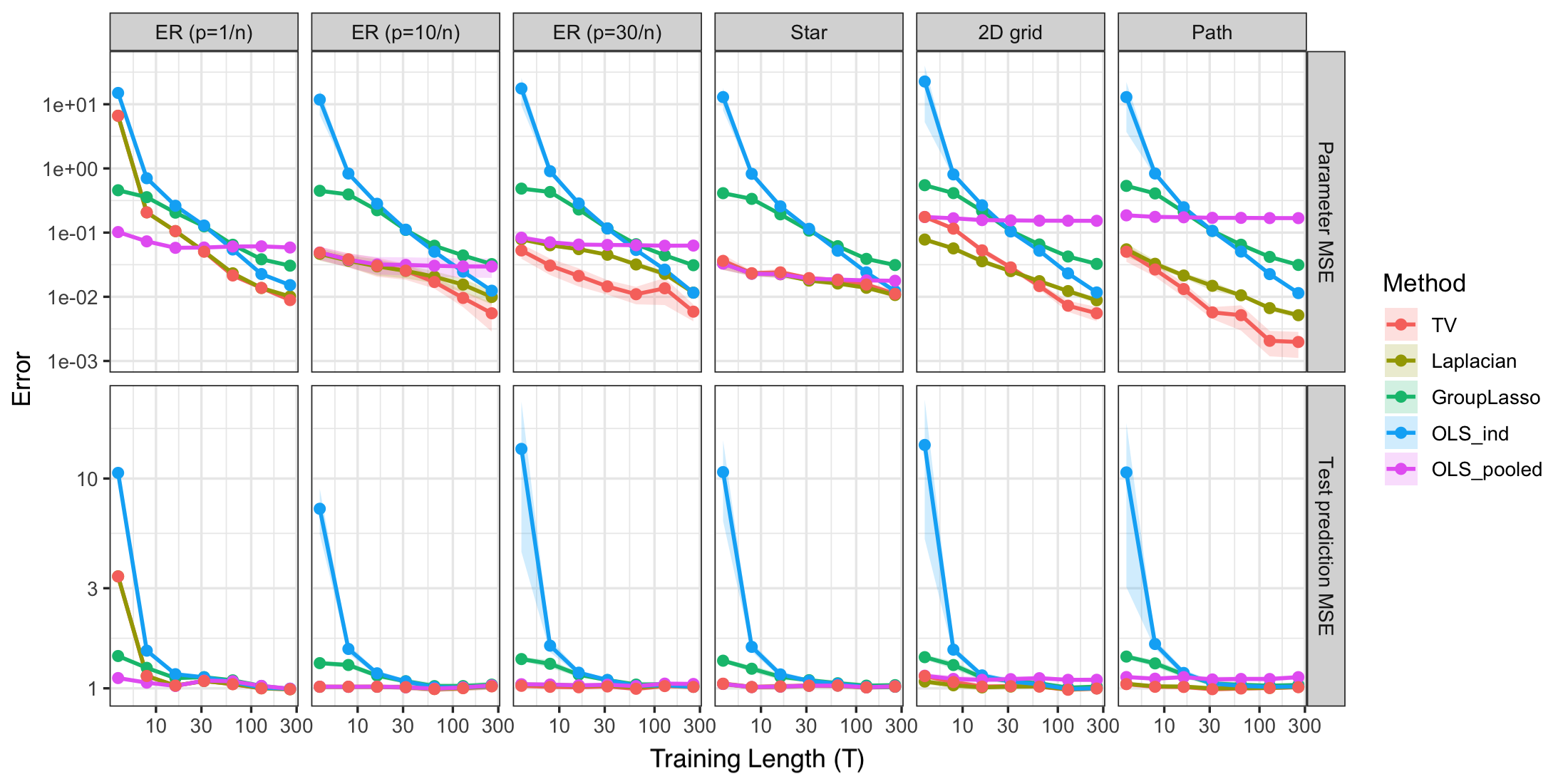}
  \caption{Parameter and test prediction MSE as a function of training time length $T_{\text{train}}$ and graph topology. The signal generated in these plots is piecewise constant (as per Equation~\ref{eq:pw}), with $\|D\beta^*\|_{\infty} =0.5$.  Lines and dots indicate MSE mean averaged over 15 trials $\pm$ 95\% CI.}
  \label{fig:function_of_time_pw}
  \end{subfigure}
\begin{subfigure}[t]{\textwidth}
        \includegraphics[width=.95\linewidth]{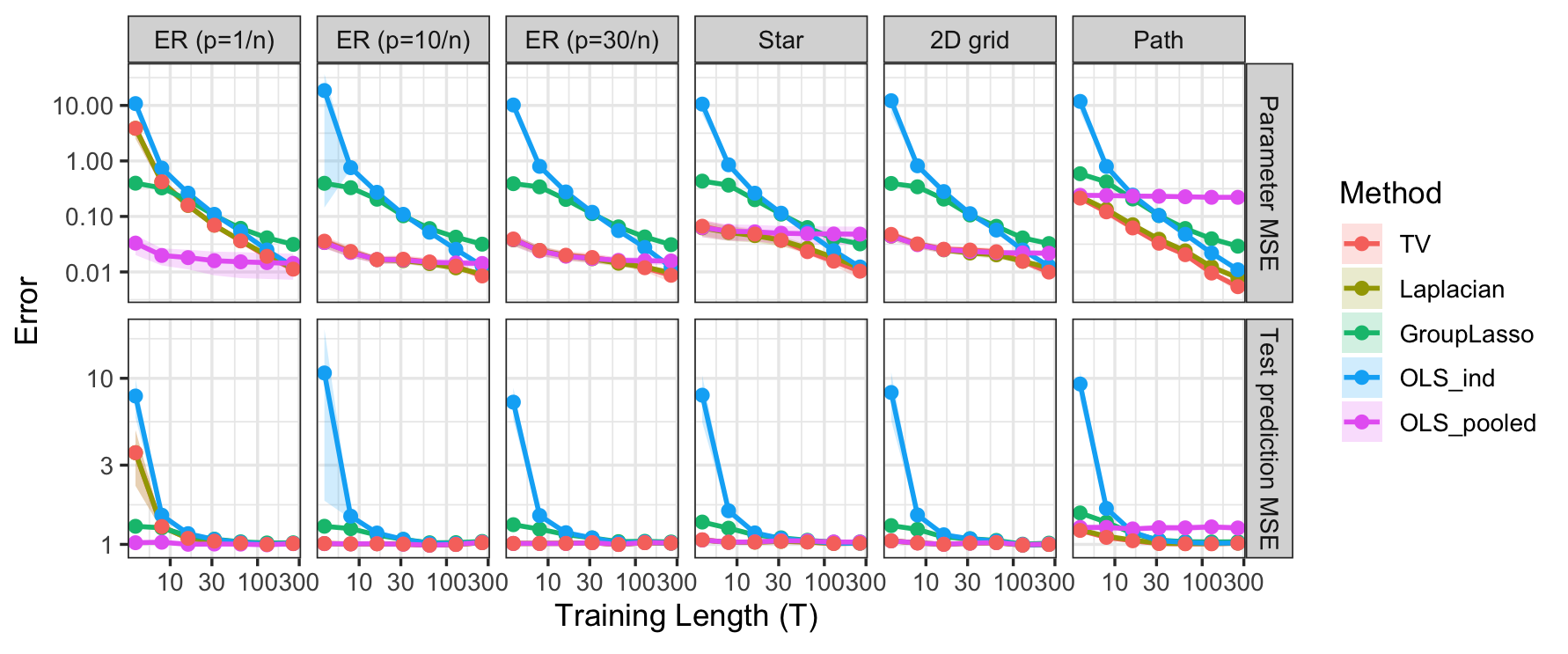}
  \caption{Parameter and test prediction MSE as a function of training time length $T_{\text{train}}$ and graph topology. The signal generated here was smooth (as per Equation~\ref{eq:smooth}), with $\|D\beta^*\|_{\infty} =0.5$. Lines and dots indicate MSE mean averaged over 15 trials $\pm$ 95\% CI.}
  \label{fig:function_of_time_smooth}
  \end{subfigure}
  \caption{Parameter and test prediction MSE as a function of training time length $T_{\text{train}}$ for different graph topologies. }
\end{figure}
\paragraph{(c) Error vs.\ number of nodes $m$.} We hold $T_{\text{train}}$ fixed (here, $T_{\text{train}}=10$) and increase the graph size ($m$). The results are shown in Figure~\ref{fig:error-vs-m-grid} for a piecewise constant signal. We note that our
TV estimator improves with larger $m$. By contrast, OLS\textsubscript{ind} does not benefit from added nodes, while especially in the path or grid graph, the pooled OLS (OLS\textsubscript{pooled}) oversmoothes the signal, inducing a substantial parameter MSE compared to both graph-regularized penalties.

\begin{figure}[H]
  \centering
\includegraphics[width=.95\linewidth]{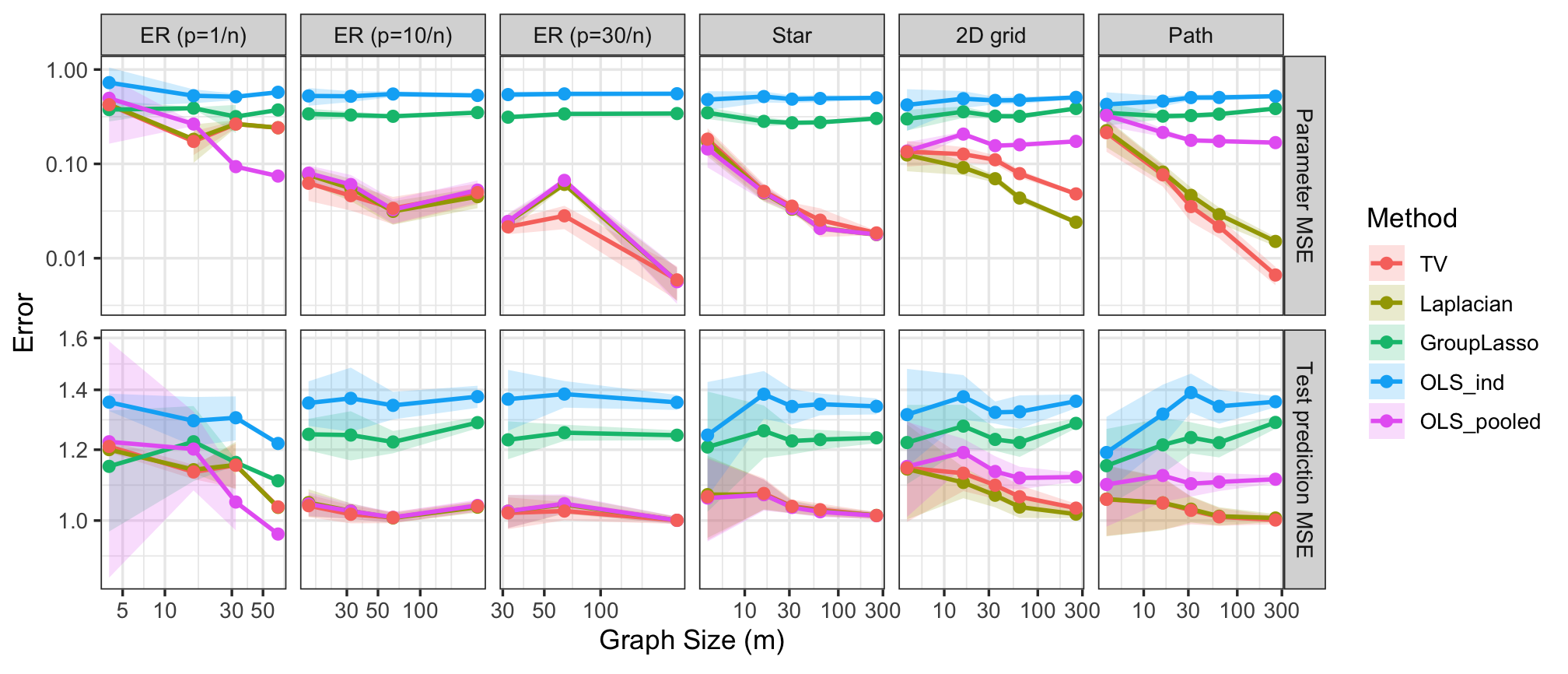}
  \caption{Parameter and test prediction MSE versus $m$ for a piecewise constant signal, for different graph topologies. Lines and dots indicate MSE mean averaged over 15 trials $\pm$ 95\% CI.}
  \label{fig:error-vs-m-grid}
\end{figure}


\subsection{Real Data Experiments: EPA data}

We evaluate the proposed estimator on data from the U.S. Environmental Protection Agency (EPA) national air-quality monitoring network. Specifically, we analyze daily arithmetic-mean concentrations of nitric oxide (NO), particulate matter with aerodynamic diameter $\leq 10 \mu m$ (PM\textsubscript{10}), and daily mean temperature from 2010-2024.\footnote{Data are available from the EPA: \url{https://www.epa.gov/outdoor-air-quality-data/download-daily-data}.}
These measurements are obtained from 398 monitoring stations across the United States. We process the data by retaining measurements from primary instruments with at least 75\% temporal coverage, removing observations flagged as special events, and applying variable-specific transformations to approximate normality (none for temperature, $\log$ for PM\textsubscript{10}, and a log–log transform $\log(\log(x+1))$ for NO). Stations are further filtered so that, within each year, fewer than 5\% of observations are missing; this reduces the number of stations by roughly 45\% on average (e.g., 221 stations in 2020; see Figure~\ref{fig:map}). Missing values are imputed via last observation carried forward. Using the reported geographic coordinates, we construct a spatial graph by connecting each station to its five nearest neighbors; the resulting graph is shown in Figure~\ref{fig:map}, with colors indicating NO on April 9, 2024 (day 100).

\begin{figure}[!ht]
    \centering
    \includegraphics[width=\linewidth]{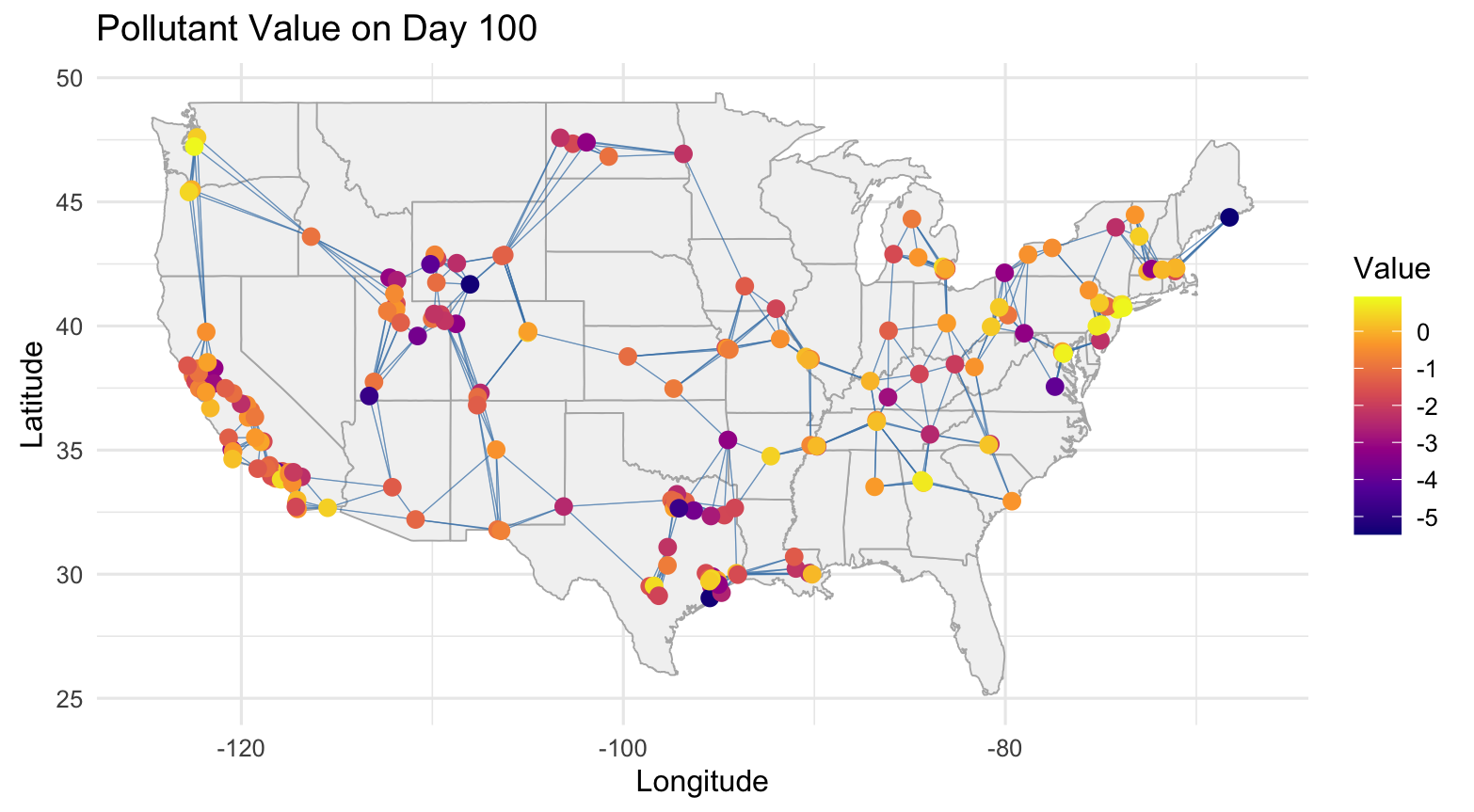}
    \caption{Air quality stations, along with the induced kNN graph ($k=5).$}
    \label{fig:map}
\end{figure}

We compare our method for fitting a $d=2$-LDS to our 3 EPA datasets to the benchmarks described in the previous section. We also compare the results to a persistence baseline that, for each forecast horizon, predicts using the value observed at time $t=0$. For each year, we draw a start day $s$ uniformly from the first half of the year and define the training window as days $s$ to $s+T_{\text{train}}$. To ensure independence across splits, we insert 20-day buffer windows between train, validation, and test sets. Validation error is computed at 10 evenly spaced evaluation points over the 100-day validation window, and test error at 10 evenly spaced points over the remaining test window. This procedure is repeated six times per year; results are first averaged within year, then aggregated across 15 years. 
A visualization of the coefficients obtained by the different methods is provided in Figure~\ref{fig:map_comp_coef}. 

\paragraph{Results.} Figure~\ref{fig:map_comp_coef} illustrates the structural differences and implicit regularization of each method. With $T_{\text{train}}=15$ (a data-limited regime), the station-wise OLS fits are highly heterogeneous, as evidenced by the wide spread of coefficients. By contrast, pooled OLS yields a single global coefficient shared across stations. Both the graph total-variation (TV) denoiser and the Laplacian regularizer produce coefficients that vary smoothly over the spatial graph. As expected, TV recovers piecewise-constant (or piecewise-smooth) structures with sharp transitions aligned to graph boundaries, whereas Laplacian regularization yields a smoother, more diffuse spatial distribution.

To quantify performance, we evaluate two criteria: (a) the stability of the estimated coefficients across repeated fits; and (b) the prediction error. In the absence of ground truth parameters, coefficient stability serves as a proxy for estimation accuracy: a suitable regularizer should yield reproducible coefficients and support more reliable inferences. Figure~\ref{fig:results_h1} highlights the results for prediction on the next day (horizon $=1$) for the NO dataset, while Figure \ref{fig:results_h1_pm} shows the same for the PM10 dataset. 
%
%
%
Table~\ref{tab:tab:airquality} shows the average test RMSE across all 3 datasets, where $T_{\text{train}}=10$, with the same 1-day horizon.  Consistent with our simulations, we note that the prediction MSE for the TV estimator matches that of the best competing method. However, we note substantial differences in coefficient stabilities across datasets. In particular, we note that the Laplacian estimator tends to produce more variable coefficients across runs. In contrast, the TV estimator is relatively more consistent across fittings.
Figures~\ref{fig:results_h1} and \ref{fig:results_h1_pm} also highlight the performance of the method as a function of the training length and the number of neighbors $k$.  In particular, we observe that as $k$ and $T_{\text{train}}$ increase, the difference in coefficient stability between the Laplacian and TV regularizer increases.
\begin{figure}[h]
  \centering

  \begin{subfigure}{\textwidth}
    \centering
    \includegraphics[width=0.72\linewidth]{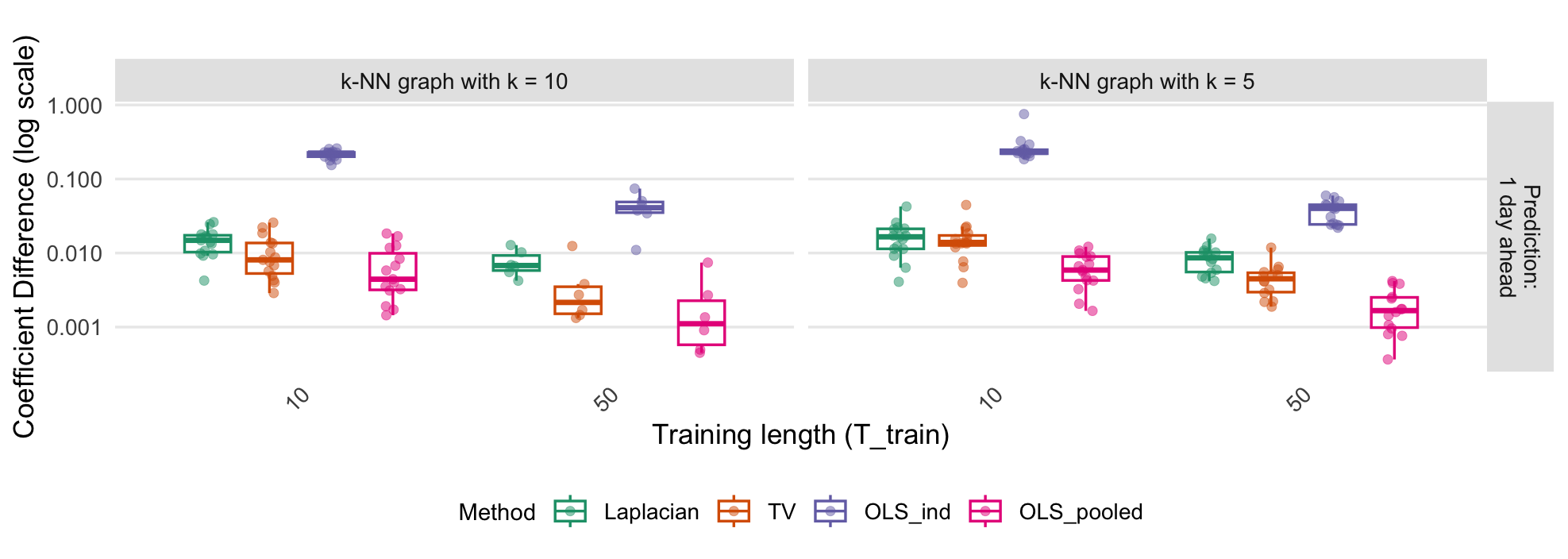}
    \subcaption{Coefficient Stability: NO }
  \end{subfigure}
  \hfill
  \begin{subfigure}{\textwidth}
    \centering
    \includegraphics[width=0.72\linewidth]{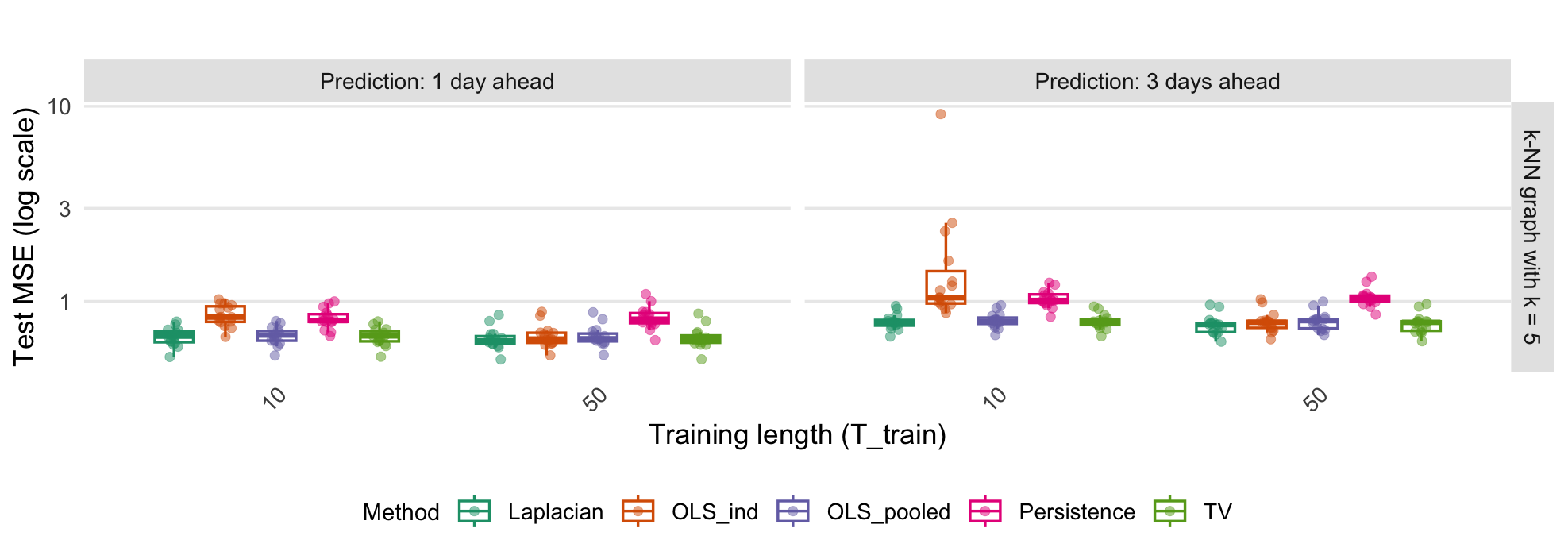}
    \subcaption{Test MSE: NO}
  \end{subfigure}
  \caption{Coefficient stability and test RMSE for predicting NO concentration. Boxplots show the distribution of errors over 15 independent trials (years).}
  \label{fig:results_h1}
\end{figure}
\vspace{-0.3cm}
\begin{figure}[H]
  \centering
  \begin{subfigure}{0.72\textwidth}
    \centering
    \includegraphics[width=\linewidth]{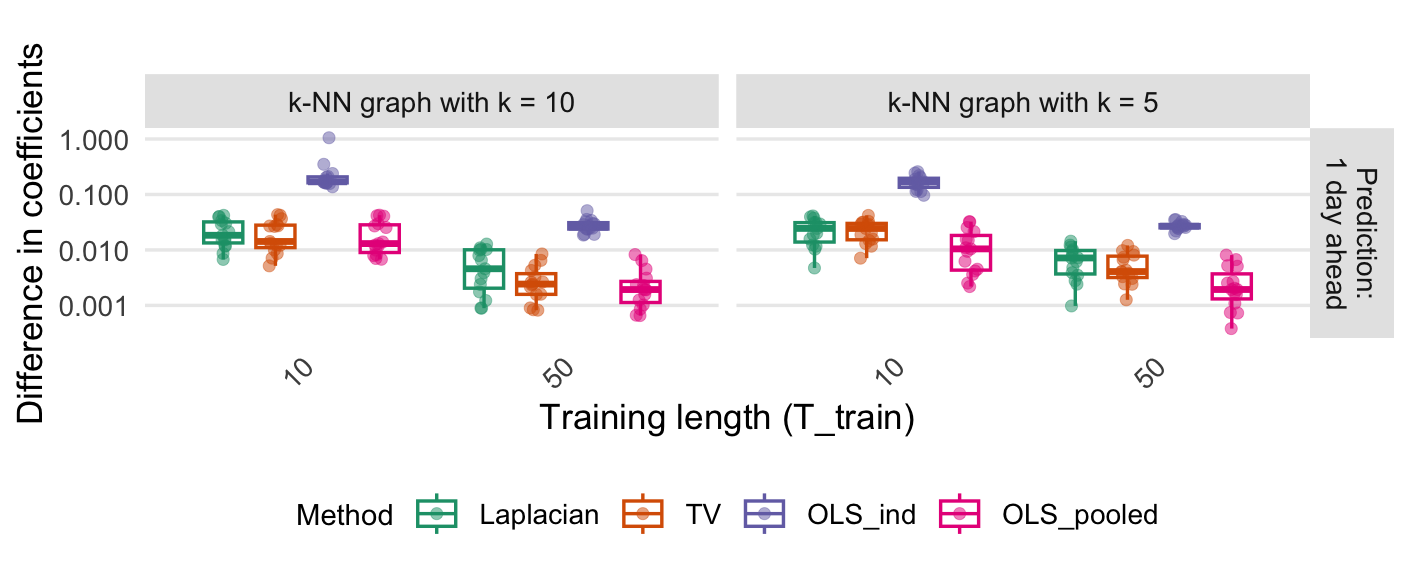}
    \subcaption{Coefficient Stability: PM10 }
  \end{subfigure}
  \hfill
  \begin{subfigure}{\textwidth}
    \centering
    \includegraphics[width=0.72\linewidth]{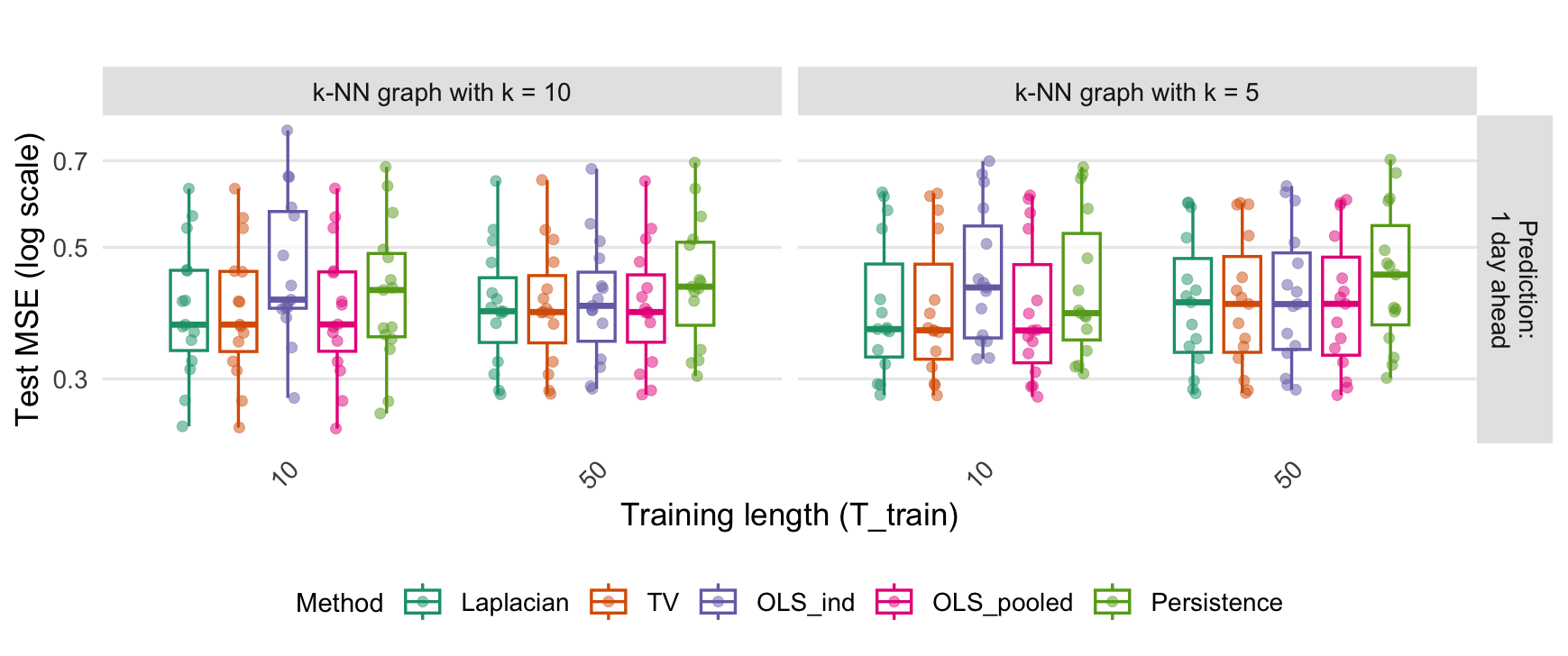}
    \subcaption{Test MSE: PM10}
  \end{subfigure}

  \caption{Coefficient stability and Test RMSE for predicting PM10 concentration.  Boxplots show the distribution of errors over 15 independent trials (years).}
  \label{fig:results_h1_pm}
\end{figure}

In this setting, both the Laplacian and graph total-variation (TV) denoisers outperform the alternative estimators, with performance between the Laplacian and TV being comparably close.

\begin{figure}[H]
  \centering

  \begin{subfigure}{0.49\textwidth}
    \centering
    \includegraphics[width=\linewidth]{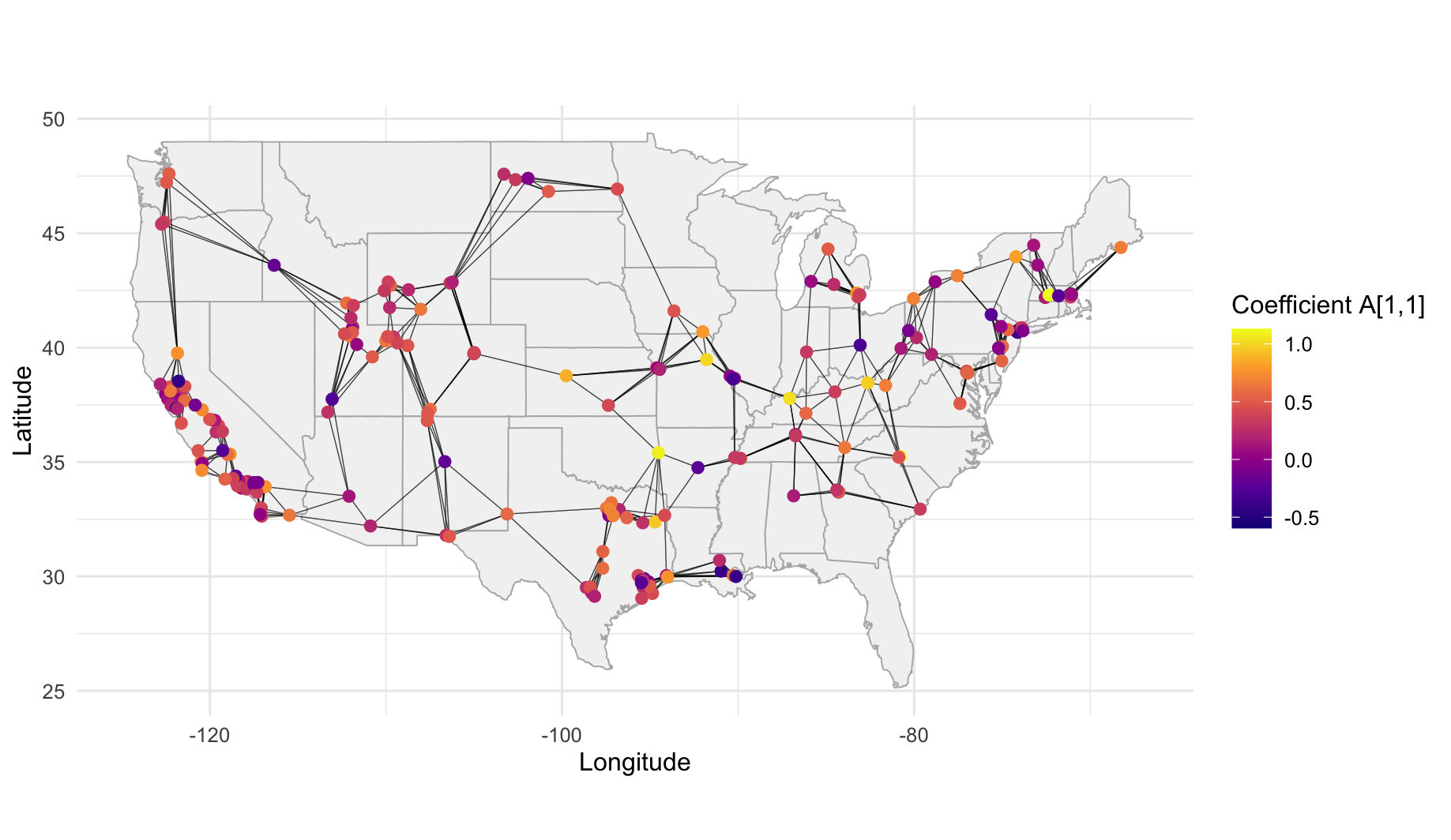}
    \subcaption{Individual OLS ( OLS\textsubscript{ind}) baseline}\label{fig:maps:ols_ind}
  \end{subfigure}
  \hfill
  \begin{subfigure}{0.49\textwidth}
    \centering
    \includegraphics[width=\linewidth]{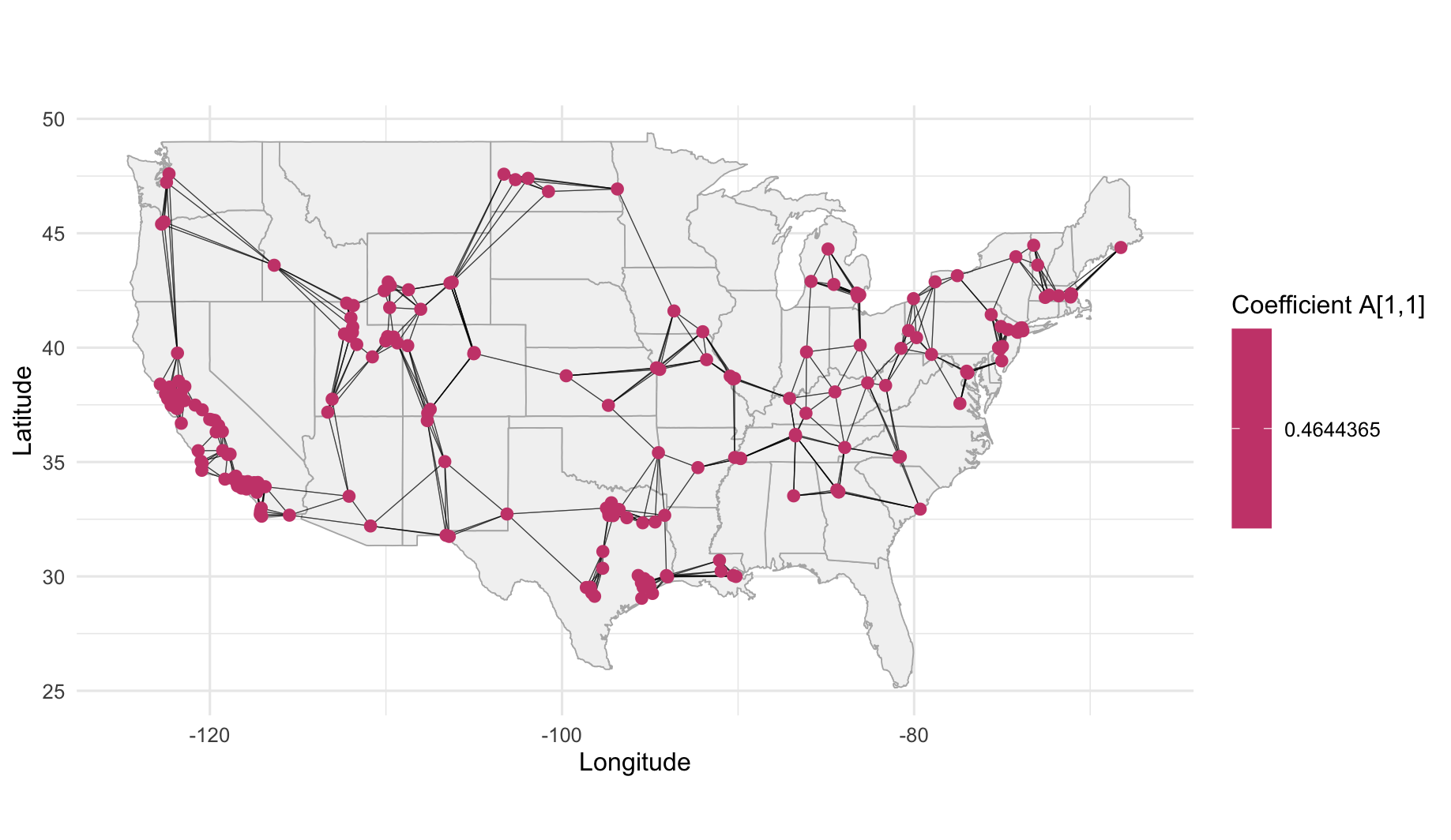}
    \subcaption{Pooled OLS ( OLS\textsubscript{pool}) benchmark}\label{fig:maps:ols_pool}
  \end{subfigure}

  \vspace{0.6em}

  \begin{subfigure}{0.49\textwidth}
    \centering
    \includegraphics[width=\linewidth]{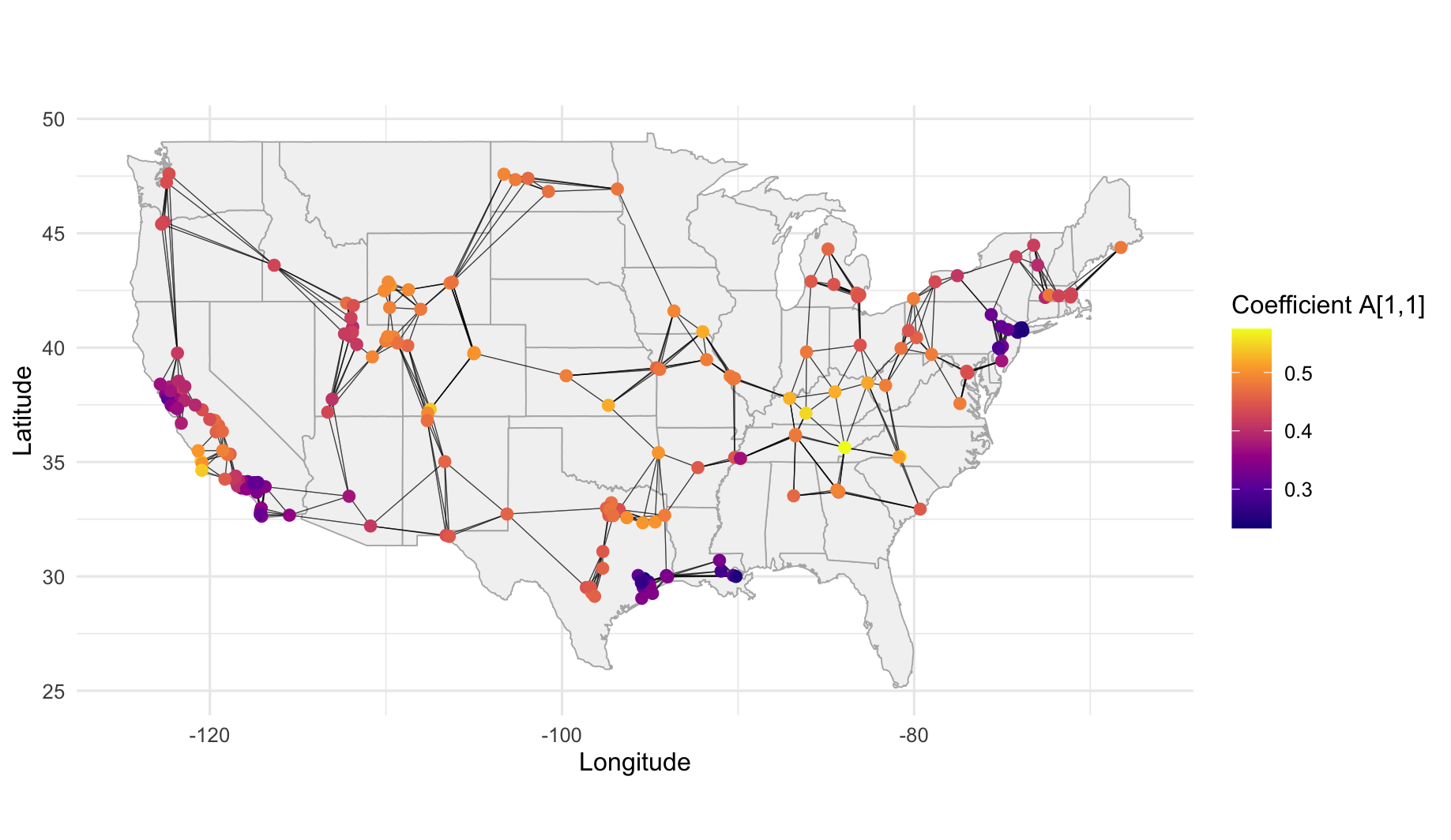}
    \subcaption{Laplacian denoiser}\label{fig:maps:lap}
  \end{subfigure}
  \hfill
  \begin{subfigure}{0.49\textwidth}
    \centering
    \includegraphics[width=\linewidth]{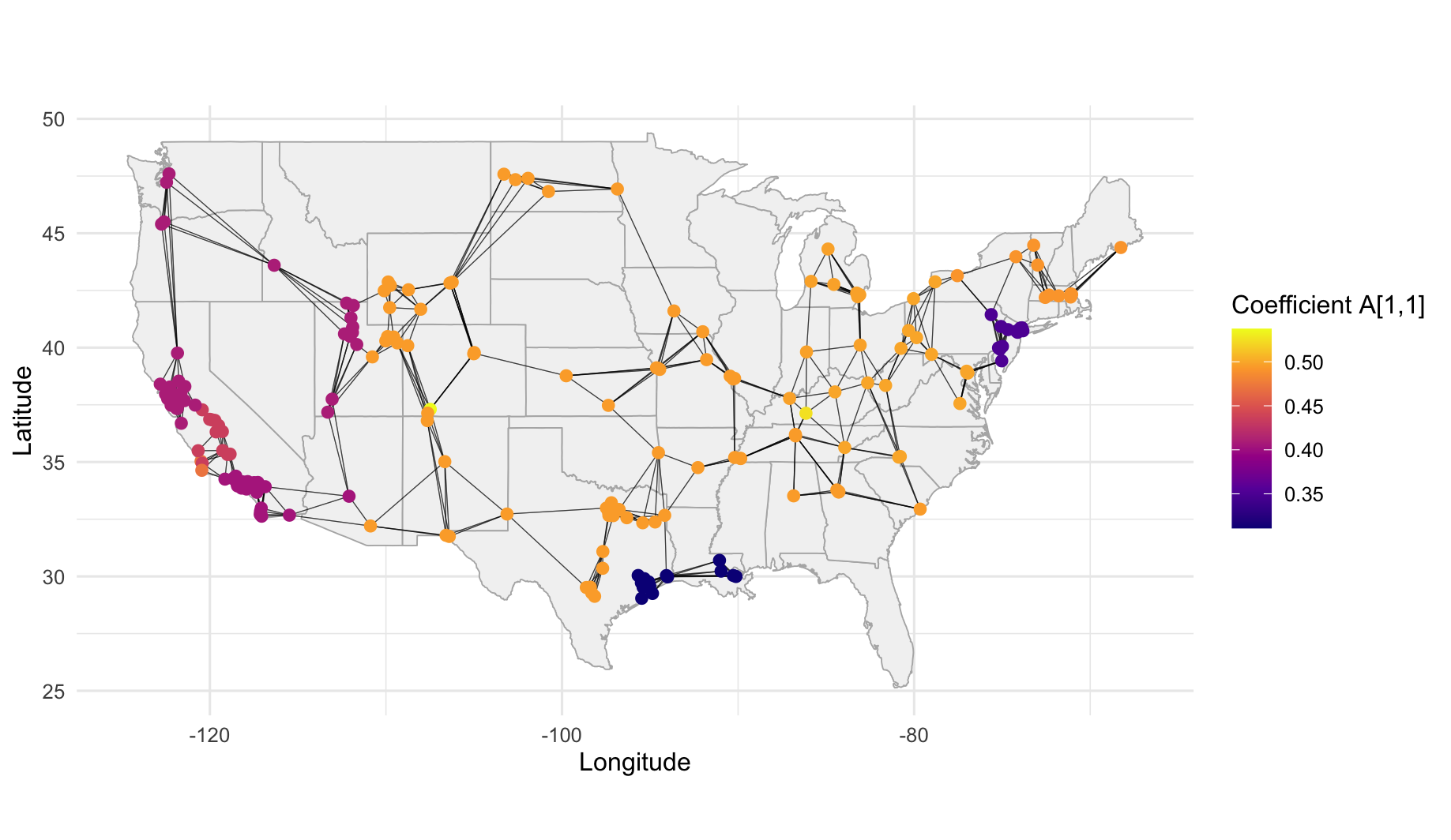}
    \subcaption{Graph TV denoiser}\label{fig:maps:tv}
  \end{subfigure}

  \caption{Monitoring stations colored by the value of the estimated LDS coefficient $\est{A}_l[1,1]$ under four methods. Nodes are stations; (optional) gray edges indicate the $k$-NN graph ($k=5$).}
  \label{fig:map_comp_coef}
\end{figure}


\begin{table}[]
\centering
\resizebox{\textwidth}{!}{%
\begin{tabular}{c|cc|cc|cc|}
\cline{2-7}
\textbf{} &
  \multicolumn{2}{c|}{\textbf{NO}} &
  \multicolumn{2}{c|}{\textbf{PM10 } } &
  \multicolumn{2}{c|}{\textbf{Temperature} }\\ \cline{2-7} 
\rowcolor[HTML]{EFEFEF} 
\cellcolor[HTML]{FFFFFF}\textbf{} &
  \multicolumn{1}{c|}{\cellcolor[HTML]{EFEFEF}\textbf{Stability}} &
  \textbf{Test RMSE} &
  \multicolumn{1}{c|}{\cellcolor[HTML]{EFEFEF}\textbf{Stability}} &
  \textbf{Test RMSE} &
  \multicolumn{1}{c|}{\cellcolor[HTML]{EFEFEF}\textbf{Stability}} &
  \textbf{Test RMSE} \\ \hline
\multicolumn{1}{|c|}{Persistence} &
  \multicolumn{1}{c|}{0  $\pm$ 0} &
  0.822 $\pm$ 0.0508 &
  \multicolumn{1}{c|}{0  $\pm$ 0} &
  0.431 $\pm$ 0.0626 &
  \multicolumn{1}{c|}{0  $\pm$ 0} &
  68.8 $\pm$ 18.4 \\ \hline
\rowcolor[HTML]{EFEFEF} 
\multicolumn{1}{|c|}{\cellcolor[HTML]{EFEFEF}OLS ind} &
  \multicolumn{1}{c|}{\cellcolor[HTML]{EFEFEF}0.214 $\pm$ 0.0144} &
  0.802 $\pm$ 0.0461 &
  \multicolumn{1}{c|}{\cellcolor[HTML]{EFEFEF}0.246 $\pm$ 0.116} &
  0.478 $\pm$ 0.0717 &
  \multicolumn{1}{c|}{\cellcolor[HTML]{EFEFEF}0.168 $\pm$ 0.0241} &
  84.9 $\pm$ 19.2 \\ \hline
\multicolumn{1}{|c|}{OLS Pooled} &
  \multicolumn{1}{c|}{0.00693 $\pm$ 0.00279} &
  0.673 $\pm$ 0.0391 &
  \multicolumn{1}{c|}{0.0198  $\pm$ 0.00676} &
  0.403 $\pm$ 0.0547 &
  \multicolumn{1}{c|}{0.0130 $\pm$ 0.00525} &
  69.4 $\pm$ 17.8 \\ \hline
\rowcolor[HTML]{EFEFEF} 
\multicolumn{1}{|c|}{\cellcolor[HTML]{EFEFEF}{\color[HTML]{000000} Laplacian}} &
  \multicolumn{1}{c|}{\cellcolor[HTML]{EFEFEF}{\color[HTML]{000000} 0.0148 $\pm$ 0.00292}} &
  {\color[HTML]{000000} 0.660 $\pm$ 0.0385} &
  \multicolumn{1}{c|}{\cellcolor[HTML]{EFEFEF}{\color[HTML]{000000} 0.0226 $\pm$ 0.00613}} &
  {\color[HTML]{000000} 0.404 $\pm$ 0.0546} &
  \multicolumn{1}{c|}{\cellcolor[HTML]{EFEFEF}{\color[HTML]{000000} 0.0237 $\pm$ 0.00576}} &
  {\color[HTML]{000000} 77.2 $\pm$ 17.7} \\ \hline
\multicolumn{1}{|c|}{TV} &
  \multicolumn{1}{c|}{0.0105 $\pm$ 0.00352} &
  0.666 $\pm$ 0.0391 &
  \multicolumn{1}{c|}{0.0209 $\pm$ 0.00653} &
  0.403 $\pm$ 0.0545 &
  \multicolumn{1}{c|}{0.0233 $\pm$ 0.00491} &
  71.1 $\pm$ 18.1 \\ \hline
\end{tabular}%
}
\caption{Mean test MSE ( $\pm 1.96 \times$ sd ) by method, for forecasting  pollutant one day ahead. Both NO and PM10 are fitted using the 10 nearest-neighbor graph, while TEMP uses a 5 nearest neighbor graph.}
\label{tab:tab:airquality}
\end{table}

\section*{Acknowledgment}
Authors are listed alphabetically. The work of O. Klopp was funded by the CY Initiative Grant Investissements d’Avenir
Agence Nationale de Recherche-16-Initiatives d’Excellence-0008 and Labex MME-DII Grant
ANR11-LBX-0023-01.
H.T was supported by a Nanyang Associate Professorship (NAP) grant from NTU Singapore. C.D. was supported by the National Science Foundation under Award Number 2238616, as well as the resources provided by the University of Chicago’s Research Computing Center.

\appendix
%
\section{Proof of Proposition \ref{prop:compatfac_bd}} \label{appsec:proof_compatfac_bd}
With a slight abuse of notation, we will sometimes identify the set $\calT_{\calE}$ as a subset of $\calE$, for convenience. This is clear since each $\set{l,l'} \in \calE$ can be identified by a unique label $j \in \abs{\calE}$, based on the ordering of the rows of $D$. 

Then for each $j \in \calT_{\calE}$, we can define the set $\calP_{j} \subseteq [d^2]$ as
\begin{equation*}
    \calP_{j} := \set{i \mod d^2: \ i \in \calT \ \text{and} \ (j-1)d^2 +1 \leq i \leq jd^2 },
\end{equation*}
%
%
where $\calP_j$ will be identified\footnote{Essentially, for each edge $j \in \calT_{\calE}$ which ``appears'' in $\calT$, we want to identify (through $\calP_j$) the set of coordinates that are ``attached'' to $\calT_{\calE}$. Alternatively, the set $\calT$ is completely specified by: (a)$\calT_{\calE}$ and, (b) the collection $(\calP_j)_{j \in \calT_{\calE}}$} by $\calP_{\set{l,l'}}$. Now for $\theta \in \matR^{md^2}$ formed by column-stacking $\theta_1,\dots,\theta_m \in \matR^{d^2}$, we obtain the bound 
%
\begin{align*}
    \norm{(\Dtil \theta)_{\calT}}_1 &= \sum_{\set{l,l'} \in \calT_{\calE}} \sum_{i \in \calP_{\set{l,l'}}} \abs{(\theta_l)_i - (\theta_{l'})_i}
    \\
    &\leq \sqrt{\abs{\calT}}\sqrt{\sum_{\set{l,l'} \in \calT_{\calE}} \sum_{i \in \calP_{\set{l,l'}}} \abs{(\theta_l)_i - (\theta_{l'})_i}^2} \\
    &\leq \sqrt{\abs{\calT}}\sqrt{2\sum_{\set{l,l'} \in \calT_{\calE}} \sum_{i \in \calP_{\set{l,l'}}}  ((\theta_l)_i^2 + (\theta_{l'})_i^2)} \\
    &\leq \sqrt{\abs{\calT}} \sqrt{2 \sum_{\set{l,l'} \in \calT_{\calE}}(\norm{\theta_l}_2^2 + \norm{\theta_{l'}}_2^2)} \\
    &\leq \sqrt{\abs{\calT}} \min\set{\sqrt{2\triangle}, 2\sqrt{\abs{\calT_{\calE}}}} \norm{\theta}_2,
\end{align*}
which completes the proof.

%
%
\section{Proof of Corollary \ref{cor:stable-main}} \label{subsec:proof_maincorr_stab}
We bound $\zeta_1$ and $\zeta_2$ (recall Lemma's \ref{lem:l2_bd_term1} and \ref{lem:linf_bd_term2}). Using 
\(\lVert A^{k}\rVert_{2}\le\rho_{\max}^{k}\) and
\(\lVert X\rVert_{F}\le\sqrt{d}\,\lVert X\rVert_{2}\), we obtain
	\begin{align*}
		\operatorname{Tr}\!\bigl(\Gamma_{t}(A^{\!*}_{l})\bigr)
		&= \sum_{k=0}^{t}\lVert A^{\!*k}_{l}\rVert_{F}^{2}
		\;\le\; d \sum_{k=0}^{t}\rho_{\max}^{2k}
		\;\le\; \frac{d}{1-\rho_{\max}^{2}}.
	\end{align*}
	For any canonical basis vector \(e_i\),
	\begin{align} \label{eq:corr_schur_stab_temp1}
		e_{i}^{\!\top}\Gamma_{t-1}(A^{\!*}_{l})e_{i}
		&=\sum_{k=0}^{t-1}\lVert A^{\!*k}_{l}e_{i}\rVert_{2}^{2}
		\;\le\;\sum_{k=0}^{t-1}\lVert A^{\!*k}_{l}\rVert_{2}^{2}
		\;\le\; \frac{1}{1-\rho_{\max}^{2}}, \quad \forall \ l \in [m].
	\end{align}
	We now get the following upper bounds  for $\zeta_{1}$ and $\zeta_{2}$. Using
	\[
	\sum_{l=1}^{m}\sum_{t=0}^{T-1}\operatorname{Tr}\!\bigl(\Gamma_{t}(A^{\!*}_{l})\bigr)
	\;\le\; mT\;\frac{d}{1-\rho_{\max}^{2}},
	\]
we obtain the bound
	\begin{equation}\label{bd_zeta_1_stable}
			\zeta_{1}(m,T,\delta)
			\;\;\le\;\;
			\frac{c_{1}\,m\,d\,T}{1-\rho_{\max}^{2}}
			\,\log\!(1/\delta)\;
	\end{equation}
	On the other hand, \eqref{eq:corr_schur_stab_temp1} readily implies
	\begin{equation}\label{bd_zeta_2_stable}
			\zeta_{2}(m,T,\delta)
			\;\;\le\;\;
			\frac{c_{1}\,\mu^{2}\,T}{1-\rho_{\max}^{2}}
			\,\log^{2}\!\Bigl(\tfrac{d^{2}|\calE|}{\delta}\Bigr).
	\end{equation}
	Now we bound $\beta:= \max_{l} \norm{\Atil^*_l}_2$. This is done using \cite[Lemma 5]{Jedra20}, which in our notation implies  
\begin{align*}
    \norm{\Atil^*_l}_2 \leq \sum_{t=0}^{T-1}  \norm{(A^*_l)^t}_2 \leq \sum_{t=0}^{T-1} \norm{A^*_l}_2^t \leq \sum_{t=0}^{\infty}  \norm{A^*_l}_2^t \leq \frac{1}{1-\rho_{\max}}.
\end{align*}
Hence we have $\beta\leq \frac{1}{1-\rho_{\max}}$.
    %
    %
    %
%
%
Using \eqref{bd_zeta_1_stable} and \eqref{bd_zeta_2_stable}, there exist constants $c, c', c'', c_1 > 1$ and $c_2 \in (0,1)$ such that for $\delta \in (0,c_2)$, we can respectively bound $F_1$, $F_2$ as follows.
\begin{align*}
    F_1 
    &\leq \sqrt{2}\Big(\frac{c_{1}\,d\,T}{1-\rho_{\max}^{2}}
			\,\log\!(1/\delta) + 1 \Big)^{1/2} \Big( \log(1/\delta) + \frac{d^2}{2} \log\Big(\frac{c_{1}\,d\,T}{1-\rho_{\max}^{2}}
			\,\log\!(1/\delta) + 1 \Big)\Big)^{1/2} \\
    &\leq c \Big(\frac{d\,T}{1-\rho_{\max}^{2}}
			\,\log\!(1/\delta) \Big)^{1/2} \Big( d^2\log(1/\delta) + d^2 \log\Big(\frac{c\,d\,T}{1-\rho_{\max}^{2}}
			\,\log\!(1/\delta) + 1 \Big)\Big)^{1/2}  \\
    &\leq c' \Big(\frac{d^3 \,T}{1-\rho_{\max}^{2}}
			\,\log\!(1/\delta) \Big)^{1/2} \log^{1/2}\Big(\frac{c\,d\,T}{\delta(1-\rho_{\max}^{2})}
			 \Big)  \\
    &\leq c'' \Big(\frac{d^3 \,T}{1-\rho_{\max}^{2}}
			\Big)^{1/2} \log\Big(\frac{d\,T}{\delta(1-\rho_{\max}^{2})}
			 \Big).
\end{align*}
Furthermore, 
\begin{align*}
    F_2 \leq c \zeta_2^{1/2}(m,T, \delta)  \leq c_1 \frac{\mu \sqrt{T}}{\sqrt{1-\rho_{\max}^2}} \log\left( \frac{d\abs{\calE}}{\delta}\right).
\end{align*}
The statement of the corollary now follows in a straightforward manner.
%
%
%

%
%
\section{Some technical tools}

\subsection{Talagrand's functionals and Gaussian width} \label{appsec:tal_prelim}
Recall Talagrand's $\gamma_{\alpha}$ functionals \cite{talagrand2014upper} which can be interpreted as a measure of the complexity of a (not necessarily convex) set.
\begin{definition}[\cite{talagrand2014upper}] \label{def:gamma_fun}
Let $(\calT,d)$ be a metric space. We say that a sequence of subsets of $\calT$, namely $(\calT_r)_{r \geq 0}$ is an admissible sequence if $\abs{\calT_0} = 1$ and $\abs{\calT_r} \leq 2^{2^{r}}$ for every $r \geq 1$. Then for any $0 < \alpha < \infty$, the $\gamma_{\alpha}$ functional of $(\calT, d)$ is defined as 
\begin{equation*}
    \gamma_{\alpha}(\calT, d) := \inf \sup_{t \in \calT} \sum_{r=0}^{\infty} 2^{r/\alpha} d(t, \calT_r)
\end{equation*}
with the infimum being taken over all admissible sequences of $\calT$.  
\end{definition}
The following properties of the $\gamma_{\alpha}$ functional are useful to note.
\begin{enumerate}
\item It follows directly from Definition \ref{def:gamma_fun} that for any two metrics $d_1, d_2$ such that $d_1 \leq a d_2$ for some $a > 0$,  we have
that $\gamma_{\alpha}(\calT, d_1) \leq a \gamma_{\alpha}(\calT, d_2)$.

\item For $\calT' \subset \calT$, we have that $\gamma_{\alpha}(\calT', d) \leq C_{\alpha} \gamma_{\alpha}(\calT, d)$ for $C_{\alpha} > 0$ depending only on $\alpha$.

\item If $f: (\calT, d_1) \rightarrow (\calU, d_2)$ is onto and for some constant $L$ satisfies
\begin{equation*}
    d_2(f(x), f(y)) \leq L d_1(x,y) \quad \text{for all } x,y \in \calT,
\end{equation*}
then $\gamma_{\alpha}(\calU, d_2) \leq C_{\alpha} L \gamma_{\alpha}(\calT, d_1)$ for $C_{\alpha} > 0$ depending only on $\alpha$.
\end{enumerate}
Properties $2$ and $3$ are stated in \cite[Theorem 1.3.6]{tal05} for the $\gamma_{\alpha}$ functional defined in \cite[Definition 2.2.19]{talagrand2014upper} with $C_{\alpha} = 1$. However, $\gamma_{\alpha}$ as in Definition \ref{def:gamma_fun} is equivalent to that in \cite[Definition 2.2.19]{talagrand2014upper} up to a constant depending only on $\alpha$; see \cite[Section 2.3]{talagrand2014upper}.
\begin{remark}
  Note that  \cite[Theorem 1.3.6]{tal05} is stated for another version of $\gamma_\alpha$ functional (see \cite[Definition 2.2.19]{talagrand2014upper}) which is equivalent to the one in Definition \ref{def:gamma_fun} up to a constant $C_\alpha$ depending only on $\alpha$ (see \cite[Section 2.3]{talagrand2014upper}).  
\end{remark}
%
%
%
%

By Talagrand's majorizing measure theorem \cite[Theorem 2.4.1]{talagrand2014upper}, the expected suprema of centered Gaussian processes $(X_s)_{s \in \calS}$ are characterized by $\gamma_2(\calS,d)$ as
\begin{equation} \label{eq:talag_maj_meas_thm}
  c \gamma_2(\calS,d) \leq \expec \sup_{s \in \calS}  X_s \leq C \gamma_2(\calS,d)
\end{equation}
for universal constants $c, C > 0$, with $d$ denoting the canonical distance $d(s,s') := (\expec[X_s - X_{s'}]^2)^{1/2}$. For instance, if $\calS \subset \matR^{n \times m}$, and $X_s = \dotprod{G}{s}$ for a $n \times m$ matrix $G$ with iid standard Gaussian entries, we obtain $d(s,s') = \norm{s - s'}_F^2$. Then \eqref{eq:talag_maj_meas_thm} implies $\expec \sup_{s \in \calS} \dotprod{G}{s} \asymp \gamma_2(\calS, \norm{\cdot}_F)$. Here, $\expec \sup_{s \in \calS} \dotprod{G}{s}$ is known as the \emph{Gaussian width} of the set $\calS$, denoted as $w(\calS)$. For an overview of the properties of the Gaussian width, the reader is referred to \cite{HDPbook}.

%
%
\subsection{Suprema of second order subgaussian chaos processes} \label{appsec:sup_chaos_krahmer}
For a 
set of matrices $\calA$, define the terms 
\begin{equation} \label{eq:rad_frob_spec}
d_F(\calA) := \sup_{A \in \calA} \norm{A}_F, \quad d_2(\calA) := \sup_{A \in \calA} \norm{A}_2.
\end{equation}
These can be interpreted as other types of complexity measures of the set $\calA$. The following result from \cite{krahmer14} is a concentration bound for the suprema of second order subgaussian chaos processes involving positive semidefinite (p.s.d) matrices. 
\begin{theorem}[\cite{krahmer14}] \label{thm:krahmer_chaos}
Let $\calA$ be a set of matrices and $\xi$ be a vector whose entries are independent, zero-mean, variance $1$, and are $L$-subgaussian random variables. Denote
\begin{align*}
F &= \gamma_2(\calA, \norm{\cdot}_2) [\gamma_2(\calA, \norm{\cdot}_2) + d_F(\calA)] + d_F(\calA) d_2(\calA) \\
G &= d_2(\calA)[\gamma_2(\calA, \norm{\cdot}_2) + d_F(\calA)], \quad \text {and} \quad H = d_2^2(\calA)    
\end{align*}
where $d_2, d_F$ are as in \eqref{eq:rad_frob_spec}. Then there exist constants $c_1,c_2 > 0$ depending only on $L$ such that for any $t > 0$ it holds that 
\begin{equation*}
    \prob\left(\sup_{A \in \calA} \abs{\norm{A \xi}_2^2 - \expec[\norm{A \xi}_2^2]} \geq c_1 F + t \right) \leq 2\exp\left(-c_2 \min \set{\frac{t^2}{G^2}, \frac{t}{H}} \right).
\end{equation*}
\end{theorem}

%
\section{Proof of Proposition \ref{prop:set_incl_basu}} \label{appsec:proof_lem_set_incl}
For any $h \in \calC_{\calS} \cap \unitsph$, we have for $u(h) = \Dtil h$ that 
\begin{align*}
\norm{u(h)}_1 
= \norm{\Dtil h}_1 
    &= \norm{(\Dtil h)_{\calS}}_1 + \norm{(\Dtil h)_{\calS^c}}_1 \\ 
    &\leq 4\norm{(\Dtil h)_{\calS}}_1 + \rev{\frac{4}{r}}\norm{\Dtilacomp}_1 + 1 \tag{since $h \in \calC_{\calS} \cap \unitsph$} \\
    &\leq 4\frac{\sqrt{\abs{\calS}}}{\kappa_{\calS}} \norm{h}_2 + \rev{\frac{4}{r}}\norm{\Dtilacomp}_1 + 1 \tag{Recall Definition \ref{def:invfac_compat}} \\
    &= 4\frac{\sqrt{\abs{\calS}}}{\kappa_{\calS}} + \rev{\frac{4}{r}}\norm{\Dtilacomp}_1 + 1 \tag{since $h \in \unitsph$}.
\end{align*}

%
\section{Additional details from Section \ref{subsec:proof_l2_bd_term1}} \label{appsec:self_norm_args}
Denote the sigma algebra $\calF_t := \sigma(\eta_{l,1}, \dots,\eta_{l,t})_{l=1}^m$, we then obtain a filtration $(\calF_t)_{t=1}^\infty$. The idea now is to show the following result, analogous to that in \cite[Lemma 9]{abbasi11}.
\begin{lemma} \label{lem:self_norm_supermart}
    For any $u \in \matR^d$ and $t \geq 1$, consider
    \begin{equation*}
        M_t^u := \exp\left(\sum_{s=1}^t \sum_{l=1}^m \Big[\dotprod{u}{x_{l,s} \otimes \eta_{l,s+1}} - \frac{\norm{\mu}_{(x_{l,s} x_{l,s}^\top) \otimes I_d}^2}{2} \Big] \right)
    \end{equation*}
    Let $\bar{\tau}$ be a stopping time  w.r.t $(\calF_t)_{t=1}^\infty$. Then $M_{\bar{\tau}}^u$ is well defined and $\expec[M_{\bar{\tau}}^u] \leq 1$.
\end{lemma}
\begin{proof}
    Denote 
    \begin{align*}
        D_t^u = \exp\left(\sum_{l=1}^m \Big[\dotprod{u}{x_{l,t} \otimes \eta_{l,t+1}} - \frac{\norm{\mu}_{(x_{l,t} x_{l,t}^\top) \otimes I_d}^2}{2} \Big] \right)
    \end{align*}
    Observe that $D_t^u$ and $M_t^u$ are both $\calF_{t+1}$ measurable. Hence it follows that 
    \begin{align*}
        \expec[D_t^u \vert \calF_t] &= \prod_{l=1}^m \expec\left[\exp\left(\sum_{l=1}^m \Big[\dotprod{u}{x_{l,t} \otimes \eta_{l,t+1}} - \frac{\norm{\mu}_{(x_{l,t} x_{l,t}^\top) \otimes I_d}^2}{2} \Big] \right) \Big\vert \calF_t \right] \\
        &\leq 1 \ \text{a.s}
    \end{align*}
    since $\eta_{l,t}$ has independent $1$-subgaussian entries (for each $l,t$). Using this bound, we obtain 
    \begin{align*}
        \expec[[D_t^u \vert \calF_t]] \leq D_1^u D_2^u \cdots D_{t-1}^u = M_{t-1}^u
    \end{align*}
    which implies $(M_t^u)_{t=1}^\infty$ is a super-martingale and also that $\expec[M_t^u] \leq 1$. Now one can show using the same arguments as presented in the proof of \cite[Lemma 9]{abbasi11} to show that $M_{\bar{\tau}}^u$ is well defined and $\expec[M_{\bar{\tau}}^u] \leq 1$; the details are omitted. 
\end{proof}
Equipped with Lemma \ref{lem:self_norm_supermart}, the bound in \eqref{eq:pitilbd_selfnorm} now follows in an analogous manner to that in the proof of \cite[Theorem 1]{abbasi11}.

\section{Control of the inverse scaling factor}\label{sec:control_inv_scaling_fac}
%
%
\subsection{2D grid} \label{subsec:invfac_2dgrid}
We introduce the following notation.	
	\begin{itemize}
		\item Let \(G_{N\times N} = (V,E)\) denote the \(N \times N\) two‑dimensional rectangular grid with vertex set
		\(V = \{0,\dots,N-1\}^2\).  Hence \(|V| = m = N^2\) and let $D_2$ be the 
		\emph{incidence matrix} of the 2D grid.
		\item We write the pseudoinverse of \(D_2\) as
		\[
		S \;=\; D_2^{\dagger}\;=\;L^{\dagger} D_2^\top,
		\]
		where \(L^{\dagger}\) is the pseudoinverse of \(L\).
	\end{itemize}
	
	It was shown in \cite{huetter16} that
	\[
	\max_{e \in E}\, \bigl\| S_{:,e} \bigr\|_2 \;\le\; C \sqrt{\log m},
	\]
	i.e.\ the \emph{column} norms grow only logarithmically.  
	We now bound the \emph{row} norms.
	
	
	\begin{lemma}\label{lem:row-diag}
		For any vertex \(v\in V\), the squared $\ell_2$ norm of the \(v\)-th row of \(S\) equals the \(v\)-th diagonal entry of \(L^{\dagger}\), i.e., 
		\[
		\bigl\| S_{v,:} \bigr\|_2^2 
		\;=\; (L^{\dagger})_{vv}.
		\]
	\end{lemma}
	
	\begin{proof}
		Let \(e_v \in \mathbb{R}^{n}\) be the \(v\)-th canonical basis vector.
		Because \(S = L^{\dagger} D_2^\top\),
		\[
		S_{v,:}^\top 
		\;=\; e_v^\top L^{\dagger} D_2^\top,
		\quad\Longrightarrow\quad
		\| S_{v,:} \|_2^2
		\;=\;
		e_v^\top L^{\dagger} D_2^\top D_2 L^{\dagger} e_v.
		\]
		Since \(D_2^\top D_2 = L\) and \(L^{\dagger} L L^{\dagger} = L^{\dagger}\),
		\[
		\| S_{v,:} \|_2^2
		\;=\;
		e_v^\top L^{\dagger} e_v
		\;=\;
		(L^{\dagger})_{vv}.
		\]
	\end{proof}
	
	Consequently, an upper bound on
	\(\displaystyle \max_{v\in V} (L^{\dagger})_{vv}\)
	implies the desired row‑norm bound.
	
	\paragraph{One‑dimensional path.}
	Let \(L_1\) be the \(N\times N\) Laplacian of the path graph
	\(P_N\).
	It possesses the discrete cosine eigenbasis\footnote{In a slight abuse of notation, we write the eigenvalues here as $\lambda_0 \leq \lambda_1 \leq \cdots \leq \lambda_{N-1}$ for convenience.}
	\[
	L_1 = V_1 \Lambda_1 V_1^\top,
	\qquad
	\lambda_{k}
	= 2-2\cos\!\Bigl(\frac{k\pi}{N}\Bigr),
	\quad
	k = 0,\dots,N-1.
	\]
	For \(k\neq 0\), the entries of the $k$‑th eigenvector satisfy
	\(|(V_1)_{ik}| \le \sqrt{\tfrac{2}{N}}\);
	for \(k=0\), \((V_1)_{i0} = \tfrac1{\sqrt{N}}\) (see, e.g. \cite[Chapter 1.5]{strang2007computational}). 
	
	\paragraph{Two‑dimensional grid.}
	Using the (Kronecker) Cartesian product,
	\[
	L \;=\; L_1 \otimes I_N \;+\; I_N \otimes L_1
	\;=\; V_2 \Lambda_2 V_2^\top,
	\]
	where
	\(V_2 = V_1 \otimes V_1\)
	and
	\(\Lambda_2 = \Lambda_1 \otimes I_N + I_N \otimes \Lambda_1\). 
	Writing a vertex \(v\) as a pair \((i,j)\),
	the corresponding diagonal entry of \(L^{\dagger}\) expands to
	\begin{equation}\label{eq:diag-expand}
		(L^{\dagger})_{vv}
		\;=\;
		\sum_{\,\substack{k,\ell = 0 \\ (k,\ell)\neq (0,0)}}^{N-1}
		\frac{ \, (V_1)_{ik}^2 \, (V_1)_{j\ell}^2 }{\lambda_{k} + \lambda_{\ell}}.
	\end{equation}
	Indeed, let $\phi_{k\ell}$ denote the eigenvector and $\mu_{k\ell}$ the eigenvalue of $L$, we have
\begin{align*}
\mu_{k\ell} &:= \lambda_k + \lambda_\ell,
\qquad
\phi_{k\ell} := v_k \otimes v_\ell,
\qquad k,\ell \in \{0,\dots,N-1\}.
\end{align*}
Because $L$ is symmetric, its pseudoinverse is obtained by inverting the positive eigenvalues and leaving the zero mode at 0, i.e., 
\begin{align*}
L^\dagger
= V_2 \,\Lambda_2^\dagger\, V_2^\top
= \sum_{\substack{0 \le k,\ell \le N-1\\(k,\ell)\neq(0,0)}}
\frac{1}{\lambda_k + \lambda_\ell}\;\phi_{k\ell}\,\phi_{k\ell}^\top,
\qquad
(\Lambda_2^\dagger)_{k\ell,k\ell} :=
\begin{cases}
\dfrac{1}{\lambda_k+\lambda_\ell}, & (k,\ell)\neq(0,0),\\[4pt]
0, & (k,\ell)=(0,0), 
\end{cases}
\end{align*}
and
\begin{align*}
(L^\dagger)_{vv}
&= e_{(i,j)}^\top L^\dagger\, e_{(i,j)}
= \sum_{\substack{0 \le k,\ell \le N-1\\(k,\ell)\neq(0,0)}}
\frac{\big(\phi_{k\ell}(i,j)\big)^2}{\lambda_k + \lambda_\ell} \\[4pt]
&= \sum_{\substack{0 \le k,\ell \le N-1\\(k,\ell)\neq(0,0)}}
\frac{\big(v_k(i)\,v_\ell(j)\big)^2}{\lambda_k + \lambda_\ell}
= \sum_{\substack{0 \le k,\ell \le N-1\\(k,\ell)\neq(0,0)}}
\frac{(V_1)_{ik}^2\,(V_1)_{j\ell}^2}{\lambda_k + \lambda_\ell}.
\end{align*}
Note\footnote{Since $\sin x \asymp x$ for $x \in [0,\pi/2]$.} that \(\lambda_{k} = 4\sin^2(\pi k/(2N)) \asymp (\pi k / N)^2\) for \(k=0,1,\dots,N-1\),
	which implies \(\lambda_{k} \ge c\) (for some absolute \(c>0\)) once \(k \ge N/(2\pi)\). Hence, 
	we now split the double sum \eqref{eq:diag-expand} into ``low‑'' and ``high‑frequency''
	blocks. 
	
	\subsection*{Low frequencies: \(k,\ell \le \frac{N}{2\pi}\)}
	
	Using \(|(V_1)_{ik}|, |(V_1)_{j\ell}| \le \sqrt{\tfrac{2}{N}}\) and
	\(\lambda_{k} \asymp (\pi k/N)^2\),
	\[
	\sum_{\substack{1\le k,\ell \le N/(2\pi)}} 
	\frac{ (V_1)_{ik}^2 (V_1)_{j\ell}^2 }{\lambda_{k}+\lambda_{\ell}}
	\;\le\;
	\frac{4}{N^2}
	\sum_{\substack{1\le k,\ell \le N/(2\pi)}} 
	\frac{N^2/\pi^2}{k^2 + \ell^2}
	\;=\;
	\frac{4}{\pi^2}
	\sum_{k,\ell=1}^{{\lfloor N/(2\pi)\rfloor}} 
	\frac1{k^2+\ell^2}.
	\]
     Let $R:=\lfloor N/(2\pi)\rfloor$ and define
$\mathcal A_j:=\{(k,\ell)\in\mathbb N^2: 2^j<\sqrt{k^2+\ell^2}\le 2^{j+1}\},\ j=0,\dots,J:=\lfloor\log_2 R\rfloor$.
Then for $(k,\ell)\in\mathcal A_j, 1/(k^2+\ell^2)\le 2^{-2j}$, so
\[
\sum_{k=1}^{R}\sum_{\ell=1}^{R}\frac{1}{k^2+\ell^2}
\le \sum_{j=0}^{J}\frac{|\mathcal A_j|}{2^{2j}}.
\]
Since $|\mathcal A_j|\le (2^{j+1})^2=4\cdot 2^{2j}$, each shell contributes at most $4$, hence 
\[
\sum_{k=1}^{R}\sum_{\ell=1}^{R}\frac{1}{k^2+\ell^2}
\le 4(J+1)=O(\log R).
\]
Substituting into
\(
\frac{4}{\pi^2}\sum_{1\le k,\ell\le R}\frac{1}{k^2+\ell^2}
\)
gives the low-frequency bound \(O(\log N)\).
	
	\subsection*{High frequencies}
	
	If \(k\) or \(\ell\) exceeds \(N/(2\pi)\), then
	\(\lambda_{k} + \lambda_{\ell} \ge c\) for some
	absolute \(c>0\).
	Together with \(|(V_1)_{ik}|^2,(V_1)_{j\ell}|^2 \le 2/N\),
	the contribution is \(O(1)\).
	
	\medskip
	Combining both regions,
	\[
	(L^{\dagger})_{vv} \;\le\; C \log N
	\quad\Longrightarrow\quad
	\max_{v\in V}(L^{\dagger})_{vv} \;\le\; C \log N,
	\]
	for an absolute constant \(C\).
	This proves the following proposition.
	\begin{proposition}
	    [Bound on $\mu'$ for $2D$ grid]\label{prop:2D_grid}
		Let \(D_2\) be the incidence matrix of the \(N\times N\) two‑dimensional grid
		and \(S = D_2^{\dagger}\) its Moore–Penrose inverse.
		Then
		\[
		\boxed{ \;
			\max_{v \in V} \; \bigl\| S_{v,:} \bigr\|_{2}
			\;\le\;
			C \sqrt{\log {m}}
			\;} ,
		\qquad {m} = N^2,
		\]
		where \(C>0\) is an absolute constant.
	\end{proposition}
%
%
\subsection{$n$-dimensional grid} \label{subsec:invfac_ndgrid}
We can extend the proof of Proposition \ref{prop:2D_grid} to the general case of a $n$-dimensional grid.
\begin{proposition}[Bound on $\mu'$ for n-dimensional grid]\label{prop:dD_grid}
  Let $G_{N,n}$ be the $N^n$ grid, $D_n$ its incidence, $L_n=D_n^\top D_n,
S=D_n^\dagger$. Then for some $C_n>0$ (dimension only) {where $C_n=O(2^n)$}, 
\[
\max_v \|S_{v,:}\|_2^2=(L_n^\dagger)_{vv}\le
\begin{cases}
C_1 N, & n=1,\\
C_2 \log N, & n=2,\\
C_n, & n \ge 3.
\end{cases}
\]  
\end{proposition} 
\begin{proof}
 Lemma~7 implies that $\|S_{v,:}\|_2^2=(L_n^\dagger)_{vv}$ and we have $L_n=\sum_{r=1}^n I^{\otimes (r-1)}\otimes L_1\otimes I^{\otimes (n-r)}$
with $L_1=V_1\Lambda_1 V_1^\top$, $\lambda_k=2-2\cos(k\pi/N)$.
Then for $v=(i_1,\ldots,i_n)$,
\begin{equation}\label{prp_dDgrid_eq1}
  (L_n^\dagger)_{vv}=\sum_{\boldsymbol{k}\neq \mathbf 0}
\frac{\prod_{r=1}^n (V_1)_{i_r k_r}^2}{\sum_{r=1}^n \lambda_{k_r}}
\end{equation}
where $\boldsymbol{k}=(k_r)$.
 Using $(V_1)_{ik}^2\le 2/N$ (and $1/N$ for $k=0$) gives a uniform
numerator bound $(2/N)^n$. Let
		\[
		R:=\Big\lfloor \frac{N}{2\pi}\Big\rfloor.
		\]
		For $1\le k\le R$ we have $\frac{k\pi}{N}\le \frac{1}{2}$ and (by Taylor’s inequality) $2-2\cos x\ge \frac{x^2}{2}$ for $x\in[0,1/2]$. Hence
		\[
		\lambda_{k}\;\ge\;\frac{1}{2}\Big(\frac{\pi k}{N}\Big)^2,\qquad
		\sum_{r=1}^n \lambda_{k_r}\;\ge\;\frac{\pi^2}{2N^2}\,\|\boldsymbol{k}\|_2^2
		\quad\text{whenever }\|\boldsymbol{k}\|_\infty\le R,\ \boldsymbol{k}\neq \mathbf{0}.
		\]
		Then, we get
		\[
		\sum_{\substack{\boldsymbol{k}\neq \mathbf{0}\\ \|\boldsymbol{k}\|_\infty\le R}}
		\frac{\prod_{r=1}^n (V_1)_{\,i_r,k_r}^{\,2}}{\sum_{r=1}^n \lambda_{k_r}}
		\;\le\;
		\Big(\frac{2}{N}\Big)^{n}\cdot \frac{2N^2}{\pi^2}
		\sum_{\substack{\boldsymbol{k}\neq \mathbf{0}\\ \|\boldsymbol{k}\|_\infty\le R}}
		\frac{1}{\|\boldsymbol{k}\|_2^2}
		\;=\; \frac{C 2^n}{N^{n-2}}\,S_n(R),
		\]
		where
		\[
		S_n(R):=\sum_{\substack{\boldsymbol{k}\in\{0,\dots,R\}^n\\ \boldsymbol{k}\neq \mathbf{0}}}\frac{1}{\|\boldsymbol{k}\|_2^2}.
		\]
Next, we prove the following lemma.		
\begin{lemma}\label{lem:lattice}
    There exists $C_n>0$ such that
			\[
			S_n(R)\ \le\
			\begin{cases}
				C_1, & n=1,\\[2pt]
				C_2 \,\log R, & n=2,\\[2pt]
				C_n \,R^{\,n-2}, & n\ge 3.
			\end{cases}
			\]
\end{lemma}
\begin{proof}[Proof of Lemma~\ref{lem:lattice}]
			Partition into dyadic shells
			\[
			\mathcal{A}_j:=\{\boldsymbol{k}\in\mathbb{Z}_{\ge 0}^n: \ 2^j<\|\boldsymbol{k}\|_2\le 2^{j+1}\},\qquad j=0,1,\dots,\lfloor \log_2 R\rfloor.
			\]
			On $\mathcal{A}_j$ we have $\|\boldsymbol{k}\|_2^{-2}\le 2^{-2j}$, and $|\mathcal{A}_j|\le C'\,2^{(j+1)n}$ (crude cardinality bound by volume). Thus
			\[
			S_n(R)\;\le\; \sum_{j=0}^{\lfloor \log_2 R\rfloor}\frac{|\mathcal{A}_j|}{2^{2j}}
			\;\le\; C' 2^n \sum_{j=0}^{\lfloor \log_2 R\rfloor} 2^{j(n-2)}.
			\]
			The sum is $O(1)$ if $n=1$, equals $O(\log R)$ if $n=2$, and is $O(R^{n-2})$ if $n\ge 3$.
\end{proof}
		Lemma \ref{lem:lattice} implies
		\[
		\sum_{\substack{\boldsymbol{k}\neq \mathbf{0}\\ \|\boldsymbol{k}\|_\infty\le R}}
		\frac{\prod_{r=1}^n (V_1)_{\,i_r,k_r}^{\,2}}{\sum_{r=1}^n \lambda_{k_r}}
		\ \le\
		\begin{cases}
			C\,N, & n=1,\\[2pt]
			C\,\log N, & n=2,\\[2pt]
			C, & n\ge 3.
		\end{cases}
		\]
		
		If $\|\boldsymbol{k}\|_\infty>R$, then some $k_r\ge R+1$ and $\frac{k_r\pi}{N}\ge \frac{1}{2}$, hence $\lambda_{k_r}=2-2\cos\!\big(\frac{k_r\pi}{N}\big)\ge c_0$ for a universal $c_0>0$. Thus
		\[
		\sum_{r=1}^n \lambda_{k_r}\ \ge\ c_0,\qquad
		\frac{\prod_{r=1}^n (V_1)_{\,i_r,k_r}^{\,2}}{\sum_{r=1}^n \lambda_{k_r}}
		\ \le\ \frac{1}{c_0}\Big(\frac{2}{N}\Big)^{n}.
		\]
		Summing over at most $N^n$ indices gives the uniform bound
		\[
		\sum_{\|\boldsymbol{k}\|_\infty>R}\frac{\prod_{r=1}^n (V_1)_{\,i_r,k_r}^{\,2}}{\sum_{r=1}^n \lambda_{k_r}}
		\ \le\ \frac{1}{c_0}\Big(\frac{2}{N}\Big)^{n}N^n \;=\; \frac{2^n}{c_0}\;=\;C_n'.
		\]
		Combining the two blocks in  yields the result.
 \end{proof}

\section{Controlling $\Delta_G$}\label{appendix:control_Delta_gr}
%
%
%
\begin{lemma}[Control of $\Delta_G$ by edgewise Frobenius variation]\label{lem:deltaG-edge-Frob}
Let $G=([m],\mathcal E)$ be a connected graph and suppose each local system matrix is stable:
$\|A_\ell^\ast\|_2 \le \rho_{\max} < 1$ for all $\ell\in[m]$.  Define the controllability
aggregate
\[
G_\ell \;:=\; \sum_{t=1}^T \Gamma_{t-1}(A_\ell^\ast)
\quad\text{with}\quad
\Gamma_{t-1}(A) \;=\; \sum_{s=0}^{t-1} A^s (A^s)^{\!\top},
\]
and let $\bar G := \frac1m\sum_{\ell=1}^m G_\ell$.  Recall
\[
\Delta_G \;:=\; \max_{a,b\in[d]}\;\Bigg\{\sum_{\ell=1}^m\Big[(G_\ell)_{b,a}-(\bar G)_{b,a}\Big]^2\Bigg\}^{1/2}.
\]
Then
\[
\boxed{\;
\Delta_G \;\le\; \frac{L_T(\rho_{\max})}{\sqrt{\lambda_{m-1}(G)}}
\Bigg(\,\sum_{\{\ell,\ell'\}\in \calE}\,\|A_\ell^\ast-A_{\ell'}^\ast\|_F^2\Bigg)^{\!1/2}\;,
\qquad
L_T(\rho) \;:=\; 2\sum_{s=1}^{T-1}(T-s)s\,\rho^{\,2s-1}
\;\le\; \frac{2\rho\,T}{(1-\rho^2)^2}\,.
\;}
\]
\end{lemma}

\begin{proof}
\emph{Step 1 (entrywise Lipschitz control of $G(\cdot)$).}
Fix $a,b\in[d]$ and set $g_{a,b}(A):=e_b^\top G(A)e_a
= \sum_{t=1}^T\sum_{s=0}^{t-1} e_b^\top A^s(A^s)^{\!\top} e_a$.
For stable $A,B$ with $\|A\|_2,\|B\|_2\le\rho_{\max}$, the standard telescoping bound
\[
\|A^s - B^s\|_2 \;\le\; \sum_{r=0}^{s-1}\|A\|_2^{\,s-1-r}\,\|A-B\|_2\,\|B\|_2^{\,r}
\;\le\; s\,\rho_{\max}^{\,s-1}\,\|A-B\|_2
\]
implies
\[
\|A^s(A^s)^\top - B^s(B^s)^\top\|_2
\;\le\; \|A^s\|_2\|A^s-B^s\|_2 + \|A^s-B^s\|_2\|B^s\|_2
\;\le\; 2s\,\rho_{\max}^{\,2s-1}\|A-B\|_2.
\]
Hence
\begin{equation}\label{bound_g_(a,b)}
    \big|g_{a,b}(A) - g_{a,b}(B)\big|
\;\le\; \sum_{t=1}^T \sum_{s=1}^{t-1} 2s\,\rho_{\max}^{\,2s-1}\,\|A-B\|_2
\;=\; L_T(\rho_{\max})\,\|A-B\|_2
\;\le\; L_T(\rho_{\max})\,\|A-B\|_F.
\end{equation}

\smallskip
\emph{Step 2 (from Lipschitz to nodewise variance).}
Let $\bar A^\ast := \frac1m\sum_{\ell=1}^m A_\ell^\ast$. The empirical mean minimizes squared deviations,
so for each fixed $(a,b)$,
\[
\sum_{\ell=1}^m \Big[(G_\ell)_{b,a}-(\bar G)_{b,a}\Big]^2
\;\le\; \sum_{\ell=1}^m \Big[(G_\ell)_{b,a}-g_{a,b}(\bar A^\ast)\Big]^2.
\]
Apply Step 1 entrywise with $A=A_\ell^\ast$ and $B=\bar A^\ast$ to get
\[
\sum_{\ell=1}^m \Big[(G_\ell)_{b,a}-(\bar G)_{b,a}\Big]^2
\;\le\; L_T^2(\rho_{\max}) \sum_{\ell=1}^m \|A_\ell^\ast-\bar A^\ast\|_F^2.
\]

\smallskip
\emph{Step 3 (graph Poincaré, matrix‑valued signal).}
For any matrix‑valued signal on the nodes, the graph Poincaré inequality yields
\[
\sum_{\ell=1}^m \|A_\ell^\ast - \bar A^\ast\|_F^2
\;\le\; \frac{1}{\lambda_{m-1}(G)}\sum_{\{\ell,\ell'\}\in \calE}\|A_\ell^\ast-A_{\ell'}^\ast\|_F^2.
\]
Combining Step 2 with Step 3 and taking the maximum over $(a,b)$ we obtain the stated bound.
\end{proof}

%

\newpage
\bibliographystyle{plain}
\bibliography{references}

\end{document}